\documentclass[11pt,fleqn]{article}

\usepackage{graphicx}
\usepackage{amsthm,amsmath}
\usepackage{amssymb,dsfont,mathrsfs}

\numberwithin{equation}{section}
\newtheorem{theorem}{Theorem}[section]
\newtheorem{lemma}{Lemma}[section]
\newtheorem{proposition}{Proposition}[section]
\newtheorem{corollary}{Corollary}[section]
\newtheorem{remark}{Remark}[section]
\newtheorem{assumption}{Assumption}[section]
\newtheorem{definition}{Definition}[section]

\def\ba{\boldsymbol{a}}
\def\bb{\boldsymbol{b}}
\def\bc{\boldsymbol{c}}

\def\bg{\boldsymbol{g}}
\def\bi{\boldsymbol{i}}

\def\bk{\boldsymbol{k}}
\def\bl{\boldsymbol{l}}

\def\bu{\boldsymbol{u}}
\def\bv{\boldsymbol{v}}

\def\bY{\boldsymbol{Y}}

\def\bphi{\boldsymbol{\phi}}
\def\bvarphi{\boldsymbol{\varphi}}
\def\bpi{\boldsymbol{\pi}}
\def\bnu{\boldsymbol{\nu}}

\def\bzero{\mathbf{0}}
\def\bone{\mathbf{1}}

\def\calD{\mathcal{D}}
\def\calE{\mathcal{E}}

\def\calS{\mathcal{S}}

\def\calL{\mathcal{L}}

\def\scrI{\mathscr{I}}

\def\spr{\mbox{\rm spr}}
\def\diag{\mbox{\rm diag}}
\def\adj{\mbox{\rm adj}}

\def\Ker{\mbox{\rm Ker}}
\def\vecM{\mbox{\rm vec}}

\title{Exact asymptotic formulae of the stationary distribution of a discrete-time two-dimensional QBD process}
\author{Toshihisa Ozawa${}^{\dagger}$ and Masahiro Kobayashi${}^{\dagger\dagger}$ \\ 
${}^{\dagger}$Faculty of Business Administration, Komazawa University \\
${}^{\dagger\dagger}$Department of Mathematical Science, Tokai University \\
${}^{\dagger}$1-23-1 Komazawa, Setagaya-ku, Tokyo 154-8525, Japan \\
E-mail: ${}^{\dagger}$toshi@komazawa-u.ac.jp, ${}^{\dagger\dagger}$m\_kobayashi@tsc.u-tokai.a.c.jp 
}

\begin{document}

\maketitle

\begin{abstract}
We consider a discrete-time two-dimensional process $\{(L_{1,n},L_{2,n})\}$ on $\mathbb{Z}_+^2$ with a supplemental process $\{J_n\}$ on a finite set, where individual processes $\{L_{1,n}\}$ and $\{L_{2,n}\}$ are both skip free. We assume that the joint process $\{Y_n\}=\{(L_{1,n},L_{2,n},J_n)\}$ is Markovian and that the transition probabilities of the two-dimensional process $\{(L_{1,n},L_{2,n})\}$ are modulated depending on the state of the background process $\{J_n\}$. This modulation is space homogeneous except for the boundaries of $\mathbb{Z}_+^2$. 
We call this process a discrete-time two-dimensional quasi-birth-and-death (2D-QBD) process and obtain the exact asymptotic formulae of the stationary distribution in the coordinate directions. 

\smallskip
{\it Key wards}: quasi-birth-and-death process, stationary distribution, asymptotic property, matrix analytic method, two-dimensional reflecting random walk

\end{abstract}

%
%
\section{Introduction} \label{sec:intro}

We deal with a discrete-time two-dimensional process $\{(L_{1,n},L_{2,n})\}$ on $\mathbb{Z}_+^2$ with a supplemental process $\{J_n\}$ on a finite set $S_0$. We call the supplemental process a phase process. Individual processes $\{L_{1,n}\}$ and $\{L_{2,n}\}$ are skip free, which means that their increments take values in $\{-1,0,1\}$. 
We assume that the joint process $\{\bY_n\}=\{(L_{1,n},L_{2,n},J_n)\}$ is Markovian and that the transition probabilities of the two-dimensional process $\{(L_{1,n},L_{2,n})\}$ are modulated depending on the state of the phase process $\{J_n\}$. This modulation is space homogeneous except for the boundaries of $\mathbb{Z}_+^2$. 
Since a one-dimensional version of this model is called a discrete-time quasi-birth-and-death (QBD) process (see, for example, Latouche and Ramaswami \cite{Latouche99}), we call it a discrete-time two-dimensional quasi-birth-and-death (2D-QBD) process \cite{Ozawa13}.  
Note that a 2D-QBD process is also a two-dimensional skip-free Markov modulated reflecting random walk (MMRRW) \cite{Ozawa15} and stochastic models arising from various Markovian queueing networks with two nodes such as generalized two-node Jakson networks with Markovian arrival processes and phase-type service processes can be represented as 2D-QBD processes (see, for example, Ozawa \cite{Ozawa13} and \cite{Ozawa15}).

Assume that the 2D-QBD process $\{\bY_n\}$ is irreducible, aperiodic and positive recurrent, and denote by $\bnu=(\bnu_{k,l},\,(k,l)\in\mathbb{Z}^2)$ its stationary distribution, where $\bnu_{k,l}=(\nu_{k,l,j},\,j\in S_0)$ and $\nu_{k,l,j}=\lim_{n\to\infty} \mathbb{P}(\bY_n=(k,l,j))$.  
Our aim is to reveal asymptotic properties of the stationary distribution $\bnu$, especially, to obtain the directional exact asymptotic formulae $h_1(k)$ and $h_2(k)$ that satisfy, for some nonzero vector $\bc_1$ and $\bc_2$, 
\begin{equation}
\lim_{k\to\infty} \frac{\bnu_{k,0}}{h_1(k)}=\bc_1,\qquad 
\lim_{k\to\infty} \frac{\bnu_{0,k}}{h_2(k)}=\bc_2.
\end{equation}
A 2D-QBD process {\it without a phase process} is a two-dimensional skip-free reflecting random walk (RRW), which is called a double QBD process in Miyazawa \cite{Miyazawa09}, and the directional exact asymptotic formulae of the stationary distribution of the double QBD process, denoted by $h'_1(k)$ and $h'_2(k)$, are obtained in Kobayashi and Miyazawa \cite{Kobayashi13}. For $i\in\{1,2\}$, $h'_i(k)$ is given in a form
\begin{equation}
h'_i(k) = k^{\alpha'_i-1} (r'_i)^{-k}, 
\end{equation}
where $\alpha'_i\in\{-\frac{1}{2},\frac{1}{2},1,2\}$ and $r'_i$ is the geometric decay rate satisfying $r'_i>1$. 
We will demonstrate that, under certain conditions, the same results also hold for the 2D-QBD process, i.e., the directional exact asymptotic formulae $h_i(k),\,i=1,2,$ are given in a form
\begin{equation}
h_i(k) = k^{\alpha_i-1} r_i^{-k}, 
\end{equation}
where $\alpha_i\in\{-\frac{1}{2},\frac{1}{2},1,2\}$ and $r_i>1$. 

While there are several possible approaches for getting asymptotic properties of the stationary distribution (see Miyazawa \cite{Miyazawa11}), we shall adopt an analytic function approach using the convergence domain, which has been used in Kobayashi and Miyazawa \cite{Kobayashi13}. 
Let $\bvarphi_1(z)$ be the generating function defined as $\bvarphi_1(z)=\sum_{k=0}^\infty \bnu_{k,0} z^k$, where $z$ is a complex variable, and let $\bvarphi_2(z)$ be defined analogously.  By Cauchy's criterion, we see that the radius of convergence of $\bvarphi_1(z)$ and that of $\bvarphi_2(z)$ are respectively given by the directional geometric decay rates $r_1$ and $r_2$ of the stationary distribution $\bnu$, which have already been obtained in Ozawa \cite{Ozawa13}. 
Therefore, in order to obtain the exact asymptotic formula $h_1(k)$ (resp.\ $h_2(k)$), it suffices to analytically extend $\bvarphi_1(z)$ (resp.\ $\bvarphi_2(z)$) beyond its convergence domain and clarify the singularities of $\bvarphi_1(z)$ (resp.\ $\bvarphi_2(z)$) on the circle $|z|=r_1$ (resp.\ $|z|=r_2$). 
To this end, we use the following key expression (see Lemma \ref{le:varphi1}): 
\begin{align*}
\bvarphi_1(z) 
&= \biggl\{ \sum_{j=1}^\infty \bnu_{0,j} \Big( A_{*,-1}^{(2)}(z) + A_{*,0}^{(2)}(z) G_1(z) + A_{*,1}^{(2)}(z) G_1(z)^2 \Big) G_1(z)^{j-1} - \sum_{j=1}^\infty \bnu_{0,j} G_1(z)^j \cr 
&\qquad\quad + \bnu_{0,0} \Big(A_{*,0}^{(0)}(z)+A_{*,1}^{(0)}(z) G_1(z)-I \Big) \biggr\} \Bigl( I-A_{*,0}^{(1)}(z)-A_{*,1}^{(1)}(z) G_1(z)) \Bigr)^{-1}, 
\end{align*}
where each $A_{*,k}^{(l)}(z)$ is a matrix whose entries are Laurent polynomials in $z$; $G_1(z)$ is a solution to the following matrix quadratic equation of $X$: 
\[
A_{*,-1}(z) + (A_{*,0}(z)-I) X + A_{*,1}(z) X^2 = O, 
\]
where each $A_{*,k}(z)$ is also a matrix whose entries are Laurent polynomials in $z$. 
A similar expression also holds for $\bvarphi_2(z)$. 
The key expression of $\bvarphi_1(z)$ corresponds to equation (29) in Kobayashi and Miyazawa \cite{Kobayashi13}, and $G_1(z)$ corresponds to a so-called G-matrix of QBD process (see, for example, Latouche and Ramaswami \cite{Latouche99}). We first define this matrix function $G_1(z)$ on an annular domain on the complex plane, then analytically extend it beyond the annular domain. Using the key expression above and the extended $G_1(z)$, we clarify the singularities of $\bvarphi_1(z)$ on the circle $|z|=r_1$. The singularities of $\bvarphi_2(z)$ on the circle $|z|=r_2$ can also be clarified in the same way. 
Note that, under the assumptions that we state in Section \ref{sec:model}, $\bvarphi_1(z)$ (resp.\ $\bvarphi_2(z)$) has just one singularity $z=r_1$ (resp.\ $z=r_2$) on the circle $|z|=r_1$ (resp.\ $|z|=r_2$), and the singularity is a pole and/or branch point. 

The rest of the paper is organized as follows.  
In Section \ref{sec:model}, the 2D-QBD process we consider is described in detail and our main results are stated. 
In Section \ref{sec:generating_function}, we consider the generating function of the stationary distribution and derive the key expressions mentioned above. 
In Section \ref{sec:Gmatrix}, we redefine the matrix function $G_1(z)$ on an annular domain and analytically extend it. 
The asymptotic formulae $h_1(k)$ and $h_2(k)$ are obtained in Section \ref{sec:asymptotics}.

%
%
\section{Preliminary and main results} \label{sec:model} 

Before describing the model, we present several notations used in the paper. 
$\mathbb{R}$ is the set of all real numbers and $\mathbb{R}_+$ that of all nonnegative real numbers. $\mathbb{Z}$ is the set of all integers and $\mathbb{Z}_+$ that of all nonnegative integers. $\mathbb{N}$ is the set of all positive integers. 
A set $\mathbb{H}$ is defined as $\mathbb{H}=\{-1,0,1\}$ and $\mathbb{H}_+$ as $\mathbb{H}_+=\{0,1\}$. 
$\mathbb{C}$ is the set of all complex numbers. For $a,b\in\mathbb{R}_+$, $\mathbb{C}[a,b]$ and $\mathbb{C}[a,b)$ are defined as $\mathbb{C}[a,b] = \{z\in\mathbb{C}: a\le |z|\le b \}$ and $\mathbb{C}[a,b) = \{z\in\mathbb{C}: a\le |z|< b \}$, respectively. $\mathbb{C}(a,b]$ and $\mathbb{C}(a,b)$ are analogously defined. 
For $z\in\mathbb{C}$ and $r\in\mathbb{R}_+$, $\Delta(z,r)$ is the open disc of center $z$ and radius $r$. $\bar{\Delta}(z,r)$ and $\partial \Delta(z,r)$ are the closed disc and circle of the same center and radius, respectively. 
We denote by $\Delta_r$ the open disc with center $0$ and radius $r$, i.e., $\Delta_r=\Delta(0,r)$. 
For a matrix $A=(a_{ij})$, we denote by $[A]_{i,j}$ the $(i,j)$-entry of $A$. The transpose of a matrix $A$ is denoted by $A^\top$. 
We denote by $\spr(A)$ the spectral radius of $A$, which is the maximum modulus of eigenvalue of $A$. 
We denote by $|A|$ the matrix each of whose entries is the modulus of the corresponding entry of $A$, i.e., $|A|=(|a_{ij}|)$. 
$O$ is a matrix of $0$'s, $\bone$ is a column vector of $1$'s and $\bzero$ is a column vector of $0$'s; their dimensions are determined in context. $I$ is the identity matrix. 
For a $k\times l$ matrix $A=\begin{pmatrix} \ba_1 & \ba_2 & \cdots & \ba_l \end{pmatrix}$, ${\rm vec}(A)$ is a $kl\times 1$ vector defined as 
\[
\vecM(A) = \begin{pmatrix} \ba_1 \cr \ba_2 \cr \vdots \cr \ba_k \end{pmatrix}. 
\]
For matrices $A$, $B$ and $C$, the identity $\vecM(ABC)=(C^\top\otimes A)\,\vecM(B)$ holds (see, for example, 
Horn and Johnson \cite{Horn91}).

%
%
\subsection{Two-dimensional quasi-birth-and-death process}

Consider a two-dimensional process $\{(X_{1,n}, X_{2,n})\}$ on $\mathbb{Z}_+^2$ and a background process $\{J_n\}$ on a finite state space $S_0$, where we denote by $s_0$ the number of elements of $S_0$, i.e., $S_0=\{1,2,...,s_0\}$. We assume that both $\{X_{1,n}\}$ and $\{X_{2,n}\}$ are skip free and that the joint process $\{\bY_n\} = \{ (X_{1,n}, X_{2,n}, J_n) \}$ is Markovian. 
To be precise, $\{\bY_n\}$ is a discrete-time Markov chain on the state space $\calS = \mathbb{Z}_+^2 \times S_0$ and the transition probability matrix 
\[
P=\begin{pmatrix} p_{(x_1,x_2,j),(x_1',x_2',j')},(x_1,x_2,j),(x_1',x_2',j')\in\mathbb{Z}_+^2\times S_0 \end{pmatrix},
\]
where $p_{(x_1,x_2,j),(x_1',x_2',j')} = \mathbb{P}(\bY_1=(x_1',x_2',j')\,|\,\bY_0=(x_1,x_2,j))$, is given in terms of $s_0\times s_0$ non-negative block matrices 
\[
A_{i,j}, i,j\in\mathbb{H},\quad
A^{(1)}_{i,j}, i\in\mathbb{H},j\in\mathbb{H}_+,\quad
A^{(2)}_{i,j}, i\in\mathbb{H}_+,j\in\mathbb{H},\quad
A^{(0)}_{i,j}, i,j\in\mathbb{H}_+,
\]
as follows: for $(x_1,x_2),(x_1',x_2')\in\mathbb{Z}_+^2$, 
\[
\begin{pmatrix} p_{(x_1,x_2,j),(x_1',x_2',j')}, j,j'\in S_0 \end{pmatrix} 
= \left\{ \begin{array}{ll} 
A_{\varDelta x_1,\varDelta x_2}, & \mbox{if  $x_1\ne 0$, $x_2\ne 0$, $\varDelta x_1,\varDelta x_2\in\mathbb{H}$}, \cr 
A^{(1)}_{\varDelta x_1,\varDelta x_2}, & \mbox{if $x_1\ne 0$, $x_2=0$, $\varDelta x_1\in\mathbb{H}$, $\varDelta x_2\in\mathbb{H}_+$}, \cr
A^{(2)}_{\varDelta x_1,\varDelta x_2}, & \mbox{if $x_1=0$, $x_2\ne 0$, $\varDelta x_1\in\mathbb{H}_+$, $\varDelta x_2\in\mathbb{H}$}, \cr
A^{(0)}_{\varDelta x_1,\varDelta x_2}, & \mbox{if $x_1=x_2=0$, $\varDelta x_1,\varDelta x_2\in\mathbb{H}_+$}, \cr
O, & \mbox{otherwise}, \end{array} \right. 
\]
where $\varDelta x_1=x_1'-x_1$ and $\varDelta x_2=x_2'-x_2$. 
The matrices $A_{*,*}$, $A^{(1)}_{*,*}$, $A^{(2)}_{*,*}$ and $A^{(0)}_{*,*}$ defined as
\[
A_{*,*} = \sum_{i,j\in\mathbb{H}} A_{i,j},\quad
A^{(1)}_{*,*} = \sum_{i\in\mathbb{H},j\in\mathbb{H}_+} A^{(1)}_{i,j},\quad
A^{(2)}_{*,*} = \sum_{i\in\mathbb{H}_+,j\in\mathbb{H}} A^{(2)}_{i,j},\quad
A^{(0)}_{*,*} = \sum_{i,j\in\mathbb{H}_+} A^{(0)}_{i,j},
\]
are stochastic. 
The Markov chain $\{\bY_n\}$ is called a discrete-time two-dimensional quasi-birth-and-death (2D-QBD) process in Ozawa \cite{Ozawa13}, where $(X_{1,n},X_{2,n})$ is called the level and $J_n$ the phase. 

\begin{remark}
The 2D-QBD process explained above is a simplified model of that introduced in Ozawa \cite{Ozawa13}. In order to make the structure of the vector generating function defined later simple, we adopt the simplified model. 
A general 2D-QBD process can be reduced to a simplified one with keeping the stationary distribution unchanged: see Miyazawa \cite{Miyazawa15}. 
\end{remark}

We assume the following condition throughout the paper. 
\begin{assumption} \label{as:irreducibleYn}
The Markov chain $\{\bY_n\}$ is irreducible and aperiodic.
\end{assumption}

We consider three kinds of Markov chain generated from $\{\bY_n\}$ by removing one or two boundaries and denote them by $\{\tilde{\bY}_n\}=\{(\tilde{X}_{1,n},\tilde{X}_{2,n},\tilde{J}_n)\}$, $\{\tilde{\bY}^{(1)}_n\}=\{(\tilde{X}^{(1)}_{1,n},\tilde{X}^{(1)}_{2,n},\tilde{J}^{(1)}_n)\}$ and $\{\tilde{\bY}^{(2)}_n\}=\{(\tilde{X}^{(2)}_{1,n},\tilde{X}^{(2)}_{2,n},\tilde{J}^{(2)}_n)\}$, respectively.  
$\{\tilde{\bY}_n\}$ is the Markov chain on the state space $\mathbb{Z}^2\times S_0$, generated by removing the boundaries on the $x_1$ and $x_2$-axes, and the transition probability matrix 
\[
\tilde{P}=\begin{pmatrix} \tilde{p}_{(x_1,x_2,j),(x_1',x_2',j')},(x_1,x_2,j),(x_1',x_2',j')\in\mathbb{Z}^2\times S_0 \end{pmatrix},
\]
where $\tilde{p}_{(x_1,x_2,j),(x_1',x_2',j')} = \mathbb{P}(\tilde{\bY}_1=(x_1',x_2',j')\,|\,\tilde{\bY}_0=(x_1,x_2,j))$ is given as 
\[
\begin{pmatrix} \tilde{p}_{(x_1,x_2,j),(x_1',x_2',j')}, j,j'\in S_0 \end{pmatrix} 
= \left\{ \begin{array}{ll} A_{\varDelta x_1,\varDelta x_2}, & \mbox{if $\varDelta x_1,\varDelta x_2\in\mathbb{H}$}, \cr 
O, & \mbox{otherwise}. \end{array} \right. 
\]
where $\varDelta x_1=x_1'-x_1$ and $\varDelta x_2=x_2'-x_2$. 
The Markov chain $\{\tilde{\bY}_n\}$ is a two-dimensional skip-free random walk on $\mathbb{Z}^2$ whose transition probabilities are modulated depending on the phase state $J_n$. From the definition of $\{\tilde{\bY}_n\}$, we see that it is governed only by the block matrices $A_{k,l},\,k,l\in\mathbb{H}$. Hence, we call the Markov chain $\{\tilde{\bY}_n\}$ a Markov chain generated by $\{A_{k,l},\,k,l\in\mathbb{H}\}$. We adopt the following definition. 
\begin{definition} \label{def:inducedMC_Akl}
We say that the set of block matrices, $\{A_{k,l},\,k,l\in\mathbb{H}\}$, is irreducible (resp.\ aperiodic) if the Markov chain generated by the set of block matrices is irreducible (resp.\ aperiodic).  
\end{definition}

$\{\tilde{\bY}^{(1)}_n\}$ is the Markov chain on the state space $\mathbb{Z}\times\mathbb{Z}_+\times S_0$, generated by removing the boundary on the $x_2$-axes, and the transition probability matrix 
\[
\tilde{P}^{(1)} = \begin{pmatrix} \tilde{p}^{(1)}_{(x_1,x_2,j),(x_1',x_2',j')},(x_1,x_2,j),(x_1',x_2',j')\in\mathbb{Z}\times\mathbb{Z}_+\times S_0 \end{pmatrix},
\]
where $\tilde{p}^{(1)}_{(x_1,x_2,j),(x_1',x_2',j')} = \mathbb{P}(\tilde{\bY}^{(1)}_1=(x_1',x_2',j')\,|\,\tilde{\bY}^{(1)}_0=(x_1,x_2,j))$ is given as 
\[
\begin{pmatrix} \tilde{p}^{(1)}_{(x_1,x_2,j),(x_1',x_2',j')}, j,j'\in S_0 \end{pmatrix} 
= \left\{ \begin{array}{ll} 
A_{\varDelta x_1,\varDelta x_2}, & \mbox{if  $x_2\ne 0$, $\varDelta x_1,\varDelta x_2\in\mathbb{H}$}, \cr 
A^{(1)}_{\varDelta x_1,\varDelta x_2}, & \mbox{if $x_2=0$, $\varDelta x_1\in\mathbb{H}$, $\varDelta x_2\in\mathbb{H}_+$}, \cr
O, & \mbox{otherwise}, 
\end{array} \right. 
\]
where $\varDelta x_1=x_1'-x_1$ and $\varDelta x_2=x_2'-x_2$. 
$\{\tilde{\bY}^{(2)}_n\}$ is the Markov chain on the state space $\mathbb{Z}_+\times\mathbb{Z}\times S_0$, generated by removing the boundary on the $x_1$-axes, and the transition probability matrix $\tilde{P}^{(2)}$ is analogously given. 
$\{\tilde{\bY}^{(1)}_n\}$ is the Markov chain generated by $\{ \{A_{k,l},\,k,l\in\mathbb{H}\},\,\{A_{k,l}^{(1)},\,k\in\mathbb{H},\,j\in\mathbb{H}_+\} \}$, and $\{\tilde{\bY}^{(2)}_n\}$ is that generated by $\{ \{A_{k,l},\,k,l\in\mathbb{H}\},\,\{A_{k,l}^{(2)},\,k\in\mathbb{H}_+,\,j\in\mathbb{H}\} \}$. 
We also adopt the following definition. 
\begin{definition} \label{def:inducedMC_Akl12}
We say that the set of block matrices, $\{ \{A_{k,l},\,k,l\in\mathbb{H}\},\,\{A_{k,l}^{(1)},\,k\in\mathbb{H},\,j\in\mathbb{H}_+\} \}$, is irreducible (resp.\ aperiodic) if the Markov chain generated by the set of block matrices is irreducible (resp.\ aperiodic). Irreducibleness and aperiodicity of $\{ \{A_{k,l},\,k,l\in\mathbb{H}\},\,\{A_{k,l}^{(2)},\,k\in\mathbb{H}_+,\,j\in\mathbb{H}\} \}$ are analogously defined. 
\end{definition}

Hereafter, we assume the following condition throughout the paper.
\begin{assumption} \label{as:Akl_irreducible} 
The sets of block matrices, $\{A_{k,l},\,k,l\in\mathbb{H}\}$, $\{ \{A_{k,l},\,k,l\in\mathbb{H}\},\,\{A_{k,l}^{(1)},\,k\in\mathbb{H},\,j\in\mathbb{H}_+\} \}$ and $\{ \{A_{k,l},\,k,l\in\mathbb{H}\},\,\{A_{k,l}^{(2)},\,k\in\mathbb{H}_+,\,j\in\mathbb{H}\} \}$ are irreducible and aperiodic. 
\end{assumption}

Since $A_{*,*}$ is the transition probability matrix of the background process $\{\tilde{J}_n\}$ of $\{\tilde{\bY}_n\}$, we immediately obtain the following proposition.
\begin{proposition} \label{pr:Ass_irreducible}
Under Assumption \ref{as:Akl_irreducible}, $A_{*,*}$ is irreducible and aperiodic.
\end{proposition}
$A_{*,*}$ is, therefore, positive recurrent and ergodic since its dimension is finite. We denote by $\bpi_{*,*}$ the stationary distribution of $A_{*,*}$. 
%

%
\subsection{Stationary condition}

According to Ozawa \cite{Ozawa15}, we state the condition on which the 2d-QBD process is positive recurrent. Before doing it, we define so-called induced Markov chains and the mean drift vectors derived from the induced Markov chains (see Fayolle et al.\ \cite{Fayolle95}).
Since the 2D-QBD process is a kind of two-dimensional random walk, there exist three induced Markov chains: $\calL^{\{1,2\}}$, $\calL^{\{1\}}$ and $\calL^{\{2\}}$. 
$\calL^{\{1,2\}}$ is the phase process $\{\tilde{J}_n\}$ of $\{\tilde{\bY}_n\}$ and it is a Markov chain governed by the transition probability matrix $A_{*,*}$. The mean drift vector $\ba^{\{1,2\}}=(a^{\{1,2\}}_1,a^{\{1,2\}}_2)$ derived from $\calL^{\{1,2\}}$ is given as 
\[
a^{\{1,2\}}_1 = \bpi_{*,*} (-A_{-1,*}+A_{1,*}) \bone,\quad 
a^{\{1,2\}}_2 = \bpi_{*,*} (-A_{*,-1}+A_{*,1}) \bone,  
\]
where for $k\in\mathbb{H}$, $A_{k,*}=\sum_{l\in\mathbb{H}}A_{k,l}$ and $A_{*,k}=\sum_{l\in\mathbb{H}}A_{l,k}$. 
Define block tri-diagonal transition probability matrices $A^{(1)}_*$ and $A^{(2)}_*$ as 
\[
A^{(1)}_* = 
\begin{pmatrix}
A^{(1)}_{*,0} & A^{(1)}_{*,1} & & & \cr
A_{*,-1} & A_{*,0} & A_{*,1} & & \cr
& A_{*,-1} & A_{*,0} & A_{*,1} & \cr
& & \ddots & \ddots & \ddots 
\end{pmatrix},\ 
A^{(2)}_* = 
\begin{pmatrix}
A^{(2)}_{0,*} & A^{(2)}_{1,*} & & & \cr
A_{-1,*} & A_{0,*} & A_{1,*} & & \cr
& A_{-1,*} & A_{0,*} & A_{1,*} & \cr
& & \ddots & \ddots & \ddots 
\end{pmatrix}, 
\]
where for $k\in\mathbb{H}$, $A^{(1)}_{*,k}=\sum_{l\in\mathbb{H}}A^{(1)}_{l,k}$ and $A^{(2)}_{k,*}=\sum_{l\in\mathbb{H}}A^{(2)}_{k,l}$. We immediately obtain the following proposition.
\begin{proposition} \label{pr:A1sA2s_irreducible}
Under Assumption \ref{as:Akl_irreducible}, $A^{(1)}_*$ and $A^{(2)}_*$ are irreducible and aperiodic.
\end{proposition}
$\calL^{\{1\}}$ (resp.\ $\calL^{\{2\}}$) is a partial process $\{(\tilde{X}^{(1)}_{2,n},\tilde{J}^{(1)}_n)\}$ of $\{\tilde{\bY}^{(1)}_n\}$ (resp.\ $\{(\tilde{X}^{(2)}_{1,n},\tilde{J}^{(2)}_n)\}$ of $\{\tilde{\bY}^{(2)}_n\}$) and it is a Markov chain governed by $A^{(1)}_*$ (resp.\ $A^{(2)}_*$). 
Since $\calL^{\{1\}}$ (resp.\ $\calL^{\{2\}}$) is one-dimensional QBD process, it is positive recurrent if and only if $a^{\{1,2\}}_2<0$ (resp.\ $a^{\{1,2\}}_1<0$). 
Denote by $\bpi^{(1)}_*=(\bpi^{(1)}_{*,k},k\in\mathbb{Z}_+)$ and $\bpi^{(2)}_*=(\bpi^{(2)}_{*,k},k\in\mathbb{Z}_+)$ the stationary distributions of $\calL^{\{1\}}$ and $\calL^{\{2\}}$, respectively, if they exist. Then, the mean increment vectors $\ba^{\{1\}}=(a^{\{1\}}_1,a^{\{1\}}_2)$ and $\ba^{\{2\}}=(a^{\{2\}}_1,a^{\{2\}}_2)$ derived from $\calL^{\{1\}}$ and $\calL^{\{2\}}$ are given by 
\begin{align*}
a^{\{1\}}_1 &= \bpi_{*,0}^{(1)}\big(-A_{-1,0}^{(1)}-A_{-1,1}^{(1)}+A_{1,0}^{(1)}+A_{1,1}^{(1)} \big)\,\bone + \bpi_{*,1}^{(1)}\big(I-R^{(1)}_*\big)^{-1} \big(-A_{-1,*}+A_{1,*} \big)\,\bone,\\
a^{\{1\}}_2 &= 0,\quad 
a^{\{2\}}_1 = 0, \\
a^{\{2\}}_2 &= \bpi_{*,0}^{(2)}\big(-A_{0,-1}^{(2)}-A_{1,-1}^{(2)}+A_{0,1}^{(2)}+A_{1,1}^{(2)} \big)\,\bone + \bpi_{*,1}^{(2)}\big(I-R^{(2)}_*\big)^{-1} \big(-A_{*,-1}+A_{*,1} \big)\,\bone, 
\end{align*}
where $R^{(1)}_*$ and $R^{(2)}_*$ are the rate matrix of $\calL^{\{1\}}$ and $\calL^{\{2\}}$, respectively. 
The condition on which $\{\bY_n\}$ is positive recurrent or transient is given by the following lemma. 
%
\begin{lemma}[Theorem 3.4 of Ozawa \cite{Ozawa15}] \label{le:classification_2DMMRRW}
When $a^{\{1,2\}}_1<0$ and $a^{\{1,2\}}_2<0$, the 2D-QBD process $\{\bY_n\}$ is positive recurrent if $a^{\{1\}}_1<0$ and $a^{\{2\}}_2<0$, and it is transient if either  $a^{\{1\}}_1>0$ or $a^{\{2\}}_2>0$. 
When $a^{\{1,2\}}_1>0$ and $a^{\{1,2\}}_2<0$, the 2D-QBD process is positive recurrent if $a^{\{1\}}_1<0$, and it is transient if $a^{\{1\}}_1>0$.
When $a^{\{1,2\}}_1<0$ and $a^{\{1,2\}}_2>0$, the 2D-QBD process is positive recurrent if $a^{\{2\}}_2<0$, and it is transient if $a^{\{2\}}_2>0$.
When $a^{\{1,2\}}_1>0$ and $a^{\{1,2\}}_2>0$, the 2D-QBD process is transient.
\end{lemma}

In order for the 2D-QBD $\{\bY_n\}$ to be positive recurrent, we assume the following condition throughout the paper. 

\begin{assumption} \label{as:stability_cond}
If $a^{\{1,2\}}_1<0$ and $a^{\{1,2\}}_2<0$, then $a^{\{1\}}_1<0$ and $a^{\{2\}}_2<0$; if $a^{\{1,2\}}_1>0$ and $a^{\{1,2\}}_2<0$, then $a^{\{1\}}_1<0$; if $a^{\{1,2\}}_1<0$ and $a^{\{1,2\}}_2>0$, then $a^{\{2\}}_2<0$. 
\end{assumption}

Denote by $\bnu = \left(\nu_{k,l,j},\, (k,l,j)\in\mathbb{Z}_+^2\times S_0 \right)$ the stationary distribution of $\{\bY_n\}$. We also define $\bnu_{k,l}$ for $k,l\in\mathbb{Z}_+$ as $\bnu_{k,l} = \left(\nu_{k,l,j},\, j\in S_0 \right)$. In terms of $\bnu_{k,l}$, $\bnu$ is represented as $\bnu = \left( \bnu_{k,l},\, (k,l)\in\mathbb{Z}_+^2 \right)$.

%
\subsection{QBD process with a countable phase space}

A QBD process with a countable phase space is a skip-free one-dimensional process on $\mathbb{Z}_+$ with a background process on a countable state space. The one-dimensional process is the {\it level} and the background process the {\it phase}. 
A 2D-QBD process $\{\bY_n\} = \{ (X_{1,n}, X_{2,n}, J_n) \}$ can be represented as a QBD process with a countable phase space in two ways: One is $\{\bY_n^{(1)}\} = \{(X_{1,n}, (X_{2,n}, J_n))\}$, where $X_{1,n}$ is the level and $(X_{2,n}, J_n)$ the phase, and the other $\{\bY_n^{(2)}\} = \{(X_{2,n}, (X_{1,n}, J_n))\}$, where $X_{2,n}$ is the level and $(X_{1,n}, J_n)$ the phase.
Let $P^{(1)}$ and $P^{(2)}$ be the transition probability matrices of $\{Y_n^{(1)}\}$ and $\{Y_n^{(2)}\}$, respectively. They are represented in block form as follows:
\begin{equation} \label{eq:blockformP}
P^{(i)} = 
\begin{pmatrix}
B^{(i)}_0 & B^{(i)}_1 & & & \cr
A^{(i)}_{-1} & A^{(i)}_0 & A^{(i)}_1 & & \cr
& A^{(i)}_{-1} & A^{(i)}_0 & A^{(i)}_1 & \cr
& & \ddots & \ddots & \ddots 
\end{pmatrix},\, i=1,2, 
\end{equation}
where each block is given by, for $j=0,1$,
\[
B^{(1)}_j = 
\begin{pmatrix}
A^{(0)}_{j,0} & A^{(0)}_{j,1} & & & \cr
A^{(2)}_{j,-1} & A^{(2)}_{j,0} & A^{(2)}_{j,1} & & \cr
& A^{(2)}_{j,-1} & A^{(2)}_{j,0} & A^{(2)}_{j,1} & \cr
& & \ddots & \ddots & \ddots 
\end{pmatrix},\ 
B^{(2)}_j = 
\begin{pmatrix}
A^{(0)}_{0,j} & A^{(0)}_{1,j} & & & \cr
A^{(1)}_{-1,j} & A^{(1)}_{0,j} & A^{(1)}_{1,j} & & \cr
& A^{(1)}_{-1,j} & A^{(1)}_{0,j} & A^{(1)}_{1,j} & \cr
& & \ddots & \ddots & \ddots 
\end{pmatrix}, 
\]
and for $j=-1,0,1$, 
\[
A^{(1)}_j = 
\begin{pmatrix}
A^{(1)}_{j,0} & A^{(1)}_{j,1} & & & \cr
A_{j,-1} & A_{j,0} & A_{j,1} & & \cr
& A_{j,-1} & A_{j,0} & A_{j,1} & \cr
& & \ddots & \ddots & \ddots 
\end{pmatrix},\ 
A^{(2)}_j = 
\begin{pmatrix}
A^{(2)}_{0,j} & A^{(2)}_{1,j} & & & \cr
A_{-1,j} & A_{0,j} & A_{1,j} & & \cr
& A_{-1,j} & A_{0,j} & A_{1,j} & \cr
& & \ddots & \ddots & \ddots 
\end{pmatrix}. 
\]
The transition probability matrices $A_*^{(1)}$ and $A_*^{(2)}$ defined in the previous subsection are represented as $A_*^{(1)} = \sum_{k\in\mathbb{H}} A_k^{(1)}$ and $A_*^{(2)} = \sum_{k\in\mathbb{H}} A_k^{(2)}$, respectively. Hence, $A_*^{(1)}$ (resp.\ $A_*^{(2)}$) is also the transition probability matrix of the phase process of $\{\bY_n^{(1)}\}$ (resp.\ $\{\bY_n^{(2)}\}$) when the level is greater than zero. 

Denote by $\bnu^{(1)} = (\bnu^{(1)}_n, n\in\mathbb{Z}_+)$ the stationary distribution of $\{\bY^{(1)}\}$ and by $\bnu^{(2)} = (\bnu^{(2)}_n, n\in\mathbb{Z}_+)$ that of $\{\bY^{(2)}\}$, where $\bnu^{(1)}_n = (\bnu_{n,k}, k\in\mathbb{Z}_+)$ and $\bnu^{(2)}_n = (\bnu_{k,n}, k\in\mathbb{Z}_+)$. 
It is well known that each of $\bnu^{(1)}$ and $\bnu^{(2)}$ has a matrix geometric form (see References \cite{Neuts94,Tweedie82}, especially, Tweedie \cite{Tweedie82} for the case of QBD process with a countable phase space). 
Let $R^{(1)}$ and $R^{(2)}$ be the rate matrices of $\{\bY^{(1)}\}$ and $\{\bY^{(2)}\}$, respectively. They are given by the minimal nonnegative solutions to the matrix quadratic equations:
\begin{equation}
R^{(1)} = (R^{(1)})^2 A^{(1)}_{-1} + R^{(1)} A^{(1)}_{0} + A^{(1)}_{1},\quad 
R^{(2)} = (R^{(2)})^2 A^{(2)}_{-1} + R^{(2)} A^{(2)}_{0} + A^{(2)}_{1}.
\label{eq:rmRto1andRt02}
\end{equation}
The matrix geometric forms of $\bnu^{(1)}$ and $\bnu^{(2)}$ are given by 
\begin{equation}
\bnu^{(1)}_n = \bnu^{(1)}_1 (R^{(1)})^{n-1},\quad
\bnu^{(2)}_n = \bnu^{(2)}_1 (R^{(2)})^{n-1},\ n\ge 1. 
\label{eq:mgfnu1andnu2}
\end{equation}
Since $\bnu^{(1)}_n$ and $\bnu^{(2)}_n$ are represented as $\bnu^{(1)}_n=(\bnu_{n,k,j},\, k\in\mathbb{Z}_+, j\in S_0)$ and $\bnu^{(2)}_n=(\bnu_{k,n,j},\, k\in\mathbb{Z}_+, j\in S_0)$, our aim corresponds to clarifying the asymptotic properties of the stationary distributions $\bnu^{(1)}_n$ and $\bnu^{(2)}_n$ as $n$ tends to infinity. 

The triplet of block matrices $\{A_{-1}^{(1)}, A_0^{(1)}, A_1^{(1)}\}$ (resp.\ $\{A_{-1}^{(2)}, A_0^{(2)}, A_1^{(2)}\}$) is the Markov additive kernel of the Markov additive process obtained from $\{\bY_n^{(1)}\}$ (resp.\ $\{\bY_n^{(2)}\}$). 
Under Assumption \ref{as:Akl_irreducible}, the Markov additive kernels $\{A_{-1}^{(1)}, A_0^{(1)}, A_1^{(1)}\}$ and $\{A_{-1}^{(2)}, A_0^{(2)}, A_1^{(2)}\}$ are 1-arithmetic, where the kernel is said to be $1$-arithmetic if for some $k\ge 0$ and $j\in S_0$, the greatest common divisor of $\{n_1+\cdots+n_l; [A_{n_1}^{(i)} A_{n_2}^{(i)} \cdots A_{n_l}^{(i)}]_{(k,j),(k,j)}>0,\, n_1, ..., n_l\in\mathbb{H},\, l\ge 1 \}$ is one (see, for example, Miyazawa and Zhao \cite{Miyazawa04}).
Since the Markov additive kernels are 1-arithmetic, $R^{(1)}$ and $R^{(2)}$ are aperiodic (see Remark 4.4 of Miyazawa and Zhao \cite{Miyazawa04}).
According to Ozawa \cite{Ozawa13}, we assume the following condition throughout the paper. 
\begin{assumption} \label{as:irreducibleR1R2}
The rate matrices $R^{(1)}$ and $R^{(2)}$ are irreducible.
\end{assumption}

%
\subsection{Directional decay rates of the stationary distribution}

Here we summarize the results of Ozawa \cite{Ozawa13}. 
Let $z_1$ and $z_2$ be positive numbers. Then, $C(z_1,z_2)$ is nonnegative and, by Assumption \ref{as:Akl_irreducible}, it is irreducible and aperiodic. Let $\chi(z_1,z_2)$ be the Perron-Frobenius eigenvalue of $C(z_1,z_2)$, i.e., $\chi(z_1,z_2)=\spr(C(z_1,z_2))$. The modulus of every eigenvalue of $C(z_1,z_2)$ except $\chi(z_1,z_2)$ is strictly less than $\spr(C(z_1,z_2))$. 
According to Kingman \cite{Kingman61}, we say that a positive function $f(x,y)$ is superconvex in $(x,y)$ if $\log f(x,y)$ is convex in $(x,y)$; a superconvex function is also a convex function. We have the following proposition. 
\begin{proposition}[Proposition 3.1 of Ozawa \cite{Ozawa13}] \label{pr:chiconvex}
$\chi(e^{s_1},e^{s_2})$ is superconvex (hence convex) in $(s_1,s_2) \in \mathbb{R}^2$. 
\end{proposition}

Define a closed set $\bar{\Gamma}$ as 
\[
\bar{\Gamma}=\{ (s_1,s_2)\in\mathbb{R}^2 : \chi(e^{s_1},e^{s_2}) \le 1\}.
\]
Since $\chi(1,1)=\chi(e^0,e^0)=1$, $\bar{\Gamma}$ contains the point of $(0,0)$ and thus it is not empty. By Proposition \ref{pr:chiconvex}, $\bar{\Gamma}$ is a convex set. The following property holds for $\bar{\Gamma}$. 
\begin{lemma} \label{le:barGamma_bounded}
Under Assumption \ref{as:Akl_irreducible}, $\bar{\Gamma}$ is bounded. 
\end{lemma}

We give the proof of the lemma in Appendix \ref{sec:proof_barGamma_bounded}. 
For $i\in\{1,2\}$, define the lower and upper extreme values of $\bar{\Gamma}$ with respect to $s_i$, denoted by $\underline{s}_i^*$ and $\bar{s}_i^*$, as 
\[
\underline{s}_i^* = \min_{(s_1,s_2)\in\bar{\Gamma}} s_i,\quad 
\bar{s}_i^* = \max_{(s_1,s_2)\in\bar{\Gamma}} s_i,
\]
where $-\infty< \underline{s}_i^* \le 0$ and $0\le \bar{s}_i^* < \infty$. 
For $i\in\{1,2\}$, define $\underline{z}_i^*$ and $\bar{z}_i^*$ as $\underline{z}_i^*=e^{\underline{s}_i^*}$ and $\bar{z}_i^*=e^{\bar{s}_i^*}$, respectively, where $0<\underline{z}_i^*\le 1$ and $1\le \bar{z}_i^*<\infty$. 
Since $\bar{\Gamma}\subset\mathbb{R}^2$ is a bounded convex set and $\chi(z_1,e^{s_2})$ is convex in $s_2\in\mathbb{R}$, we see that, for each $z_1\in(\underline{z}_1^*,\bar{z}_1^*)$, equation $\chi(z_1,e^{s_2})=1$ has just two real solutions in $s_2\in\mathbb{R}$. Furthermore, we also see that, for $z_1=\underline{z}_1^*\mbox{ or }\bar{z}_1^*$, equation $\chi(z_1,e^{s_2})=1$ has just one real solution in $s_2\in\mathbb{R}$. 
This analogously holds for $\chi(e^{s_1},z_2)$. Hence, we immediately obtain the following proposition (see Proposition 3.9 of Ozawa \cite{Ozawa13}).
\begin{proposition} \label{pr:chisolution}
For each $z_1\in(\underline{z}_1^*,\bar{z}_1^*)$ (resp.\ $z_2\in(\underline{z}_2^*,\bar{z}_2^*)$), equation $\chi(z_1,z_2)=1$ has just two different real solutions $\underline{\zeta}_2(z_1)$ and $\bar{\zeta}_2(z_1)$ (resp.\ $\underline{\zeta}_1(z_2)$ and $\bar{\zeta}_1(z_2)$), where $\underline{\zeta}_2(z_1)<\bar{\zeta}_2(z_1)$ (resp.\ $\underline{\zeta}_1(z_2)<\bar{\zeta}_1(z_2)$). 
For $z_1=\underline{z}_1^*\mbox{ or }\bar{z}_1^*$ (resp.\ $z_2=\underline{z}_2^*\mbox{ or }\bar{z}_2^*$), it has just one real solution $\underline{\zeta}_2(z_1)=\bar{\zeta}_2(z_1)$ (resp.\ $\underline{\zeta}_1(z_2)=\bar{\zeta}_1(z_2)$). 
If $z_1\notin[\underline{z}_1^*,\bar{z}_1^*]$ (resp.\ $z_2\notin[\underline{z}_2^*,\bar{z}_2^*]$), it has no real solutions. 
\end{proposition}

Consider the following matrix quadratic equations of $X$:
\begin{align}
A_{*,-1}(z) + A_{*,0}(z) X + A_{*,1}(z) X^2 = X, 
\label{eq:equationG1} \\
A_{-1,*}(z) + A_{0,*}(z) X + A_{1,*}(z) X^2 = X,  
\label{eq:equationG2}
\end{align}
where, for $i\in\mathbb{H}$,
\[
A_{*,i}(z)=\sum_{j\in\mathbb{H}} A_{j,i} z^j,\quad 
A_{i,*}(z)=\sum_{j\in\mathbb{H}} A_{i,j} z^j. 
\]
Denote by $G_1(z)$ and $G_2(z)$ the minimum nonnegative solutions to matrix equations (\ref{eq:equationG1}) and (\ref{eq:equationG2}), respectively, if they exist. 
Furthermore, consider the following matrix quadratic equations of $X$:
\begin{align}
 X^2 A_{*,-1}(z) + X A_{*,0}(z) + A_{*,1}(z) = X, 
\label{eq:equationR1} \\
 X^2 A_{-1,*}(z) + X A_{0,*}(z) + A_{1,*}(z) = X.  
\label{eq:equationR2}
\end{align}
Denote by $R_1(z)$ and $R_2(z)$ the minimum nonnegative solutions to matrix equations (\ref{eq:equationR1}) and (\ref{eq:equationR2}), respectively, if they exist.
Since $\bar{\Gamma}$ is convex, we see that, for every $s_1\in[\underline{s}_1^*,\bar{s}_1^*]$, there exists at least one real number $s_2$ satisfying $\chi(e^{s_1},e^{s_2})\le 1$. This analogously holds for every $s_2\in[\underline{s}_2^*,\bar{s}_2^*]$. Hence, we obtain the next lemma.
\begin{lemma}[Lemma 3.4 of Ozawa \cite{Ozawa13}] \label{le:existenceG1}
For $z\in\mathbb{R}_+\setminus\{0\}$, the minimum nonnegative solutions $G_1(z)$ and $R_1(z)$ (resp.\ $G_2(z)$ and $R_2(z)$) to matrix equations (\ref{eq:equationG1}) and (\ref{eq:equationR1}) (resp.\ equations (\ref{eq:equationG2}) and (\ref{eq:equationR2})) exist if and only if $z\in[\underline{z}_1^*,\bar{z}_1^*]$ (resp.\ $z\in[\underline{z}_2^*,\bar{z}_2^*])$. 
\end{lemma}

$G_1(z)$, $R_1(z)$, $G_2(z)$ and $R_2(z)$ satisfy the following property. 
\begin{proposition}[Remark 3.5 of Ozawa \cite{Ozawa13}] \label{pr:sprG1}
For $z_1\in[\underline{z}_1^*,\bar{z}_1^*]$ and $z_2\in[\underline{z}_2^*,\bar{z}_2^*]$, 
\begin{align}
&\spr(G_1(z_1)) = \underline{\zeta}_2(z_1),\quad \spr(R_1(z_1)) = \bar{\zeta}_2(z_1)^{-1}, \\
&\spr(G_2(z_2)) = \underline{\zeta}_1(z_2),\quad \spr(R_2(z_2)) = \bar{\zeta}_1(z_2)^{-1}. 
\end{align}
\end{proposition}

Define matrix functions $C_1(z,X)$ and $C_2(X,z)$ as 
\begin{equation}
C_1(z,X) = A_{*,0}^{(1)}(z) + A_{*,1}^{(1)}(z) X,\quad 
C_2(X,z) = A_{0,*}^{(2)}(z) + A_{1,*}^{(2)}(z) X, 
\end{equation}
where $X$ is an $s_0\times s_0$ matrix and, for $i\in\mathbb{H}_+$,
\[
A_{*,i}^{(1)}(z) = \sum_{j\in\mathbb{H}} A_{j,i}^{(1)} z^j,\quad 
A_{i,*}^{(2)}(z) = \sum_{j\in\mathbb{H}} A_{i,j}^{(2)} z^j.
\]
Define $\psi_1(z)$ and $\psi_2(z)$ as
\[
\psi_1(z) = \spr(C_1(z,G_1(z))),\quad 
\psi_2(z) = \spr(C_2(G_2(z),z)). 
\]
By Propositions 3.5 and 3.6 of Ozawa \cite{Ozawa13} and their proofs, we obtain the following proposition. 
\begin{proposition} 
For $i\in\{1,2\}$, $\psi_i(e^s)$ is superconvex (hence convex) in $s\in[\underline{s}_i^*,\bar{s}_i^*]$. 
\end{proposition}

Next, we introduce points $(\theta_1^{(c)},\theta_2^{(c)})$ and $(\eta_1^{(c)},\eta_2^{(c)})$ on the closed curve $\chi(e^{s_1},e^{s_2})=1$, as follows:
\begin{align*}
\theta_1^{(c)} = \max\{s_1\in[\underline{s}_1^*,\bar{s}_1^*]: \psi_1(e^{s_1})\le 1\},\quad 
\theta_2^{(c)} = \log \underline{\zeta}_2(e^{\theta_1^{(c)}}), \cr
\eta_2^{(c)} = \max\{s_2\in[\underline{s}_2^*,\bar{s}_2^*]: \psi_2(e^{s_2})\le 1\},\quad 
\eta_1^{(c)} = \log \underline{\zeta}_1(e^{\eta_2^{(c)}}).
\end{align*}
Furthermore, define $\bar{\theta}_2^{(c)}$ and $\bar{\eta}_1^{(c)}$ as
\[
\bar{\theta}_2^{(c)} = \log \bar{\zeta}_2(e^{\theta_1^{(c)}}),\quad
\bar{\eta}_1^{(c)} = \log \bar{\zeta}_1(e^{\eta_2^{(c)}}).
\]
Since $\bar{\Gamma}$ is a convex set, we can classify the possible configuration of points $(\theta_1^{(c)},\theta_2^{(c)})$ and $(\eta_1^{(c)},\eta_2^{(c)})$ in the following manner. 
\begin{align*}
&\mbox{Type I: $\eta_1^{(c)}<\theta_1^{(c)}$ and $\theta_2^{(c)}<\eta_2^{(c)}$},\quad 
\mbox{Type II: $\eta_1^{(c)}<\theta_1^{(c)}$ and $\eta_2^{(c)}\le\theta_2^{(c)}$}, \\
&\mbox{Type III: $\theta_1^{(c)}\le\eta_1^{(c)}$ and $\theta_2^{(c)}<\eta_2^{(c)}$}. 
\end{align*}
Define the directional decay rates of the stationary distribution of the 2D-QBD process, denoted by $\xi_1$ and $\xi_2$, as
\begin{align*}
\xi_1 &= - \lim_{n\to\infty} \frac{1}{n} \log\nu_{n,j,k}\quad\mbox{for any $j$ and $k$},\\
\xi_2 &= - \lim_{n\to\infty} \frac{1}{n} \log\nu_{i,n,k}\quad\mbox{for any $i$ and $k$}. 
\end{align*}
Then, $\xi_1$ and $\xi_2$ are given as follows.
\begin{lemma}[Corollary 4.3 of Ozawa \cite{Ozawa13}] \label{le:decay_rate}
Under Assumptions \ref{as:irreducibleYn} to \ref{as:irreducibleR1R2}, the directional decay rates $\xi_1$ and $\xi_2$ exist and they are given by 
\begin{align}
&(\xi_1,\xi_2) = \left\{ 
\begin{array}{ll}
(\theta_1^{(c)},\eta_2^{(c)}), & \mbox{Tye I}, \cr
(\bar{\eta}_1^{(c)},\eta_2^{(c)}), & \mbox{Type II}, \cr
(\theta_1^{(c)},\bar{\theta}_2^{(c)}), & \mbox{Type III}.
\end{array} \right.
\end{align}
\end{lemma}

Hence, under Assumptions \ref{as:irreducibleYn} to \ref{as:irreducibleR1R2}, the directional geometric decay rate $r_1$ in $x_1$-coordinate and $r_2$ in $x_2$-coordinate exist,  and they are given by $r_1=e^{\xi_1}$ and $r_2=e^{\xi_2}$.

%
%
\subsection{Main results}

We assume the following technical condition. 
\begin{assumption} \label{as:G1_eigen_z1max}
All the eigenvalues of $G_1(r_1)$ are distinct. Also, those of $G_2(r_2)$ are distinct.
\end{assumption}

\begin{remark}
Assumption \ref{as:G1_eigen_z1max} is not necessary in Section \ref{sec:generating_function}. 
In Section \ref{sec:Gmatrix}, another assumption (Assumption \ref{as:eigenG1_distinct}) will be introduced instead of Assumption \ref{as:G1_eigen_z1max}. Assumption \ref{as:eigenG1_distinct} is a necessary condition for Assumption \ref{as:G1_eigen_z1max}. 
In Section \ref{sec:asymptotics}, we rely on Assumption \ref{as:G1_eigen_z1max}.
\end{remark}

The exact asymptotic formula $h_1(k)$ of the stationary distribution in $x_1$-coordinate and $h_2(k)$ in $x_2$-coordinate, where 
\[
\lim_{k\to\infty} \frac{\bnu_{k,0}}{h_1(k)} = nonzero\ constant\ vector,\quad
\lim_{k\to\infty} \frac{\bnu_{0,k}}{h_2(k)} = nonzero\ constant\ vector, 
\]
are given as follows. 
\begin{theorem}[Main results] \label{th:mainresults}
Under Assumptions \ref{as:irreducibleYn} through \ref{as:G1_eigen_z1max}, in the case of Type I, the exact asymptotic formulae are given by
\begin{align}
h_1(k) 
&= \left\{ \begin{array}{ll}
r_1^{-k}, & \psi_1(\bar{z}_1^*)>1, \cr
k^{-\frac{1}{2}} (\bar{z}_1^*)^{-k}, & \psi_1(\bar{z}_1^*)=1, \cr
k^{-\frac{1}{2}(2 l_1+1)} (\bar{z}_1^*)^{-k}, & \psi_1(\bar{z}_1^*)<1,
\end{array} \right. \quad
%
h_2(k) 
= \left\{ \begin{array}{ll}
r_2^{-k}, & \psi_2(\bar{z}_2^*)>1, \cr
k^{-\frac{1}{2}} (\bar{z}_2^*)^{-k}, & \psi_2(\bar{z}_2^*)=1, \cr
k^{-\frac{1}{2}(2 l_2+1)} (\bar{z}_2^*)^{-k}, & \psi_2(\bar{z}_2^*)<1,
\end{array} \right. 
\end{align}
where $l_1$ and $l_2$ are some positive integers. In the case of Type II, they are given by 
\begin{align}
h_1(k) 
&= \left\{ \begin{array}{ll}
r_1^{-k}, & \eta_2^{(c)}<\theta_2^{(c)}, \cr
k\,r_1^{-k}, & \eta_2^{(c)}=\theta_2^{(c)}\ \mbox{and}\ \psi_1(\bar{z}_1^*)>1, \cr
(\bar{z}_1^*)^{-k}, & \eta_2^{(c)}=\theta_2^{(c)}\ \mbox{and}\ \psi_1(\bar{z}_1^*)=1, \cr
k^{-\frac{1}{2}} (\bar{z}_1^*)^{-k}, & \eta_2^{(c)}=\theta_2^{(c)}\ \mbox{and}\ \psi_1(\bar{z}_1^*)<1,
\end{array} \right. \quad
%
h_2(k) =r_2^{-k}. 
\end{align}
In the case of Type III, the are given by 
\begin{align}
h_1(k) &=r_1^{-k}, \quad
%
h_2(k) 
= \left\{ \begin{array}{ll}
r_2^{-k}, & \theta_1^{(c)}<\eta_1^{(c)}, \cr
k\,r_2^{-k}, & \theta_1^{(c)}=\eta_1^{(c)}\ \mbox{and}\ \psi_2(\bar{z}_2^*)>1, \cr
(\bar{z}_2^*)^{-k}, & \theta_1^{(c)}=\eta_1^{(c)}\ \mbox{and}\ \psi_2(\bar{z}_2^*)=1, \cr
k^{-\frac{1}{2}} (\bar{z}_2^*)^{-k}, & \theta_1^{(c)}=\eta_1^{(c)}\ \mbox{and}\ \psi_2(\bar{z}_2^*)<1. 
\end{array} \right. \quad
\end{align}
\end{theorem}

This theorem will be proved in Section \ref{sec:asymptotics}.

%
%
\section{Generating functions and key expressions} \label{sec:generating_function}

\subsection{Vector generating function}

Let $\bvarphi(z,w)$ be the vector generating function defined as 
\[
\bvarphi(z,w)=\sum_{i=0}^\infty \sum_{j=0}^\infty \bnu_{i,j} z^i w^j. 
\]
This vector generating function satisfies 
\begin{equation}
\bvarphi(z,w) 
= \bnu_{0,0} + \bvarphi_1(z) + \bvarphi_2(w) + \bvarphi_+(z,w), 
\label{eq:bpsiz1z2}
\end{equation}
where 
\[
\bvarphi_1(z) = \sum_{i=1}^\infty \bnu_{i,0} z^i ,\quad 
\bvarphi_2(w)=\sum_{j=1}^\infty \bnu_{0,j} w^j,\quad 
\bvarphi_+(z,w)=\sum_{i=1}^\infty \sum_{j=1}^\infty \bnu_{i,j} z^i w^j. 
\]
We define the following matrix functions:
\begin{align*}
&C(z,w) = \sum_{i\in\mathbb{H}} \sum_{j\in\mathbb{H}} A_{i,j} z^i w^j, \quad
C_0(z,w) = \sum_{i\in\mathbb{H}_+} \sum_{j\in\mathbb{H}_+} A_{i,j}^{(0)} z^i w^j, \\ 
&C_1(z,w) = \sum_{i\in\mathbb{H}} \sum_{j\in\mathbb{H}_+} A_{i,j}^{(1)} z^i w^j, \quad
C_2(z,w) = \sum_{i\in\mathbb{H}_+} \sum_{j\in\mathbb{H}} A_{i,j}^{(2)} z^i w^j. 
\end{align*}
From the stationary equation $\bnu = \bnu P$, if $\bvarphi(z,w)$ is finite, we obtain 
\begin{align}
&\bvarphi(z,w) \cr
&\quad = \bnu_{0,0} C_0(z,w) + \bvarphi_1(z) C_1(z,w) + \bvarphi_2(w) C_2(z,w) + \bvarphi_+(z,w) C(z,w). 
\label{eq:bvarphi}
\end{align}

%
%
\subsection{Convergence domain}

Define the convergence domain $\calD$ of the vector generating function $\bvarphi(z,w)$ as
\[
\calD = \mbox{the interior of }\{(z,w)\in\mathbb{R}_+^2 : \bvarphi(z,w)<\infty\}.
\]
We obtain a sufficiently large subset of the convergence domain $\calD$.  The radius of convergence of $\bvarphi_1(z)$ and that of $\bvarphi_2(z)$ are given in terms of $r_1$ and $r_2$, as follows.
\begin{proposition} \label{pr:radius_conv}
The radius of convergence of the vector generating function $\bvarphi_1(z)$ (resp.\ $\bvarphi_2(z)$) is given by $r_1$ (resp.\ $r_2$), i.e., for $i\in\{1,2\}$, 
\[
\sup\!\left\{z\in\mathbb{R}_+: \bvarphi_i(z) < \infty \right\} = r_i. 
\]
\end{proposition}

Since this proposition is immediately obtained from Lemma \ref{le:decay_rate} and Cauchy's criterion, we omit the proof.  
\begin{remark}
From Proposition \ref{pr:radius_conv}, we see that if $z>r_1$ or $w>r_2$, then at least one entry of $\bvarphi(z,w)$ is divergent. 
\end{remark}

Define domains $\calD_0$ and $\calD_1$ as 
\begin{align*}
&\calD_0 = \left\{(z,w)\in\mathbb{R}_+^2 : 0<z<r_1,\ 0<w<r_2,\ \spr(C(z,w))<1 \right\}, \\
&\calD_1 = \left\{(z,w)\in\mathbb{R}_+^2 : \mbox{there exists $(z'_1,z'_2)\in\calD_0$ such that $(z,w)<(z'_1,z'_2)$} \right\}.
\end{align*}
Then, we have the following lemma (see Kobayashi et al.\ \cite{Kobayashi17}). 
\begin{lemma} \label{le:domainD1}
$\calD_0\subset\calD_1\subset\calD$.
\end{lemma}

This $\calD_1$ is the desired subset of $\calD$

\begin{remark}
$\calD_1$ is probably identical to $\calD$. This point is left as a further study. 
\end{remark}

%
%
\subsection{Matrix-type vector generating functions and key expressions}

In order to investigate singularity of the generating functions $\bvarphi_1(z)$ and $\bvarphi_2(z)$, we need key expressions for $\bvarphi_1(z)$ and $\bvarphi_2(z)$ corresponding to equation (29) in Kobayashi and Miyazawa \cite{Kobayashi13}. Therefore, we consider vector generating functions one of whose variables is a matrix. 
Here, we explain only about $\bvarphi_1(z)$ since results for $\bvarphi_2(z)$ are similar to those for $\bvarphi_1(z)$. 

Define the following vector generating function: 
\[
\bvarphi(z,X)=\sum_{i=0}^\infty \sum_{j=0}^\infty \bnu_{i,j}\,z^i\,X^j, 
\]
where $z$ is a scalar and $X$ is an $s_0\times s_0$ matrix.  
We have 
\begin{align}
\bvarphi(z,X) &= \bvarphi_+(z,X) + \bvarphi_1(z) + \bvarphi_2(X) + \bnu_{0,0}, 
\label{eq:phizX1}
\end{align}
where 
\[
\bvarphi_+(z,X) = \sum_{i=1}^\infty \sum_{j=1}^\infty \bnu_{i,j} z^i X^j,\quad 
\bvarphi_2(X) = \sum_{j=1}^\infty \bnu_{0,j} X^j.
\]
If $\bvarphi(z,X)<\infty$, we obtain from the stationary equation $\bnu P=\bnu$ that 
\begin{align}
\bvarphi(z,X) =&\ \sum_{i=1}^\infty \sum_{j=1}^\infty \bnu_{i,j} z^i \hat{C}(z,X) X^{j-1} + \bvarphi_1(z) C_1(z,X) \cr
&\qquad + \sum_{j=1}^\infty \bnu_{0,j} \hat{C}_2(z,X) X^{j-1} + \bnu_{0,0} C_0(z,X), 
\label{eq:phizX2}
\end{align}
where 
\begin{align*}
&\hat{C}(z,X) = \sum_{j\in\mathbb{H}} A_{*,j}(z) X^{j+1}, \quad
\hat{C}_2(z,X) = \sum_{j\in\mathbb{H}} A_{*,j}^{(2)}(z) X^{j+1}, \\
&C_0(z,X) = \sum_{j\in\mathbb{H}_+} A_{*,j}^{(0)}(z) X^j,\quad 
C_1(z,X) = \sum_{j\in\mathbb{H}_+} A_{*,j}^{(1)}(z)X^j,
\end{align*}
and for $j\in\mathbb{H}$ and $j'\in\mathbb{H}_+$, 
\begin{align*}
&A_{*,j}(z) = \sum_{i\in\mathbb{H}} A_{i,j} z^i, \quad
A_{*,j}^{(2)}(z) = \sum_{i\in\mathbb{H}_+} A_{i,j}^{(2)} z^i, \quad
A_{*,j'}^{(1)}(z) = \sum_{i\in\mathbb{H}} A_{i,j}^{(1)} z^i, \\
&A_{*,j'}^{(0)}(z) = \sum_{i\in\mathbb{H}_+} A_{i,j}^{(0)} z^i. 
\end{align*}
Note that $C_1(z,X)$, $A_{*,j}(z)$ and $A_{*,j}^{(1)}(z)$ have already been defined in Section \ref{sec:model}.
With respect to the convergence domain of $\bvarphi(z,X)$, we give the following lemma.  
\begin{lemma} \label{le:mgf_convergence}
For $n\in\mathbb{Z}_+$, let $\ba_n$ be an $1\times m$ complex vector and define a vector series $\bphi(w)$ as $\bphi(w)=\sum_{n=0}^\infty \ba_n w^n$. Furthermore, for an $m\times m$ complex matrix $X$, define $\bphi(X)$ as $\bphi(X)=\sum_{n=0}^\infty \ba_n X^n$. 
Assume that $\bphi(w)$ converges absolutely for all $w\in\mathbb{C}$ such that $|w|<r$ for some $r>0$. Then, if $\spr(X)<r$, $\bphi(X)$ also converges absolutely. 
\end{lemma}

Since the proof of this lemma is elementary, we give it in Appendix \ref{sec:proof_lemma_mgf_convergence}.
By Lemma \ref{le:domainD1}, $\bvarphi(z,w)$ converges absolutely for every $(z,w)\in\mathbb{C}^2$ such that $(|z|,|w|)\in\calD_1$. Hence, by Lemma \ref{le:mgf_convergence}, we immediately obtain a criterion for convergence of $\bvarphi(z,X)$, as follows. 
\begin{proposition} \label{pr:mgf_domain}
Let $z$ be a complex number and $X$ an $s_0\times s_0$ complex matrix. Then, $\bvarphi(z,X)$ converges absolutely if $(|z|,\spr(X))\in\calD_1$. 
\end{proposition}

The desired key expression for $\bvarphi_1(z)$ is given as follows.
\begin{lemma} \label{le:varphi1}
For $z\in[\underline{z}_1^*,\bar{z}_1^*]$, if $(z,\underline{\zeta}_2(z))\in\calD_1$, then we have
\begin{align}
\bvarphi_1(z) 
&= \biggl\{ \bvarphi_2^{\hat{C}_2}(z,G_1(z)) - \bvarphi_2(G_1(z)) + \bnu_{0,0} \bigl(C_0(z,G_1(z))-I\bigr) \biggr\} \Bigl(I-C_1(z,G_1(z))\Bigr)^{-1}, 
\label{eq:varphi1}
\end{align}
where
\[
\bvarphi_2^{\hat{C}_2}(z,X) = \sum_{j=1}^\infty \bnu_{0,j} \hat{C}_2(z,X) X^{j-1}.
\]
\end{lemma}
\begin{proof}
By Lemma \ref{le:existenceG1}, for $z\in[\underline{z}_1^*,\bar{z}_1^*]$, the nonnegative matrix $G_1(z)$ exists. Since \[
(z,\spr(G_1(z))=(z,\underline{\zeta}_2(z))\in\calD_1,
\]
we have, by Proposition \ref{pr:mgf_domain}, $\bvarphi(z,G_1(z))<\infty$. 
$G_1(z)$ satisfies equation (\ref{eq:equationG1}), which corresponds to $\hat{C}(z,X)=X$, and we obtain, from equation (\ref{eq:phizX2}), 
\begin{align}
\bvarphi(z,G_1(z)) &= \bvarphi_+(z,G_1(z)) + \bvarphi_1(z) C_1(z,G_1(z)) + \bvarphi_2^{\hat{C}_2}(z,G_1(z)) + \bnu_{0,0} C_0(z,G_1(z)). 
\end{align}
Hence, from this equation and equation (\ref{eq:phizX1}), we obtain
\begin{align}
\bvarphi_1(z) (I-C_1(z,G_1(z))) &= \bvarphi_2^{\hat{C}_2}(z,G_1(z)) - \bvarphi_2(G_1(z)) + \bnu_{0,0} (C_0(z,G_1(z))-I). 
\end{align}
If $(z,\underline{\zeta}_2(z))\in\calD_1$, then $z<r_1$ and we have $\psi_1(z)=\spr(C_1(z,G_1(z))<1$.  Hence, $I-C_1(z,G_1(z))$ is nonsingular and we obtain expression (\ref{eq:varphi1}).
\end{proof}

A similar expression holds for $\bvarphi_2(z)$ and it is given as follows. 
\begin{proposition} \label{pr:varphi2}
For $w\in[\underline{z}_2^*,\bar{z}_2^*]$, if $(\underline{\zeta}_1(w),w)\in\calD_1$, then we have
\begin{align}
\bvarphi_2(w) 
&= \biggl\{ \bvarphi_1^{\hat{C}_1}(G_2(w),w) - \bvarphi_1(G_2(w)) + \bnu_{0,0} \bigl(C_0(G_2(w),w)-I\bigr) \biggr\} \Bigl(I-C_2(G_2(w),w)\Bigr)^{-1}, 
\label{eq:varphi2}
\end{align}
where
\[
\bvarphi_1^{\hat{C}_1}(X,w) = \sum_{i=1}^\infty \bnu_{i,0} \hat{C}_1(X,w) X^{j-1}\quad 
\mbox{and}\quad 
\hat{C}_1(X,w) = \sum_{i\in\mathbb{H}} A_{i,*}^{(1)}(w) X^{j+1}; 
\]
$\bvarphi_1(X)$, $C_0(X,w)$ and $C_2(X,w)$ are defined in a manner similar to that used for defining $\bvarphi_2(X)$, $C_0(z,X)$ and $C_1(z,X)$. 
\end{proposition}

%
%
\section{Analytic extension of the G-matrix} \label{sec:Gmatrix}

In order to investigate singularities of $\bvarphi_1(z)$ by using expression (\ref{eq:varphi1}), we must extend $G_1(z)$ to complex variable $z$ and clarify singularities of $G_1(z)$. We carry out it in two steps: first, $G_1(z)$ is redefined as a series of matrices, which converges absolutely in a certain annular domain and next, it is analytically continued via the matrix quadratic equation (\ref{eq:equationG1}). 
Since $G_2(z)$ can analogously be extended, we explain only about $G_1(z)$ in this section.

%
%
\subsection{First definition of $G_1(z)$}

For $z\in\mathbb{C}$, we redefine $G_1(z)$ in a manner similar to that used for defining the so-called {\it G-matrix} of a QBD processes (see Neuts \cite{Neuts94}). 
For the purpose, we use the following sets of index sequences: for $n\ge 1$ and for $m\ge 1$, 
\begin{align*}
&\scrI_n = \biggl\{\bi_{(n)}\in\mathbb{H}^n:\ \sum_{l=1}^k i_l\ge 0\ \mbox{for $k\in\{1,2,...,n-1\}$}\ \mbox{and} \sum_{l=1}^n i_l=0 \biggr\}, \\
&\scrI_{D,m,n} = \biggl\{\bi_{(n)}\in\mathbb{H}^n:\ \sum_{l=1}^k i_l\ge -m+1\ \mbox{for $k\in\{1,2,...,n-1\}$}\ \mbox{and} \sum_{l=1}^n i_l=-m \biggr\}, \\
&\scrI_{U,m,n} = \biggl\{\bi_{(n)}\in\mathbb{H}^n:\ \sum_{l=1}^k i_l\ge 1\ \mbox{for $k\in\{1,2,...,n-1\}$}\ \mbox{and} \sum_{l=1}^n i_l=m \biggr\}, 
\end{align*}
where $\bi_{(n)}=(i_1,i_2,...,i_n)$. 
For $n\ge 1$, let $Q_{11}^{(n)}(z)$, $D_1^{(n)}(z)$ and $U_1^{(n)}(z)$ be defined as 
\begin{align*}
&Q_{11}^{(n)}(z) = \sum_{\bi_{(n)}\in\scrI_n} A_{*,i_1}(z) A_{*,i_2}(z) \cdots A_{*,i_n}(z), \\
&D_1^{(n)}(z) = \sum_{\bi_{(n)}\in\scrI_{D,1,n}} A_{*,i_1}(z) A_{*,i_2}(z) \cdots A_{*,i_n}(z), \\
&U_1^{(n)}(z) = \sum_{\bi_{(n)}\in\scrI_{U,1,n}} A_{*,i_1}(z) A_{*,i_2}(z) \cdots A_{*,i_n}(z), 
\end{align*}
and let $N_1(z)$, $R_1(z)$ and $G_1(z)$ be defined as 
\begin{align*}
&N_1(z) = \sum_{n=0}^\infty Q_{11}^{(n)}(z),\quad 
G_1(z) = \sum_{n=1}^\infty D_1^{(n)}(z), \quad 
R_1(z) = \sum_{n=1}^\infty U_1^{(n)}(z),
\end{align*}
where $Q_{11}^{(0)}(z)=I$. 
Let $Q(z)$ be defined as 
\[
Q(z) 
= \begin{pmatrix}
A_{*,0}(z) & A_{*,1}(z) & & & \cr
A_{*,-1}(z) & A_{*,0}(z) & A_{*,1}(z) & & \cr
& A_{*,-1}(z) & A_{*,0}(z) & A_{*,1}(z) & \cr
& \ddots & \ddots & \ddots & 
\end{pmatrix}, 
\]
and denote by $\tilde{Q}(z)$ the fundamental matrix of $Q(z)$, i.e., $\tilde{Q}(z)=\sum_{n=0}^\infty Q(z)^n$. Then, for $n\ge 0$, $Q_{11}^{(n)}$ is the $(1,1)$-block of $Q(z)^n$ and $N_1(z)$ is that of $\tilde{Q}(z)$. 
When $z=1$, the $(i,j)$-entry of $D_1^{(n)}(1)$ is the probability that a QBD process starting in phase $i$ of level $l$ enters level $l-1$ for the first time just after $n$ steps and the phase at that time is $j$. Hence, $G_1(1)$ is the G-matrix of the QBD process in a usual sense.
Furthermore, the $(i,j)$-entry of $U_1^{(n)}(1)$ is the probability that the QBD process starting in phase $i$ of level $l$ stays in phase $j$ of level $l+1$ just after $n$ steps and it does not enter any level lower than or equal to $l$ until that time. Hence, $R_1(1)$ is the rate matrix of the QBD process in a usual sense.
$N_1(z)$, $G_1(z)$ and $R_1(z)$ satisfy the following properties. 

\begin{lemma} \label{le:Gmatrix}
For $z\in\mathbb{C}[\underline{z}_1^*,\bar{z}_1^*]$, the following statements hold.
\begin{itemize}
\item[(i)] $N_1(z)$, $G_1(z)$ and $R_1(z)$ converge absolutely and satisfy 
\begin{equation}
|N(z)|\le N(|z|),\quad 
|G_1(z)|\le G_1(|z|), \quad 
|R_1(z)|\le R_1(|z|). 
\label{eq:absNRG}
\end{equation}
\item[(ii)] $G_1(z)$ and $R_1(z)$ are represented in terms of $N_1(z)$, as follows. 
\begin{equation}
G_1(z) = N_1(z) A_{*,-1}(z),\quad 
R_1(z) = A_{*,1}(z) N_1(z).
\end{equation}
\item[(iii)] $G_1(z)$ and $R_1(z)$ satisfy the following matrix quadratic equations: 
\begin{align}
A_{*,-1}(z)+A_{*,0}(z) G_1(z)+A_{*,1}(z) G_1(z)^2 = G_1(z), \label{eq:G1equation} \\
R_1(z)^2 A_{*,-1}(z)+R_1(z) A_{*,0}(z)+A_{*,1}(z) = R_1(z). \label{eq:R1equation}
\end{align}
\item[(iv)] Define $H_1(z)$ as $H_1(z) = A_{*,0}(z)+A_{*,1}(z) N_1(z) A_{*,-1}(z)$, then $N_1(z)$ satisfies 
\begin{equation}
(I-H_1(z)) N_1(z) = I. 
\label{eq:H1N1_relation}
\end{equation}
\item[(v)] For nonzero $w\in\mathbb{C}$, $I-C(z,w)$ satisfy the following factorization (e.g., see Lemma 3.1 of Miyazawa and Zhao \cite{Miyazawa04}). 
\begin{equation}
I-C(z,w) = \bigl(w^{-1}I-R_1(z)\bigr) \bigl(I-H_1(z)\bigr) \bigl(w I-G_1(z)\bigr). 
\label{eq:WFfact}
\end{equation}
\end{itemize}
\end{lemma}

Since the proof of this lemma is elementary, we give it in Appendix \ref{sec:proof_prop_Gmatrix}. 
From statement (iv) of Lemma \ref{le:Gmatrix}, we have 
\[
\det(I-H_1(z)) \det N_1(z) = 1
\]
and this leads us to the following fact. 
\begin{corollary} \label{co:N1_inverse} 
For $z\in\mathbb{C}[\underline{z}_1^*,\bar{z}_1^*]$, both $I-H_1(z)$ and $N_1(z)$ are nonsingular and they satisfy $N_1(z)=(I-H_1(z))^{-1}$. 
\end{corollary}

\begin{remark} \label{re:Gmatrix2}
For $w\in\mathbb{C}[\underline{z}_2^*,\bar{z}_2^*]$, $R_2(w)$ and $G_2(w)$ can analogously be redefined and they satisfy
\begin{equation}
I-C(z,w) = \bigl(z^{-1}I-R_2(w)\bigr) \bigl(I-H_2(w)\bigr) \bigl(z I-G_2(w)\bigr),  
\label{eq:WFfact2}
\end{equation}
where $H_2(w)=A_{0,*}(w)+A_{1,*}(w) G_2(w)$. Furthermore, $I-H_2(w)$ is nonsingular and $N_2(w)$ corresponding to $N_1(z)$ is given by $N_2(w)=(I-H_2(z))^{-1}$. We will use these facts later. 
\end{remark}

%
%

From the definition of $D_1^{(n)}(z)$, it can be seen that each entry of $D_1^{(n)}(z)$ is a Laurent polynomial in $z$ and hence, we can represent $D_1^{(n)}(z)$ as 
\[
D_1^{(n)}(z) = \sum_{k=-n}^n D_{1,k}^{(n)}\,z^k, 
\]
where each $D_{1,k}^{(n)}$ is a nonnegative square matrix. Note that $G_1(z)$ converges absolutely for $z\in\mathbb{C}[\underline{z}_1^*,\bar{z}_1^*]$. Hence, using the representation of $D_1^{(n)}(z)$ above, we obtain 
\begin{align}
G_1(z) = \sum_{n=1}^\infty \sum_{k=-n}^n D_{1,k}^{(n)} z^k = \sum_{k=-\infty}^\infty z^k \sum_{n=\max\{|k|,1\}}^\infty D_{1,k}^{(n)}. 
\end{align}
This is a Laurent expansion of $G_1(z)$, and since $G_1(z)$ is absolutely convergent in $\mathbb{C}[\underline{z}_1^*,\bar{z}_1^*]$, we obtain the following lemma (see, for example, Section II.1 of Markushevich \cite{Markushevich05}).
\begin{lemma} \label{le:G_analytic1}
$G_1(z)$ is analytic in the open annular domain $\mathbb{C}(\underline{z}_1^*,\bar{z}_1^*)$. 
\end{lemma}

%
%
\subsection{Eigenvalues of $G_1(z)$}

To obtain the Jordan canonical form of $G_1(z)$ in the following subsection, we identify the eigenvalues of $G_1(z)$ and clarify their properties. 
The Jordan canonical forms of usual G-matrices have extensively been studied in Gail et al.\ \cite{Gail96}. 
For $z,w\in\mathbb{C}$, define an $s_0$-dimensional square matrix function $L(z,w)$ as 
\[
L(z,w) = z w (C(z,w)-I) = z A_{*,-1}(z)+z(A_{*,0}(z)-I) w+z A_{*,1}(z) w^2, 
\]
and denote by $\phi(z,w)$ the determinant of $L(z,w)$, i.e., $\phi(z,w)=\det L(z,w)$. $L(z,w)$ is a matrix polynomial in $w$ with degree $2$ and every entry of $L(z,w)$ is a polynomial in $z$ and $w$. 
$\phi(z,w)$ is also a polynomial in $z$ and $w$. Denote by $m$ the degree of $\phi(z,w)$ in $w$, where $s_0\le m\le 2 s_0$. $\phi(z,w)$ is represented as 
\begin{equation}
\phi(z,w) = \sum_{k=0}^m p_k(z) w^k, 
\end{equation}
where for $k\in\{0,1,...,m\}$, $p_k(z)$ is a polynomial in $z$ and $p_m(z)$ is not identically zero. 
For $z\in\mathbb{C}$ such that $p_m(z)\ne 0$, let $\alpha_1(z),\,\alpha_2(z),...,\alpha_m(z)$ be the solutions to $\phi(z,w)=0$, counting multiplicities, i.e., $\phi(z,\alpha_k(z))=0$ for $k\in\{1,2,...,m\}$. 
For nonzero $z,w\in\mathbb{C}$, let $\lambda^C_1(z,w)$, $\lambda^C_2(z,w)$, ..., $\lambda^C_{s_0}(z,w)$ be the eigenvalues of $C(z,w)$, counting multiplicities. Recall that, for $z\in(\underline{z}_1^*,\bar{z}_1^*)$, equation $\chi(z,w)=\spr(C(z,w))=1$ has two different real solutions $\underline{\zeta}_2(z)<\bar{\zeta}_2(z)$, and for $z=\underline{z}_1^*$ or $\bar{z}_1^*$, it has just one real solution $\underline{\zeta}_2(z)=\bar{\zeta}_2(z)$. 
For $z\in\mathbb{C}[\underline{z}_1^*,\bar{z}_1^*]$, let $\lambda^{G_1}_1(z)$, $\lambda^{G_1}_2(z)$, ..., $\lambda^{G_1}_{s_0}(z)$ be the eigenvalues of $G_1(z)$, counting multiplicities, and $\lambda^{R_1}_1(z)$, $\lambda^{R_1}_2(z)$, ..., $\lambda^{R_1}_{s_0}(z)$ those of $R_1(z)$. 
Then, we have the following lemma. 
%
\begin{lemma} \label{le:zeros_and_G1eigenvalues}
Let $z$ be a number in $\mathbb{C}(\underline{z}_1^*,\bar{z}_1^*]$ and $m=m(z)$ be the degree of $\phi(z,w)$ in $w$. 
\begin{itemize}
\item[(i)] When $z\ne \bar{z}_1^*$, there exist a positive number $\varepsilon$ and just $s_0$ solutions to $\phi(z,w)=0$, say $\alpha_1(z), \alpha_2(z), ..., \alpha_{s_0}(z)$, counting multiplicities, such that $|\alpha_l(z)|<\underline{\zeta}_2(|z|)+\varepsilon$ for all $l\in\{1,2,...,s_0\}$. If $z$ is not a real number, $\varepsilon$ can be set at zero. 
Furthermore, the other $(m-s_0)$ solutions to $\phi(z,w)=0$, say $\alpha_{s_0+1}(z), \alpha_{s_0+2}(z), ..., \alpha_m(z)$, counting multiplicities, satisfy $|\alpha_l(z)|\ge\bar{\zeta}_2(|z|)$ for all $l\in\{s_0+1,s_0+2,...,m\}$. 
\item[(ii)] When $z=\bar{z}_1^*$, assume that the multiplicity of the solution $\alpha_{s_0}(z)=\underline{\zeta}_2(\bar{z}_1^*)$ as a zero of $\phi(z,w)$ is two. 
Then, there exist just $s_0-1$ solutions to $\phi(z,w)=0$, say $\alpha_1(z), \alpha_2(z), ..., \alpha_{s_0-1}(z)$, counting multiplicities, such that $|\alpha_l(z)|<\underline{\zeta}_2(\bar{z}_1^*)$ for all $l\in\{1,2,...,s_0-1\}$. 
Furthermore, there exists a solution to $\phi(z,w)=0$, say $\alpha_{s_0+1}(z)$, such that $\alpha_{s_0+1}(z)=\underline{\zeta}_2(\bar{z}_1^*)\ (=\bar{\zeta}_2(\bar{z}_1^*))$.
The other $(m-s_0-1)$ solutions to $\phi(z,w)=0$, say $\alpha_{s_0+2}(z), \alpha_{s_0+3}(z), ..., \alpha_m(z)$, counting multiplicities, satisfy $|\alpha_l(z)|>\bar{\zeta}_2(|\bar{z}_1^*|)$ for all $l\in\{s_0+2,s_0+3,...,m\}$. 
\item[(iii)] In both the cases (i) and (ii), the set of the first $s_0$ solutions corresponds to the set of the eigenvalues of $G(z)$, i.e., 
\[
\{\alpha_1(z), \alpha_2(z), ..., \alpha_{s_0}(z)\}=\{\lambda^{G_1}_1(z), \lambda^{G_1}_2(z), ..., \lambda^{G_1}_{s_0}(z)\}, 
\] 
and the set of the other $(m-s_0)$ solutions to $\phi(z,w)=0$ corresponds to the set of the reciprocals of the nonzero eigenvalues of $R_1(z)$, i.e., 
\[
\{\alpha_{s_0+1}(z),\alpha_{s_0+2}(z),...,\alpha_m(z)\}=\{\lambda^{-1}: \lambda\in\{\lambda^{R_1}_1(z),\lambda^{R_1}_2(z),...,\lambda^{R_1}_{s_0}(z)\},\ \lambda\ne 0 \}. 
\]
\end{itemize}
\end{lemma}

Before proving Lemma \ref{le:zeros_and_G1eigenvalues}, we present several propositions. 
%
\begin{proposition} \label{pr:sprCzw}
Under Assumption \ref{as:Akl_irreducible}, for $z,w\in\mathbb{C}$ such that $z\ne 0$ and $w\ne 0$, if $|z|\ne z$ or $|w|\ne w$, then $\spr(C(z,w))<\spr(C(|z|,|w|)$. 
\end{proposition}

Since the proof of this proposition is elementary, we give it in Appendix \ref{sec:sprCzw}. 
%
\begin{proposition} \label{pr:G_eigenvalues}
The following statements hold.
\begin{itemize}  
\item[(i)] If $z\in\mathbb{C}(\underline{z}_1^*,\bar{z}_1^*)$, then for every $l\in\{1,2,...,s_0\}$, $|\lambda^{G_1}_l(z)|\le \underline{\zeta}_2(|z|)$ and $|\lambda^{R_1}_l(z)|\le \bar{\zeta}_2(|z|)^{-1}$, where $\underline{\zeta}_2(|z|)<\bar{\zeta}_2(|z|)$. 
\item[(ii)] If $|z|=\bar{z}_1^*$ and $z\ne \bar{z}_1^*$, then for every $l\in\{1,2,...,s_0\}$, $|\lambda^{G_1}_l(z)|< \underline{\zeta}_2(\bar{z}_1^*)$ and $|\lambda^{R_1}_l(z)|< \bar{\zeta}_2(\bar{z}_1^*)^{-1}$, where $\underline{\zeta}_2(\bar{z}_1^*)=\bar{\zeta}_2(\bar{z}_1^*)$. 
\item[(iii)] Consider the case where $z=\bar{z}_1^*$ and assume that $\lambda^{G_1}_{s_0}(z)=\underline{\zeta}_2(\bar{z}_1^*)$ and $\lambda^{R_1}_{s_0}(z)=\bar{\zeta}_2(\bar{z}_1^*)^{-1}$. Furthermore, assume that the algebraic multiplicity of $\lambda^{G_1}_{s_0}(z)$ and that of $\lambda^{R_1}_{s_0}(z)$ are one. Then, for every $l\in\{1,2,...,s_0-1\}$, $|\lambda^{G_1}_l(z)|< \underline{\zeta}_2(\bar{z}_1^*)$ and $|\lambda^{R_1}_l(z)|< \bar{\zeta}_2(\bar{z}_1^*)^{-1}$, where $\underline{\zeta}_2(\bar{z}_1^*)=\bar{\zeta}_2(\bar{z}_1^*)$. 
\end{itemize}
\end{proposition}

\begin{proof}
By Lemma \ref{le:Gmatrix}-(i), we have, for $z\in\mathbb{C}[\underline{z}^*,\bar{z}_1^*]$, 
\begin{align}
&\spr(G_1(z))\le\spr(|G_1(z)|)\le\spr(G_1(|z|))=\underline{\zeta}_2(|z|) \label{eq:sprG1_ineq}\\
&\spr(R_1(z))\le\spr(|R_1(z)|)\le\spr(R_1(|z|))=\bar{\zeta}_2(|z|)^{-1}. \label{eq:sprR1_ineq}
\end{align}
From these formulae and the fact that if $z\in\mathbb{C}(\underline{z}^*,\bar{z}_1^*)$, then $\underline{\zeta}_2(|z|)<\bar{\zeta}_2(|z|)$, we obtain Statement (i). 

By Lemma \ref{le:Gmatrix}-(v), we have, for $z\in\mathbb{C}[\underline{z}_1^*,\bar{z}_1^*]$,
\begin{align}
\phi(z,w) &= (-1)^{s_0} z^{s_0} w^{s_0} f_H(z) \det(w^{-1} I-R_1(z)) \det(w I-G_1(z)), 
\end{align}
where $f_H(z)=\det(I-H_1(z))$ and, by Corollary \ref{co:N1_inverse}, we see $f_H(z)\ne 0$.
Let $m$ be the degree of the polynomial $\phi(z,w)$ in $w$. Then, $\phi(z,w)$ has just $m$ zeros, say $\alpha_1(z), \alpha_2(z), ..., \alpha_{m}(z)$, counting multiplicities, and we have 
\begin{equation}
\phi(z,w) = p_{m}(z) \prod_{k=1}^{m} (w-\alpha_k(z)), 
\end{equation}
where $p_{m}(z)\ne 0$. The eigenvalues of $G_1(z)$ (resp.\ $R_1(z)$) are the zeros of the characteristic polynomial $\det(w I-G_1(z))$ in $w$ (resp.\ $\det(w^{-1} I-R_1(z))$ in $w^{-1}$), and we have 
\begin{align}
&\det(w I-G_1(z)) = \prod_{l=1}^{s_0} (w-\lambda^{G_1}_l(z)),\quad 
\det(w^{-1} I-R_1(z)) = \prod_{l=1}^{s_0} (w^{-1}-\lambda^{R_1}_l(z)). 
\end{align}
Hence, we obtain 
\begin{align}
&p_{m}(z) \prod_{k=1}^{m} (w-\alpha_k(z))
= (-1)^{s_0} z^{s_0} f_H(z) \biggl( \prod_{l=1}^{s_0} \left(w-\lambda^{G_1}_l(z)\right) \biggr) \biggl( \prod_{l=1}^{s_0} \left(1-w \lambda^{R_1}_l(z)\right) \biggr). 
\label{eq:phiGR}
\end{align}
Both sides of equation (\ref{eq:phiGR}) are polynomials in $w$. Hence, the degree of the right hand side of the equation must be $m$, and $2 s_0-m$ eigenvalues of $R_1(z)$ must be zero. Therefore, without loss of generality, we assume that $\lambda^{R_1}_l(z)=0$ for $l\in\{m-s_0+1,m-s_0+2,...,2 s_0\}$. Then, for $l\in\{1,2,...,m-s_0\}$, $\lambda^{R_1}_l(z)\ne 0$ and equation (\ref{eq:phiGR}) becomes 
\begin{align}
\phi(z,w) 
&= p_{m}(z) \prod_{k=1}^{m} (w-\alpha_k(z)) \cr
&= p_{m}(z)\biggl( \prod_{l=1}^{s_0} \left(w-\lambda^{G_1}_l(z)\right) \biggr) \biggl( \prod_{l=1}^{m-s_0} \left(w- \lambda^{R_1}_l(z)^{-1}\right) \biggr), 
\label{eq:phiGR2}
\end{align}
where 
\[
p_{m}(z) = (-1)^{s_0} z^{s_0} f_H(z) \biggl( \prod_{l=1}^{m-s_0} \left(-\lambda^{R_1}_l(z)\right) \biggr) \ne 0.
\]
By Proposition \ref{pr:sprCzw}, for $|z|=\bar{z}_1^*$ such that $z\ne \bar{z}_1^*$ and for $|w|= \underline{\zeta}_2(\bar{z}_1^*)$, we obtain 
\begin{align}
\spr(C(z,w)) < \spr(C(\bar{z}_1^*,\underline{\zeta}_2(\bar{z}_1^*)) = 1. 
\end{align}
This implies that $\det(C(z,w)-I)\ne 0$ and $\phi(z,w) = z^{s_0} w^{s_0} \det(C(z,w)-I) \ne 0$. 
Hence, from equation (\ref{eq:phiGR2}), we see that this $w$ is neither an eigenvalue of $G_1(z)$ nor the reciprocal of a nonzero eigenvalue of $R_1(z)$. Therefore, from formulae (\ref{eq:sprG1_ineq}) and (\ref{eq:sprR1_ineq}), we obtain statement (ii). 

Under the assumption of statement (iii), we consider the case where $z=\bar{z}_1^*$. Since the algebraic multiplicity of $\lambda^{G_1}_{s_0}(z)=\underline{\zeta}_2(\bar{z}_1^*)$ and that of $\lambda^{R_1}_{s_0}(z)=\bar{\zeta}_2(\bar{z}_1^*)^{-1}$ are one, we have, for every $l\in\{1,2,...,s_0-1\}$, $\lambda^{G_1}_l(z)\ne\underline{\zeta}_2(\bar{z}_1^*)$ and $\lambda^{R_1}_l(z)\ne\bar{\zeta}_2(\bar{z}_1^*)^{-1}$. 
Suppose that, for some $l\in\{1,2,...,s_0-1\}$, $|\lambda^{G_1}_l(z)|=\underline{\zeta}_2(\bar{z}_1^*)$. Then, by Proposition \ref{pr:sprCzw}, we have $\phi(z,\lambda^{G_1}_l(z))\ne 0$ and this contradicts equation (\ref{eq:phiGR2}). Hence, for every $l\in\{1,2,...,s_0-1\}$, $|\lambda^{G_1}_l(z)|\ne \underline{\zeta}_2(\bar{z}_1^*)$. Analogously, we have, for every $l\in\{1,2,...,s_0-1\}$, $|\lambda^{R_1}_l(z)|\ne \bar{\zeta}_2(\bar{z}_1^*)^{-1}$. Therefore, from formulae (\ref{eq:sprG1_ineq}) and (\ref{eq:sprR1_ineq}), we obtain statement (iii) and this completes the proof.
\end{proof}

%
\begin{proof}[Proof of Lemma \ref{le:zeros_and_G1eigenvalues}]
Since $p_m(z)\ne 0$,  we obtain from equation (\ref{eq:phiGR2}) that 
\begin{align}
&\prod_{k=1}^m (w-\alpha_k(z)) = \biggl( \prod_{l=1}^{s_0} \left(w-\lambda^{G_1}_l(z)\right) \biggr) \biggl( \prod_{l=1}^{m-s_0} \left(w- \lambda^{R_1}_l(z)^{-1}\right) \biggr). 
\label{eq:alphaGR}
\end{align} 
This equation implies that the eigenvalues of $G_1(z)$ and the reciprocals of the nonzero eigenvalues of $R_1(z)$ are the zeros of the polynomial $\phi(z,w)$ in $w$ and vice versa. Therefore, denoting the zeros of $\phi(z,w)$ corresponding to the eigenvalues of $G_1(z)$ by $\alpha_1(z), \alpha_2(z), ..., \alpha_{s_0}(z)$, we have 
\[
\{\alpha_1(z), \alpha_2(z), ..., \alpha_{s_0}(z)\}=\{\lambda^{G_1}_1(z), \lambda^{G_1}_2(z), ..., \lambda^{G_1}_{s_0}(z)\}.
\]
Other zeros of $\phi(z,w)$, $\alpha_{s_0+1}(z),\alpha_{s_0+2}(z),...,\alpha_m(z)$, correspond to the reciprocals of the nonzero eigenvalues of $R_1(z)$. Hence, denoting them by $\lambda^{R_1}_{2 s_0-m+1}(z),\lambda^{R_1}_{2 s_0-m+2}(z),...,\lambda^{R_1}_{s_0}(z)$, we have
\begin{align*}
\{\alpha_{s_0+1}(z),\alpha_{s_0+2}(z),...,\alpha_m(z)\}
&=\{\lambda^{-1}: \lambda\in\{\lambda^{R_1}_1(z),\lambda^{R_1}_2(z),...,\lambda^{R_1}_{s_0}(z)\},\ \lambda\ne 0 \} \cr
&=\{\lambda^{R_1}_{2 s_0-m+1}(z)^{-1},\lambda^{R_1}_{2 s_0-m+2}(z)^{-1},...,\lambda^{R_1}_{s_0}(z)^{-1} \}.
\end{align*}
This complets the proof of statement (iii) of the Lemma. 

If $|z|\ne\bar{z}_1^*$, then by Proposition \ref{pr:G_eigenvalues}-(i), setting $\varepsilon=(\bar{\zeta}_2(|z|)-\underline{\zeta}_2(|z|))/2$, we have, for every $l\in\{1,2,...,s_0\}$ and $l'\in\{2 s_0-m+1,2 s_0-m+2,...,s_0\}$, 
\[
|\lambda^{G_1}_l(z)| < \underline{\zeta}_2(|z|)+\varepsilon < \bar{\zeta}_2(|z|) \le |\lambda^{R_1}_{l'}(z)|^{-1}. 
\]
Hence, from statement (iii) of the lemma, we obtain, for every $l\in\{1,2,...,s_0\}$ and $l'\in\{2 s_0-m+1,2 s_0-m+2,...,s_0\}$, 
\[
|\alpha_l(z)| < \underline{\zeta}_2(|z|)+\varepsilon < \bar{\zeta}_2(|z|) \le |\alpha_{l'}(z)|^{-1}. 
\]
When $|z|\ne z$, suppose $|\alpha_l(z)|=\underline{\zeta}_2(|z|)$ for some $\l\in\{1,2,...,s_0\}$. Then, by Proposition \ref{pr:sprCzw}, we have $\spr(C(z,\alpha_l(z)))<\spr(C(|z|,\underline{\zeta}_2(|z|)))=1$, and this contradicts $\phi(z,\alpha_l(z))=0$. 
Hence, $|\alpha_l(z)|\ne\underline{\zeta}_2(|z|)$ and we see that $\varepsilon$ can be set at $0$ when $|z|\ne z$. This completes the proof of statement (i) of the lemma. 

If $|z|=\bar{z}_1^*$ and $z\ne\bar{z}_1^*$, then, by Proposition \ref{pr:G_eigenvalues}-(ii), setting 
\[
\varepsilon= \frac{1}{2} \left( \min_{1\le k\le m-s_0} |\lambda^{R_1}_{2 s_0-m+k}(z)|^{-1} - \bar{\zeta}_2(\bar{z}_1^*) \right), 
\] 
we have, for every $l\in\{1,2,...,s_0\}$ and $l'\in\{2 s_0-m+1,2 s_0-m+2,...,s_0\}$, 
\[
|\lambda^{G_1}_l(z)| < \bar{\zeta}_2(\bar{z}_1^*)+\varepsilon < |\lambda^{R_1}_{l'}(z)|^{-1}. 
\]
Under the assumption of statement (ii), if $z=\bar{z}_1^*$, we have just two solutions that equal $\underline{\zeta}_2(\bar{z}_1^*)$: one is an eigenvalue of $G_1(z)$, say $\alpha_{s_0}(z)$, and the other that of $R_1(z)$, say $\alpha_{s_0+1}(z)$. The algebraic multiplicities of those eigenvalues are one. Hence, by Proposition \ref{pr:G_eigenvalues}-(iii) and statement (iii) of the lemma, we obtain statement (ii) of the lemma. 
\end{proof}

%
%
\subsection{Analytic continuation of $G_1(z)$}

Hereafter, we assume the following technical condition. 
\begin{assumption} \label{as:eigenG1_distinct}
For some $z_0\in\mathbb{C}(\underline{z}_1^*,\bar{z}_1^*]$, all the eigenvalues of $G_1(z_0)$ are distinct. 
\end{assumption}
Although this assumption always holds under Assumption \ref{as:G1_eigen_z1max}, we assume it in this section instead of Assumption \ref{as:G1_eigen_z1max} in order to make explanation simple. 
Under Assumption \ref{as:eigenG1_distinct}, we define a set of points at which the algebraic multiplicity of some eigenvalue of $G_1(z)$ is greater than one. 
Before doing it, we define a notation. Let $f(z,w)$ be an irreducible polynomial in $z$ and $w$ whose degree with respect to $w$ is $k$. Assume the degree $k$ is greater than or equal to one and denote by $a(z)$ the coefficient of $w^k$ in $f(z,w)$. This $a(z)$ is a polynomial in $z$ satisfying $a(z)\not\equiv 0$ but it may be a nonzero constant. Define a point set $\Xi(f)$ as 
\begin{align*}
\Xi(f) 
&= \{ z\in\mathbb{C}: \mbox{$a(z)=0$ or ($f(z,w)=0$ and $f_w(z,w)=0$ for some $w\in\mathbb{C}$)} \}, 
\end{align*}
where $f_w(z,w)=(\partial/\partial w)f(z,w)$. Each point in $\Xi(f)$ is called an exceptional point of the algebraic function $w=\alpha(z)$ defined by polynomial equation $f(z,w)=0$, and it is an algebraic singularity of $\alpha(z)$, i.e., a removable singularity, pole or branch point with a finite order. For any point $z\in\mathbb{C}\setminus\Xi(f)$, $f(z,w)=0$ has just $k$ distinct solutions, which correspond to the $k$ branches of the algebraic function. 

Without loss of generality, we assume that, for some $n\in\mathbb{N}$ and $l_1,l_2,...,l_n\in\mathbb{N}$, the polynomial $\phi(z,w)=\det L(z,w)$ is factorized as 
\begin{equation}
\phi(z,w) = f_1(z,w)^{l_1} f_2(z,w)^{l_2} \cdots f_n(z,w)^{l_n}, 
\label{eq:phi_fk}
\end{equation}
where $f_k(z,w),\,k=1,2,...,n$, are irreducible polynomials in $z$ and $w$ and they are relatively prime. Here we note that, for some $k$, $f_k(z,w)$ may be a polynomial in either $z$ or $w$. Since the field of coefficients of polynomials we consider is $\mathbb{C}$, this factorization is unique. 
Let $z_0$ be the pint given in Assumption \ref{as:eigenG1_distinct} and $\alpha_k(z_0),\,k=1,2,...,s_0$, the solutions to $\phi(z_0,w)=0$ that correspond to the eigenvalues of $G_1(z_0)$.  By Lemma \ref{le:zeros_and_G1eigenvalues}, we see that, for every $k\in\{1,2,...,s_0\}$, the solution $\alpha_k(z_0)$ is distinct from all other solutions to $\phi(z_0,w)=0$. 
Hence, for each $k\in\{1,2,...,s_0\}$, there exists a unique $q(k)\in\{1,2,...,n\}$ such that $\alpha_k(z_0)$ is a solution to $f_{q(k)}(z_0,w)=0$. Note that there may exist $k$ and $k'$ such that $k\ne k'$ and $q(k)=q(k')$. From factorization (\ref{eq:phi_fk}), we immediately know that $l_{q(k)}=1$ for $k\in\{1,2,...,s_0\}$. 
Define a point set $\calE_1$ as 
\[
\calE_1 = \bigcup_{k=1}^{s_0} \Xi(f_{q(k)}). 
\]
Since, for each $k$, $f_{q(k)}$ is irreducible and not identically zero, the point set $\calE_1$ is finite. 
If $z_0\in\calE_1$, then replace $z_0$ with another point $z_0'$ in a neighborhood of $z_0$ such that $z_0'\in\mathbb{C}(\underline{z}_1^*,\bar{z}_1^*]\setminus\calE_1$ and it satisfies the condition of Assumption \ref{as:eigenG1_distinct}. It is possible because of continuity of the solutions to $\phi(z_0,w)=0$ and finiteness of $\calE_1$. Hereafter, in that case, we denote $z_0'$ by $z_0$. 
Note that $\bar{z}_1^*\in\calE_1$. This will be explained in Subsection \ref{sec:G1_extension} (see Lemma \ref{le:singularity_alpha_s0}).

For every $k\in\{1,2,...,s_0\}$ and any $z\in\mathbb{C}(\underline{z}_1^*,\bar{z}_1^*]\setminus\calE_1$, let $\check{\alpha}_k(z)$ be the analytic continuation of $\alpha_k(z_0)$ along a common path on $\mathbb{C}(\underline{z}_1^*,\bar{z}_1^*]\setminus\calE_1$, starting from $z_0$ and terminating at $z$. We give the following lemma. 
%
\begin{lemma} \label{le:tildealpha_continued}
For any $z\in\mathbb{C}(\underline{z}_1^*,\bar{z}_1^*]\setminus\calE_1$, the set $\{\check{\alpha}_1(z),\check{\alpha}_2(z),...,\check{\alpha}_{s_0}(z)\}$ is unique, i.e., it does not depend on the path along which all $\check{\alpha}_k(z),\,k=1,2,...,s_0$, are simultaneously continued, and it is identical to the set of the eigenvalues of $G_1(z)$. 
\end{lemma}
%
\begin{proof}[Proof of Lemma \ref{le:tildealpha_continued}]
Note that, for each $k\in\{1,2,...,s_0\}$, $\check{\alpha}_k(z)$ is a branch of the algebraic function given by the polynomial $f_{q(k)}(z,w)$ and it can analytically be continued along any path on $\mathbb{C}\setminus\Xi(f_{q(k)})$. 
Further note that $\bar{z}_1^*\in\calE_1$ (see Subsection \ref{sec:G1_extension}). 
Consider three real valued functions of a complex variable $z$: $|\check{\alpha}_k(z)|$, $\underline{\zeta}_2(|z|)$ and $\bar{\zeta}_2(|z|)$. We have, for $z\in\mathbb{C}(\underline{z}_1^*,\bar{z}_1^*)$, $\underline{\zeta}_2(|z|)<\bar{\zeta}_2(|z|)$ and, for $z$ such that $|z|=\bar{z}_1^*$, $\underline{\zeta}_2(|z|)=\bar{\zeta}_2(|z|)$. 
Furthermore, $\check{\alpha}_k(z)$ is a solution to $\phi(z,w)=0$ and, by Lemma \ref{le:zeros_and_G1eigenvalues}, we know that, for $z\in\mathbb{C}(\underline{z}_1^*,\bar{z}_1^*)\setminus\calE_1$, $\check{\alpha}_k(z)$ satisfies either $|\check{\alpha}_k(z)|\le \underline{\zeta}_2(|z|)$ or $|\check{\alpha}_k(z)|\ge\bar{\zeta}_2(|z|)$ and, for $z\in\mathbb{C}$ such that $|z|=\bar{z}_1^*$ and $z\ne\bar{z}_1^*$, $\check{\alpha}_k(z)$ satisfies either $|\check{\alpha}_k(z)|<\underline{\zeta}_2(\bar{z}_1^*)$ or $|\check{\alpha}_k(z)|\ge\bar{\zeta}_2(\bar{z}_1^*)=\underline{\zeta}_2(\bar{z}_1^*)$. 

Let $z_s$ and $z_t$ be points in $\in\mathbb{C}(\underline{z}_1^*,\bar{z}_1^*]\setminus\calE_1$ and assume that if $|z_s|\ne\bar{z}_1^*$, then $|\check{\alpha}_k(z_s)|\le\underline{\zeta}_2(|z_s|),\,k=1,2,...,s_0$, and otherwise, $|\check{\alpha}_k(z_s)|<\underline{\zeta}_2(\bar{z}_1^*),\,k=1,2,...,s_0$. 
Consider an arbitrary path $\xi_{z_s,z_t}$ on $\mathbb{C}(\underline{z}_1^*,\bar{z}_1^*]\setminus\calE_1$, starting from $z_s$ and terminating at $z_t$. For each $k\in\{1,2,...,s_0\}$, analytically continue $\check{\alpha}_k(z_s)$ along the path $\xi_{z_s,z_t}$. Then, we obtain function $\check{\alpha}_k(z)$, which is analytic on $\xi_{z_s,z_t}$. 
We show that $|\check{\alpha}_k(z_t)|\le\underline{\zeta}_2(|z_t|)$ when $|z_t|\ne\bar{z}_1^*$ and that  $|\check{\alpha}_k(z_t)|<\underline{\zeta}_2(\bar{z}_1^*)$ when $|z_t|=\bar{z}_1^*$. 
Suppose it is not the case. Then, we have that $\underline{\zeta}_2(|z_t|)< \bar{\zeta}_2(|z_t|)\le |\check{\alpha}_k(z_t)|$ when $|z_t|\ne\bar{z}_1^*$ or that $\underline{\zeta}_2(\bar{z}_1^*)\le |\check{\alpha}_k(z_t)|$ when $|z_t|=\bar{z}_1^*$. 
In both the cases, since $|\check{\alpha}_k(z)|$, $\underline{\zeta}_2(|z|)$ and $\bar{\zeta}_2(|z|)$ are continuous on the path $\xi_{z_s,z_t}$, there exists a point $z'\in\xi_{z_s,z_t}$ such that $|z'|\ne\bar{z}_1^*$ and $\underline{\zeta}_2(|z'|)<|\check{\alpha}_k(z')|<\bar{\zeta}_2(|z'|)$ or that $|z'|=\bar{z}_1^*$ and $|\check{\alpha}_k(z')|=\underline{\zeta}_2(\bar{z}_1^*)$. 
This contradicts Lemma \ref{le:zeros_and_G1eigenvalues} since $z'\ne\bar{z}_1^*$ and $\check{\alpha}_k(z')$ is a solution to $\phi(z',w)=0$. Hence, we obtain the desired result. 

For $z_t\in\mathbb{C}(\underline{z}_1^*,\bar{z}_1^*]\setminus\calE_1$, consider an arbitrary path $\xi_{z_0,z_t}$ on $\mathbb{C}(\underline{z}_1^*,\bar{z}_1^*]\setminus\calE_1$. From the fact obtained above, we see that, for every $k\in\{1,2,...,s_0\}$, by analytically continuing $\alpha_k(z_0)$ along the path $\xi_{z_0,z_t}$, we obtain $\check{\alpha}_k(z_t)$ satisfying $|\check{\alpha}_k(z_t)| \le \underline{\zeta}_2(|z_t|)$. 
Each $\check{\alpha}_k(z_t)$ is a solution to $\phi(z_t,w)=0$ and, by Lemma \ref{le:zeros_and_G1eigenvalues}, the set $\{\check{\alpha}_1(z_t),\check{\alpha}_2(z_t),...,\check{\alpha}_{s_0}(z_t)\}$ is identical to the set of the eigenvalues of $G_1(z_t)$. This also implies that the set $\{\check{\alpha}_1(z_t),\check{\alpha}_2(z_t),...,\check{\alpha}_{s_0}(z_t)\}$ remains unchained as a set even if we use anther path connecting $z_0$ and $z_t$ for continuing each $\alpha_k(z_0),\,k=1,2,...,s_0$. 
This completes the proof.
\end{proof}

\begin{remark}
For some $k\in\{1,2,...,s_0\}$, there may exist a branch point of $\check{\alpha}_k(z)$ in $\mathbb{C}(\underline{z}_1^*,\bar{z}_1^*]$. 
Lemma \ref{le:tildealpha_continued} asserts that even in that case, the set $\{\check{\alpha}_1(z),\check{\alpha}_2(z),...,\check{\alpha}_{s_0}(z)\}$ takes a single value as a set, for any $z\in\mathbb{C}(\underline{z}_1^*,\bar{z}_1^*]\setminus\calE_1$. 
\end{remark}

\begin{remark}
From the proof of Lemma \ref{le:tildealpha_continued}, we see that, for every $k\in\{1,2,...,s_0\}$, $\check{\alpha}_k(z)$ is bounded in a neighborhood of every point in $\calE_1\cap\mathbb{C}(\underline{z}_1^*,\bar{z}_1^*)$ since, for every $z\in\mathbb{C}(\underline{z}_1^*,\bar{z}_1^*]\setminus\calE_1$, $|\check{\alpha}_k(z)|\le \sup_{z\in(\underline{z}_1^*,\bar{z}_1^*]} \underline{\zeta}_2(z)< \bar{z}_2^*$.
Hence, any point in $\calE_1\cap\mathbb{C}(\underline{z}_1^*,\bar{z}_1^*)$ is not a pole; it is a removable singularity or branch point. 

\end{remark}

Define a point set $\calE_2$ as
\begin{align*}
\calE_2 
&= \{z\in\mathbb{C}\setminus\calE_1: \mbox{$f_{q(k)}(z,w)=f_{q(k')}(z,w)=0$ } \cr
&\qquad\qquad \mbox{for some $k,k'\in\{1,2,...,s_0\}$ such that $q(k)\ne q(k')$ and for some $w\in\mathbb{C}$} \}.
\end{align*}
Since, for any $k,k'\in\{1,2,...,s_0\}$ such that $q(k)\ne q(k')$, $f_{q(k)}(z,w)$ and $ f_{q(k')}(z,w)$ are relatively prime, the point set $\calE_2$ is finite. At each $z\in\calE_2\cap\mathbb{C}(\underline{z}_1^*,\bar{z}_1^*]$, some eigenvalues of $G_1(z)$ may coalesce with each other but any point in $\calE_2$ is not a singularity of any eigenvalue of $G_1(z)$. 
From Lemma \ref{le:tildealpha_continued}, we immediately obtain the following. 
\begin{corollary} \label{co:eigenG1_distinct}
For $z\in\mathbb{C}(\underline{z}_1^*,\bar{z}_1^*]\setminus(\calE_1\cup\calE_2)$, all the eigenvalues of $G_1(z)$, $\alpha_k(z),\,k=1,2,...,s_0$, are distinct. Furthermore, for each $z\in\mathbb{C}(\underline{z}_1^*,\bar{z}_1^*]\setminus(\calE_1\cup\calE_2)$ and each $k\in\{1,2,...,s_0\}$, the multiplicity of $\alpha_k(z)$ as a zero of $L(z,w)$ is one. 
\end{corollary}

Let $z_0\in\mathbb{C}(\underline{z}_1^*,\bar{z}_1^*]\setminus\calE_1$ be the point given in Assumption \ref{as:eigenG1_distinct} and, for $z_t\in\mathbb{C}(\underline{z}_1^*,\bar{z}_1^*]\setminus\calE_1$, consider an arbitrary path $\xi_{z_0,z_t}$ on $\mathbb{C}(\underline{z}_1^*,\bar{z}_1^*]\setminus\calE_1$. 
For $z$ in a neighborhood of each point on $\xi_{z_0,z_t}$, define a diagonal matrix function $\check{J}_1(z)$ as
\[
\check{J}_1(z) = \diag(\check{\alpha}_1(z),\check{\alpha}_2(z),...,\check{\alpha}_{s_0}(z)), 
\]
which is a Jordan canonical form of $G_1(z)$ on $\xi_{z_0,z_t}\setminus\calE_2$. 
By Corollary \ref{co:eigenG1_distinct} and the discussion in Section 7.1 of Gohberg et al.\ \cite{Gohberg09}, if $z\in\xi_{z_0,z_t}\setminus\calE_2$, then for every $k\in\{1,2,...,s_0\}$, we have
\[
\dim \Ker\ L(z,\check{\alpha}_k(z)) =1, 
\]
and if $z\in\xi_{z_0,z_t}\cap\calE_2$, then we have $\dim \Ker\ L(z,\check{\alpha}_k(z)) \ge 1$. Furthermore, $L(z,\check{\alpha}_k(z))$ is entry-wise analytic on $\xi_{z_0,z_t}$. Hence, by Theorem S6.1 of Gohberg et al.\ \cite{Gohberg09}, there exists a vector function $\check{\bv}_k(z)$ that is entry-wise analytic on $\xi_{z_0,z_t}$ and satisfies 
\begin{equation}
L(z,\check{\alpha}_k(z)) \check{\bv}_k(z) = \bzero. 
\label{eq:Lv_zero}
\end{equation}
For each $z\in\xi_{z_0,z_t}\setminus\calE_2$, $\check{\bv}_k(z)$ is unique, up to multiplication by a constant, as a vector satisfying equation (\ref{eq:Lv_zero}). 
Define a matrix function $\check{V}_1(z)$ as 
\[
\check{V}_1(z) = \begin{pmatrix} \check{\bv}_1(z) & \check{\bv}_2(z) & \cdots & \check{\bv}_{s_0}(z) \end{pmatrix}.
\]
In terms of $\check{J}_1(z)$ and $\check{V}_1(z)$, define a matrix function $\check{G}_1(z)$ as
\begin{equation}
\check{G}_1(z) = \frac{\check{V}_1(z)\,\check{J}_1(z)\, \adj\ \check{V}_1(z)}{\det \check{V}_1(z)}. 
\label{eq:tildeG1_definition}
\end{equation}
Since every entry of $\check{J}_1(z)$ and $\check{V}_1(z)$ is an analytic function on $\xi_{z_0,z_t}$, every entry of $\check{G}_1(z)$ is a meromorphic function on $\xi_{z_0,z_t}$. We give the following proposition.
%
\begin{lemma} \label{le:checkG1_analytic}
For any path $\xi_{z_0,z_t}$ on $\mathbb{C}(\underline{z}_1^*,\bar{z}_1^*]\setminus\calE_1$, $\check{G}_1(z)$ is analytic and identical to $G_1(z)$ on $\xi_{z_0,z_t}$. Furthermore, for any $z\in\mathbb{C}(\underline{z}_1^*,\bar{z}_1^*]\setminus\calE_1$, $\check{G}_1(z)$ is unique, i.e., it does not depend on the path along which $\check{G}_1(z)$ is continued. 
\end{lemma}
\begin{proof}
First, we demonstrate that $\check{G}_1(z)$ is identical to $G_1(z)$ on $\xi_{z_0,z_t}\setminus\calE_2$. 
By Lemma \ref{le:zeros_and_G1eigenvalues} and Corollary \ref{co:N1_inverse}, both $\check{\alpha}_k(z)^{-1}I-R_1(z)$ and $I-H_1(z)$ are nonsingular for any $z\in\xi_{z_0,z_t}$ and $k\in\{1,2,...,s_0\}$, where $\bar{z}_1^*\notin\xi_{z_0,z_t}$. Hence, from equation (\ref{eq:WFfact}), we obtain 
\begin{equation}
\check{\alpha}_k(z) I-G_1(z) = (I-H_1(z))^{-1} (\check{\alpha}_k(z)^{-1}I-R_1(z))^{-1} (I-C(z,\check{\alpha}_k(z))). 
\label{eq:WF_inverse}
\end{equation}
For $k\in\{1,2,...,s_0\}$, since $L(z,\check{\alpha}_k(z))=z \check{\alpha}_k(z)(C(z,\check{\alpha}_k(z))-I)$, equation (\ref{eq:WF_inverse}) implies that $\check{\bv}_k(z)$ is the eigenvector of $G_1(z)$ with respect to the eigenvalue $\check{\alpha}_k(z)$. 
For any $z\in\xi_{z_0,z_t}\setminus\calE_2$, all the eigenvalues of $G_1(z)$ are distinct and the vectors $\check{\bv}_k(z),\,k=1,2,...,s_0$, are linearly independent. Hence, $\check{V}_1(z)$ is nonsingular and we have, for $z\in\xi_{z_0,z_t}\setminus\calE_2$, 
\[
G_1(z) = \check{V}_1(z) \check{J}_1(z) \check{V}_1(z)^{-1} = \check{G}_1(z). 
\]

For any $z_1\in\xi_{z_0,z_t}\cap\calE_2$, we have 
\[
\lim_{\xi_{z_0,z_t}\ni z\to z_1} \check{G}_1(z) = \lim_{\xi_{z_0,z_t}\ni z\to z_1} G_1(z) = G_1(z_0). 
\]
This implies that the point $z_1$ is not a pole of any entry of $\check{G}_1(z)$, and $\check{G}_1(z)$ is entry-wise analytic at $z=z_1$ since it is entry-wise meromorphic on $\xi_{z_0,z_t}$. 
By Lemma \ref{le:tildealpha_continued}, since $\check{G}_1(z)$ remains unchanged by any permutation of the eigenvalues and eigenvectors in expression (\ref{eq:tildeG1_definition}), it does not depend on the path along which $\check{G}_1(z)$ is continued. 
This completes the proof.
\end{proof}

By Lemma \ref{le:checkG1_analytic} and the identity theorem, we see that $\check{G}_1(z)$ is the analytic extension of $G_1(z)$. Denoting the extension by the same notation $G_1(z)$, we immediately obtain the following corollary. 
%
\begin{corollary} \label{co:G1_analytic_extension}
The extended $G_1(z)$ is analytic on $\partial\Delta_{\bar{z}_1^*}\setminus\calE_1$, where $\partial\Delta_{\bar{z}_1^*}=\partial\Delta(0,\bar{z}_1^*)$. 
\end{corollary}

%
%
\subsection{Singularity of $G_1(z)$ on $\partial\Delta_{\bar{z}_1^*}$} \label{sec:G1_extension}

By Lemma \ref{le:G_analytic1}, $G_1(z)$ is entry-wise analytic in $\mathbb{C}(\underline{z}_1^*,\bar{z}_1^*)$ and, by Corollary \ref{co:G1_analytic_extension}, the extended $G_1(z)$ is entry-wise analytic on $\partial\Delta_{\bar{z}_1^*}\setminus\calE_1$. 
In this subsection, we will demonstrate that the extended $G_1(z)$ is also entry-wise analytic on $\partial\Delta_{\bar{z}_1^*}\cap\calE_1$ except for the point $\bar{z}_1^*$. 

First, we consider property of the eigenvalues of $G_1(z)$ at $z=\bar{z}_1^*$. Let $\lambda_1^C(z,w)$, $\lambda_2^C(z,w)$, ..., $\lambda_{s_0}^C(z,w)$ be the eigenvalues of $C(z,w)$, counting multiplicity. Without loss of generality, we assume $\lambda_{s_0}^C(z,w)$ satisfies $\lambda_{s_0}^C(\bar{z}_1^*,\underline{\zeta}_2(\bar{z}_1^*))=\chi(\bar{z}_1^*,\underline{\zeta}_2(\bar{z}_1^*))=1$. 
Since $C(1,1)$ ($=A_{*,*}$) is irreducible and aperiodic, $C(\bar{z}_1^*,\underline{\zeta}_2(\bar{z}_1^*))$ is also irreducible and aperiodic. Hence, the eigenvalue $\lambda_{s_0}^C(\bar{z}_1^*,\underline{\zeta}_2(\bar{z}_1^*))$ is simple, i.e., the algebraic multiplicity is one, since it is the Perron-Frobenius eigenvalue of $C(\bar{z}_1^*,\underline{\zeta}_2(\bar{z}_1^*))$. 
By the implicit function theorem, this implies that $\lambda_{s_0}^C(z,w)$ is analytic as a function of two complex variables in a neighborhood $\mathbb{U}^2$ of the point $(\bar{z}_1^*,\underline{\zeta}_2(\bar{z}_1^*))$. $\lambda^C_{s_0}(z,w)$ is, therefore, the analytic extension of $\chi(z,w)$ in $\mathbb{U}^2$. 
Since $L(z,w)=z w(C(z,w)-I)$, the eigenvalues of $L(z,w)$ are given by $z w(\lambda^C_k(z,w)-1),\,k=1,2,...,s_0$, and they are the solutions to the characteristic equation $\det\,(x I-L(z,w))=0$. Hence, we have 
\begin{equation}
z^{s_0} w^{s_0} \prod_{k=1}^{s_0} (\lambda^C_k(z,w)-1) = \det L(z,w) = \phi(z,w). 
\label{eq:lambdaC}
\end{equation}
This implies that if $\lambda^C_k(z,h(z))=1$ for some $k$ and for some function $h(z)$, then $w=h(z)$ is a solution to the polynomial equation $\phi(z,w)=0$ and vice versa, where we consider only the case where $z\ne 0$ and $h(z)\ne 0$.

Define functions $g$, $g_w$ and $g_{w^2}$ as 
\[
g(z,w)=\lambda^C_{s_0}(z,w)-1,\quad 
g_w(z,w)=\frac{\partial}{\partial w} g(z,w),\quad 
g_{w_2}(z,w)=\frac{\partial^2}{\partial w^2} g(z,w), 
\]
where $g(z,w)$ is analytic in a neighborhood $\mathbb{U}^2$ of $(\bar{z}_1^*,\underline{\zeta}_2(\bar{z}_1^*))$ and satisfies $g(\bar{z}_1^*,\underline{\zeta}_2(\bar{z}_1^*))=0$. 
Since equation $\chi(\bar{z}_1^*,w)=1$ has the multiple real root $\underline{\zeta}_2(\bar{z}_1^*)=\bar{\zeta}_2(\bar{z}_1^*)$, we have $g_w(\bar{z}_1^*,\underline{\zeta}_2(\bar{z}_1^*))=0$. Let the neighborhood $\mathbb{U}^2$ be so small that $g(z,w)\ne 0$ and $g_w(z,w)\ne 0$ in $\mathbb{U}^2\setminus\{(\bar{z}_1^*,\underline{\zeta}_2(\bar{z}_1^*))\}$. It is possible since both $g(z,w)$ and $g_w(z,w)$ are analytic at $(z,w)=(\bar{z}_1^*,\underline{\zeta}_2(\bar{z}_1^*))$. 
By the implicit function theorem and the identity theorem, we, therefore, have a function $h(z)$ that is analytic and satisfies $g(z,h(z))=0$ for $z\in\mathbb{U}_{\bar{z}_1^*}\setminus\{\bar{z}_1^*\}$ and $h(z)=\underline{\zeta}_2(z)$ for $z\in(\underline{z}_1^*,\bar{z}_1^*)\cap\mathbb{U}_{\bar{z}_1^*}$, where $\mathbb{U}_{\bar{z}_1^*}$ is a neighborhood of the point $\bar{z}_1^*$. 
By the discussion above, $w=h(z)$ is a solution to $\phi(z,w)=0$ and hence we denote it by $\alpha_{s_0}(z)$. This $\alpha_{s_0}(z)$ is also an eigenvalue of $G_1(z)$, which is the maximum eigenvalue when $z\in(\underline{z}_1^*,\bar{z}_1^*]\cap\mathbb{U}_{\bar{z}_1^*}$. 
We give the following lemma. 
%
\begin{lemma} \label{le:singularity_alpha_s0}
The point $\bar{z}_1^*$ is a branch point of $\alpha_{s_0}(z)$ with order one. 
\end{lemma}
\begin{proof}
By Proposition \ref{pr:chiconvex}, $\chi(\bar{z}_1^*,e^t)$ is convex in $t\in\mathbb{R}$ and we obtain 
\begin{align*}
0 &\ne \frac{d^2}{d t^2} g(\bar{z}_1^*,e^t) \Big|_{t=\log \underline{\zeta}_2(\bar{z}_1^*)} 
= \underline{\zeta}_2(\bar{z}_1^*)^2 g_{w^2}(\bar{z}_1^*,\underline{\zeta}_2(\bar{z}_1^*)). 
\end{align*}
where we use the fact that $g_w(\bar{z}_1^*,\underline{\zeta}_2(\bar{z}_1^*))=0$. 
This implies that $g_{w^2}(\bar{z}_1^*,\underline{\zeta}_2(\bar{z}_1^*))\ne 0$ and we know that the multiplicity of $\underline{\zeta}_2(\bar{z}_1^*)$ as a zero of $g(\bar{z}_1^*,w)$ is just two. 
We show that the multiplicity of $\underline{\zeta}_2(\bar{z}_1^*)$ as a zero of $\phi(\bar{z}_1^*,w)$ is also just two. Suppose it is not the case. Then, there are at least three identical solutions to $\phi(\bar{z}_1^*,w)=0$, including $\underline{\zeta}_2(\bar{z}_1^*)$ and $\bar{\zeta}_2(\bar{z}_1^*)$. 
Denote the other solution by $\alpha(\bar{z}_1^*)$, which equals $\underline{\zeta}_2(\bar{z}_1^*)$. Then, by equation (\ref{eq:lambdaC}), the solution must satisfies, for some $k\in\{1,2,...,s_0-1\}$, $\lambda^C_k(\bar{z}_1^*,\alpha(\bar{z}_1^*))=1$. This contradicts that the eigenvalue $\lambda^C_{s_0}(\bar{z}_1^*,\underline{\zeta}_2(\bar{z}_1^*))$, which equals one, is simple, and we obtain the desired result. 
By the implicit function theorem, for $z$ in a neighborhood of $\bar{z}_1^*$ such that $z\ne\bar{z}_1^*$, we have 
\[
\alpha_{s_0}'(z) = - \frac{g_z(z,\alpha_{s_0}(z))}{g_w(z,\alpha_{s_0}(z))},
\]
where $\alpha_{s_0}'(z)=(d/d z)\alpha_{s_0}(z)$ and $g_z(z,w)=(\partial/\partial z)g(z,w)$. 
Since $\chi(e^s,e^t)$ is convex in $(s,t)\in\mathbb{R}^2$, we have $g_z(\bar{z}_1^*,\underline{\zeta}_2(\bar{z}_1^*))\ne 0$. Hence, 
\[
\lim_{\mathbb{R}\ni z\to\bar{z}_1^*-} |\alpha_{s_0}'(z)| = \infty, 
\]
and $\bar{z}_1^*$ is a singularity not being removable. Since $\alpha_{s_0}(z)$ is bounded in a neighborhood of $\bar{z}_1^*$, the point $\bar{z}_1^*$ is not a pole of $\alpha_{s_0}(z)$. Hence, it is a branch point of $\alpha_{s_0}(z)$. 
Since the multiplicity of $\alpha_{s_0}(\bar{z}_1^*)$ as a zero of $\phi(\bar{z}_1^*,w)$ is two, the order of the branch point is one. 
\end{proof}

Denote by $\alpha_{s_0+1}(z)$ the other branch that coalesces with $\alpha_{s_0}(z)$ at $z=\bar{z}_1^*$. $\alpha_{s_0+1}(z)$ is a solution to $f_{q(s_0)}(z,w)=0$ and satisfies $\alpha_{s_0+1}(z)=\bar{\zeta}_2(z)$ for every $z\in(\underline{z}_1^*,\bar{z}_1^*]\cap\mathbb{U}_{\bar{z}_1^*}$, where $\mathbb{U}_{\bar{z}_1^*}$ is a neighborhood of $\bar{z}_1^*$. It is also the reciprocal of an eigenvalue of $R_1(z)$. 
From Lemma \ref{le:singularity_alpha_s0}, it is expected that the point $\bar{z}_1^*$ is also a branch point of $G_1(z)$. We state this point in the next lemma. Before doing it, we define a matrix function $G_1^{co}(z)$ that coalesces with $G_1(z)$ at $z=\bar{z}_1^*$, according to Lemma 4 of Li and Zhao \cite{Li03}. 
We use the same notations in the previous subsection. 
Let $\bu_{s_0}(z)$ be the left eigenvector of $G_1(z)$ with respect to the eigenvalue $\alpha_{s_0}(z)$. Let $\bv_{s_0+1}(z)$ be the column vector function that satisfies $L(z,\alpha_{s_0+1}) \bv_{s_0+1}(z)=\bzero$ and $\bu_{s_0}(z) \bv_{s_0+1}(z)=1$. Define $G_1^{co}(z)$ as 
\begin{equation}
G_1^{co}(z) = G_1(z) + \bigl( \alpha_{s_0+1}(z) I-G_1(z) \bigr) \bv_{s_0+1}(z) \bu_{s_0}(z). 
\label{eq:define_Gmax}
\end{equation}
Then, in a manner similar to that used in the proof of Lemma 4 of Li and Zhao \cite{Li03}, we see that $G_1^{co}(z)$ satisfies the following matrix quadratic equation:
\[
A_{*,-1}(z)+(I-A_{*,0}(z)) G_1^{co}(z)+A_{*,1}(z) G_1^{co}(z)^2 =O.
\]
The eigenvalues of $G_1^{co}(z)$ are given by $\alpha_1(z),\alpha_2(z),...,\alpha_{s_0-1}(z)$ and $\alpha_{s_0+1}(z)$, and the corresponding right eigenvectors are given by $\bv_1(z),\bv_2(z),...,\bv_{s_0-1}(z)$ and $\bv_{s_0+1}(z)$, respectively. 
We present the following lemma. 
%
\begin{lemma} \label{le:G1_singularity} 
$G_1(z)$ is analytic on the circle $\partial\Delta_{\bar{z}_1^*}$ except for $\bar{z}_1^*$. The point $\bar{z}_1^*$ is a branch pint of $G_1(z)$ with order one, and $G_1(z)$ coalesces with $G_1^{co}(z)$ at the point. 
\end{lemma}
%
\begin{proof}
Let $z_1$ be a point on $\partial\Delta_{\bar{z}_1^*}\cap\calE_1$. First, assuming $(\partial\Delta_{\bar{z}_1^*}\setminus\{\bar{z}_1^*\})\cap\calE_1\ne\emptyset$, we consider the case where $z_1\ne\bar{z}_1^*$. 
Let $\mathbb{U}_{z_1}$ be a neighborhood of $z_1$ satisfying $(\mathbb{U}_{z_1}\setminus\{z_1\})\cap\calE_1=\emptyset$. We set $\mathbb{U}_{z_1}$ being so small that there exists a positive number $\varepsilon$ such that for every $z\in\mathbb{U}_{z_1}$, $|\alpha_k(z)|<\underline{\zeta}_2(\bar{z}_1^*)-\varepsilon$ for $k\in\{1,2,...,s_0\}$ and $|\alpha_k(z)|\ge\underline{\zeta}_2(\bar{z}_1^*)-\varepsilon$ for $k\in\{s_0+1,s_0+2,...,m\}$. 
By Lemma \ref{le:zeros_and_G1eigenvalues}, it is possible since $|\alpha_k(z_1)|<\underline{\zeta}_2(\bar{z}_1^*)$ for $k\in\{1,2,...,s_0\}$ and $|\alpha_k(z_1)|\ge\underline{\zeta}_2(\bar{z}_1^*)$ for $k\in\{s_0+1,s_0+2,...,m\}$.
%
For $k\in\{1,2,...,s_0\}$ and for $z\in\mathbb{U}_{z_1}$, let $\bv_k(z)$ be a vector function satisfies 
\begin{equation}
L(z,\alpha_k(z)) \bv_k(z) = \bzero. 
\end{equation}
Like $\check{\bv}_k(z)$, $\bv_k(z)$ is unique, up to multiplication by a constant, and analytic in $\mathbb{U}_{z_1}\setminus\{z_1\}$. For $z\in\mathbb{U}_{z_1}$, define a matrix function $V_1(z)$ as
\[
V_1(z) = \begin{pmatrix} \bv_1(z) & \bv_2(z) & \cdots & \bv_{s_0}(z) \end{pmatrix}.
\]
Define a point set $\calE_3(\mathbb{U}_{z_1})$ as
\[
\calE_3(\mathbb{U}_{z_1}) = \{z\in\mathbb{U}_{z_1}\setminus\{z_1\}: \det V_1(z)= 0\}.
\]
Since $\det V_1(z)$ is analytic in $\mathbb{U}_{z_1}\setminus\{z_1\}$ and $\det V_1(z) \not\equiv 0$, we know by the identity theorem that there are no accumulation points of $\calE_3(\mathbb{U}_{z_1})$ in $\mathbb{U}_{z_1}\setminus\{z_1\}$. 
We show that $z_1$ is not an accumulation point of $\calE_3(\mathbb{U}_{z_1})$. 
Note that, by Lemma \ref{le:zeros_and_G1eigenvalues}, for every $k\in\{1,2,...,s_0\}$, $\alpha_k(z_1)$ is bounded and the point $z_1$ is a branch point or removable singularity of $\alpha_k(z)$. 
For $k\in\{1,2,...,s_0\}$, if $z_1$ is a branch point of $\alpha_k(z)$, let $\nu_k$ be the order of the branch point; if $z_1$ is a removable singularity, let $\nu_k=0$. Furthermore, let $\nu$ be the least common multiple of $\{\nu_1+1,\nu_2+1,...,\nu_{s_0}+1\}$. 
For $k\in\{1,2,...,s_0\}$, define a function $\tilde{\alpha}_k(\zeta)$ analytic in a neighborhood of the origin, as follows: if $z_1$ is a branch point of $\alpha_k(z)$, let $\tilde{\alpha}_k(\zeta)$ be the function that is analytic in a neighborhood of the origin and satisfies
\[
\alpha_k(z) = \tilde{\alpha}_k((z_1-z)^{\frac{1}{\nu_k+1}}); 
\]
if $z_1$ is a removable singularity, define $\tilde{\alpha}_k(\zeta)$ as $\tilde{\alpha}_k(\zeta) = \alpha_k(z_1-\zeta)$. In the former case, considering the Puiseux series expansion of $\alpha_k(z)$ around $z_1$, it can be seen that such an analytic function exists. 
For $k\in\{1,2,...,s_0\}$, define a vector function $\tilde{\bv}_k(\zeta)$ entry-wise analytic in a neighborhood of the origin, as follows: if $z_1$ is a branch point of $\alpha_k(z)$, let $\tilde{\bv}_k(\zeta)$ be the vector function that is entry-wise analytic in a neighborhood of the origin and satisfies
\begin{equation}
L(z_1-\zeta^{\nu_k+1},\tilde{\alpha}_k(\zeta)) \tilde{\bv}_k(\zeta) = \bzero; 
\label{eq:Ltildev_zero}
\end{equation}
if $z_1$ is a removable singularity, define $\tilde{\bv}_k(\zeta)$ as $\tilde{\bv}_k(\zeta) = \bv_k(z_1-\zeta)$. 
In the former case, since every entry of $L(z_1-\zeta^{\nu_k+1},\tilde{\alpha}_k(\zeta))$ is analytic in a neighborhood of the origin as a function of $\zeta$, it can be seen, by Theorem S6.1 of Gohberg et al.\ \cite{Gohberg09}, that such an analytic vector function exists. 
In terms of $\tilde{\bv}_k(\zeta)$, $V_1(z)$ is represented as 
\begin{equation}
V_1(z) = \begin{pmatrix} 
\tilde{\bv}_1((z_1-z)^{\frac{1}{\nu_1+1}}) & \tilde{\bv}_2((z_1-z)^{\frac{1}{\nu_2+1}}) & \cdots &\tilde{\bv}_{s_0}((z_1-z)^{\frac{1}{\nu_{s_0}+1}})
\end{pmatrix}. 
\end{equation}
Define $V_1^*(\zeta)$ as 
\[
V_1^*(\zeta) 
= V_1(z_1-\zeta^\nu)
= \begin{pmatrix} 
\tilde{\bv}_1(\zeta^{\frac{\nu}{\nu_1+1}}) & \tilde{\bv}_2(\zeta^{\frac{\nu}{\nu_2+1}}) & \cdots &\tilde{\bv}_{s_0}(\zeta^{\frac{\nu}{\nu_{s_0}+1}})
\end{pmatrix},  
\]
which is entry-wise analytic in a neighborhood $\mathbb{U}_0$ of the origin, and a point set $\calE_3^*(\mathbb{U}_0)$ as 
\[
\calE_3^*(\mathbb{U}_0) = \{ \zeta\in\mathbb{U}_0: z_1-\zeta^\nu\in\calE_3(\mathbb{U}_{z_1}) \}.
\]
Suppose $z_1$ is an accumulation point of $\calE_3(\mathbb{U}_{z_1})$. By the definition of $\calE_3^*(\mathbb{U}_0)$, this implies that the origin (point $0$) is an accumulation point of $\calE_3^*(\mathbb{U}_0)$. 
Note that $\det V_1^*(\zeta)=0$ for every $\zeta\in\calE_3^*(\mathbb{U}_0)$. Since $\det V_1^*(\zeta)$ is analytic in $\mathbb{U}_0$ and $\det V_1^*(\zeta)\not\equiv 0$, there are no accumulation points of $\calE_3^*(\mathbb{U}_0)$ in $\mathbb{U}_0$. 
This is a contradiction. Therefore, $z_1$ is not an accumulation point of $\calE_3(\mathbb{U}_{z_1})$. 
By this result, we can set the neighborhood $\mathbb{U}_{z_1}$ so small that $\det V_1(z)\ne 0$ for every $z\in\mathbb{U}_{z_1}\setminus\{z_1\}$. Hereafter, we assume $\mathbb{U}_{z_1}$ satisfies this property. 

Let $z_0$ be a point in $\mathbb{U}_{z_1}\cap\mathbb{C}(\underline{z}_1^*,\bar{z}_1^*]$ satisfying the condition of Assumption \ref{as:eigenG1_distinct} and $z_t$ an arbitrary point in $\mathbb{U}_{z_1}\setminus\{z_1\}$.
Consider an arbitrary path $\xi_{z_0,z_t}$ on $\mathbb{U}_{z_1}\setminus\{z_1\}$ and analytically continue $\check{\alpha}_k(z_0)$ ($=\alpha_k(z_0)$), $k=1,2,...,s_0$, along the path $\xi_{z_0,z_t}$. 
Then, because of the fact explained at the beginning of the proof, the same result of Lemma \ref{le:tildealpha_continued} holds, i.e., for any $z\in\mathbb{U}_{z_1}\setminus\{z_1\}$, the set $\{\check{\alpha}_1(z),\check{\alpha}_2(z),...,\check{\alpha}_{s_0}(z)\}$ is unique and it is identical to the set $\{\alpha_1(z),\alpha_2(z),...,\alpha_{s_0}(z)\}$. 
Consider $\check{G}_1(z)$ define by formula (\ref{eq:tildeG1_definition}). Since $\det V_1(z)\ne 0$ for every $z\in\mathbb{U}_{z_1}\setminus\{z_1\}$, we have $\det \check{V}_1(z)\ne 0$ for every $z\in\xi_{z_0,z_t}$ and $\check{G}_1(z)$ is entry-wise analytic on $\xi_{z_0,z_t}$. 
For any $z\in\mathbb{U}_{z_1}\setminus\{z_1\}$, since the set $\{\check{\alpha}_1(z),\check{\alpha}_2(z),...,\check{\alpha}_{s_0}(z)\}$ is unique, this $\check{G}_1(z)$ is also unique, i.e., it does not depend on the path along which $\check{G}_1(z)$ is continued. 
Hence, every entry of $\check{G}_1(z)$ is a single valued function in $\mathbb{U}_{z_1}\setminus\{z_1\}$ and the point $z_1$ is not a branch point. 
Consider a matrix function $G_1^*(\zeta)$ defined as 
\[
G_1^*(\zeta) = \check{G}_1(z_1-\zeta^\nu).
\]
By the discussion about $V_1^*(\zeta)$, we see that $G_1^*(\zeta)$ is entry-wise analytic in a neighborhood of the origin and entry-wise finite on the boundary of the neighborhood. Hence, $G_1^*(\zeta)$ is bounded around the origin and this implies that $\check{G}_1(z)$ is bounded around the point $z_1$. As a result, we see that the point $z_1$ is a removable singularity for any entry of $\check{G}_1(z)$.

Next, consider the case where $z_1=\bar{z}_1^*$. 
Let $\mathbb{U}_{\bar{z}_1^*}$ be a neighborhood of $\bar{z}_1^*$ satisfying $(\mathbb{U}_{\bar{z}_1^*}\setminus\{\bar{z}_1^*\})\cap\calE_1=\emptyset$. We set $\mathbb{U}_{\bar{z}_1^*}$ being so small that there exists a positive number $\varepsilon$ such that for every $z\in\mathbb{U}_{\bar{z}_1^*}$, $|\alpha_k(z)|<\underline{\zeta}_2(\bar{z}_1^*)-\varepsilon$ for $k\in\{1,2,...,s_0-1\}$ and $|\alpha_k(z)|\ge\underline{\zeta}_2(\bar{z}_1^*)-\varepsilon$ for $k\in\{s_0+2,s_0+3,...,m\}$. It is possible by Lemma \ref{le:zeros_and_G1eigenvalues}. 
The point $\bar{z}_1^*$ is a branch point of $\alpha_{s_0}(z)$ with order one and $\alpha_{s_0}(z)$ coalesces with $\alpha_{s_0+1}(z)$ at $z=\bar{z}_1^*$. 
Therefore, we focus on two sets of solutions to $\phi(z,w)=0$: $\{\alpha_1(z), ..., \alpha_{s_0-1}(z),\alpha_{s_0}(z)\}$ and $\{\alpha_1(z), ..., \alpha_{s_0-1}(z),\alpha_{s_0+1}(z)\}$. 
Further, let the neighborhood $\mathbb{U}_{\bar{z}_1^*}$ be so small that $\det V_1(z)\ne 0$ for every $z\in\mathbb{U}_{\bar{z}_1^*}\setminus\{\bar{z}_1^*\}$. It is possible because of the same reason as that used in the case of $z_1\ne\bar{z}_1^*$. 
Note that, in this case, $V_1(z)$ may be generated from the set $\{\alpha_1(z), ..., \alpha_{s_0-1}(z),\alpha_{s_0+1}(z)\}$. 
Let $z_0$ be a point in $\mathbb{U}_{\bar{z}_1^*}\cap\mathbb{C}(\underline{z}_1^*,\bar{z}_1^*]$ satisfying the condition of Assumption \ref{as:eigenG1_distinct} and $z_t$ an arbitrary point in $\mathbb{U}_{\bar{z}_1^*
}\setminus\{\bar{z}_1^*\}$.
Consider an arbitrary path $\xi_{z_0,z_t}$ on $\mathbb{U}_{\bar{z}_1^*}\setminus\{\bar{z}_1^*\}$ and analytically continue $\check{\alpha}_k(z_0)$ ($=\alpha_k(z_0)$), $k=1,2,...,s_0$, along the path $\xi_{z_0,z_t}$. 
Then, for any $z\in\xi_{z_0,z_t}$, the set $\{\check{\alpha}_1(z),\check{\alpha}_2(z),...,\check{\alpha}_{s_0}(z)\}$ is given by $\{\alpha_1(z), ..., \alpha_{s_0-1}(z), \alpha_{s_0}(z)\}$ or $\{\alpha_1(z), ..., \alpha_{s_0-1}(z), \alpha_{s_0+1}(z)\}$. 
Consider $\check{G}_1(z)$ define by formula (\ref{eq:tildeG1_definition}). We have $\det \check{V}_1(z)\ne 0$ for every $z\in\xi_{z_0,z_t}$ and $\check{G}_1(z)$ is entry-wise analytic on $\xi_{z_0,z_t}$. Because of the same reason as that used in the case of $z_1\ne\bar{z}_1^*$, $\check{G}_1(z)$ is bounded around $\bar{z}_1^*$. 
For any $z\in\mathbb{U}_{\bar{z}_1^*}\setminus\{\bar{z}_1^*\}$, if $\{\check{\alpha}_1(z),\check{\alpha}_2(z),...,\check{\alpha}_{s_0}(z)\}$ is given by $\{\alpha_1(z), ..., \alpha_{s_0-1}(z), \alpha_{s_0}(z)\}$, $\check{G}_1(z)$ is given by $G_1(z)$; if $\{\check{\alpha}_1(z),\check{\alpha}_2(z),...,\check{\alpha}_{s_0}(z)\}$ is given by $\{\alpha_1(z), ..., \alpha_{s_0-1}(z), \alpha_{s_0+1}(z)\}$, it is given by $G_1^{co}(z)$. 
Hence, every entry of $\check{G}_1(z)$ is a two valued function in $\mathbb{U}_{\bar{z}_1^*}\setminus\{\bar{z}_1^*\}$ and the point $\bar{z}_1^*$ is a branch point of $\check{G}_1(z)$ with order one. 
\end{proof}

\begin{remark} \label{re:alphas0_analytic}
From the proof of Lemma \ref{le:G1_singularity}, we see that, in a neighborhood of $\bar{z}_1^*$, the eigenvalue $\alpha_{s_0}(z)$ of $G_1(z)$ and the corresponding eigenvector $\bv_{s_0}(z)$ are given as
\begin{align}
\alpha_{s_0}(z) = \tilde{\alpha}_{s_0}\big((\bar{z}_1^*-z)^{\frac{1}{2}}\big),\quad 
\bv_{s_0}(z) = \tilde{\bv}_{s_0}\big((\bar{z}_1^*-z)^{\frac{1}{2}}\big),
\end{align}
where $\tilde{\alpha}_{s_0}(\zeta)$ and $\tilde{\bv}_{s_0}(\zeta)$ are analytic in a neighborhood of the origin.  We will use this point in the following section. 
\end{remark}

%
%
\section{Asymptotics} \label{sec:asymptotics}

%
\subsection{Way to obtain the asymptotic formulae}

We introduce the following notation. For $r>0$, $\varepsilon>0$ and $\theta\in[0,\pi/2)$, define 
\[
\tilde{\Delta}_r(\varepsilon,\theta) = \{z\in\mathbb{C}: |z|<r+\varepsilon,\ z\ne r,\ |\arg(z-r)|>\theta \}. 
\]
For $r>0$, we denote by ``$\tilde{\Delta}_r\ni z\to r$" that $\tilde{\Delta}_r(\varepsilon,\theta)\ni z\to r$ for some $\varepsilon>0$ and $\theta\in [0,\pi/2)$. 
We use the following lemma for obtaining the directional exact asymptotic formulae of the stationary distribution of the 2D-QBD process, described in Theorem \ref{th:mainresults}. 
\begin{lemma}[Theorem VI.4 of Flajolet and Sedgewick \cite{Flajolet09}] \label{le:asympto_formula_ss}
Let $f$ be a generating function of a sequence of real numbers $\{a_n,\, n\in\mathbb{Z}_+\}$, i.e., $f(z)=\sum_{n=0}^\infty a_n z^n$. If $f(z)$ is singular at $z=z_0>0$ and analytic on the set $\tilde{\Delta}_{z_0}(\varepsilon,\theta)$ for some $\varepsilon>0$ and $\theta\in[0,\pi/2)$ and if it satisfies 
\begin{align}
\lim_{\tilde{\Delta}_{z_0}\ni z\to z_0} (z_0-z)^\alpha f(z) = c_0
\label{eq:lim_fz_alpha}
\end{align}
for $\alpha\in\mathbb{R}\setminus\{0,-1,-2,...\}$ and some nonzero constant $c_0\in\mathbb{R}$, then 
\begin{align}
\lim_{n\to\infty} \left( \frac{n^{\alpha-1}}{\Gamma(\alpha)} z_0^{-n} \right)^{-1} a_n = c
\label{eq:asymp_an}
\end{align}
for some real number $c$, where $\Gamma(z)$ is the gamma function. This means that the exact asymptotic formula of the sequence $\{a_n\}$ is given by $n^{\alpha-1} z_0^{-n}$. 
\end{lemma}

\begin{remark} \label{le:asympto_formula_ss_sub}
In Lemma \ref{le:asympto_formula_ss}, if the value of $\alpha$ satisfying equation (\ref{eq:lim_fz_alpha}) is 0, then replace $f(z)$ in equation (\ref{eq:lim_fz_alpha}) with $f_1(z)=(f(z)-f(z_0))$ and seek $\alpha$ satisfying equation (\ref{eq:lim_fz_alpha}). 
Furthermore, if the obtained value of $\alpha$ is a negative integer, say $\alpha=k\in\{-1,-2, ...\}$, then replace $f_1(z)$ with $f_2(z)=(f_1(z)-\hat{f}_1(z_0))$ and seek $\alpha$ satisfying equation (\ref{eq:lim_fz_alpha}) again, where $\hat{f}_1(z_0)=\lim_{\tilde{\Delta}_{z_0}\ni z\to z_0} (z_0-z)^k f_1(z)$. 
Repeat this procedure until the value of $\alpha$ not being a negative integer is obtained. Then, the obtained value of $\alpha$ satisfies equation (\ref{eq:asymp_an}). 
For example, if $z_0$ is a branch point of $f(z)$ with order one, the Puiseux series expansion for $f(z)$ around $z=z_0$ is represented as $f(z) = \sum_{n=0}^\infty \tilde{a}_n\, (z_0-z)^{\frac{n}{2}}$ and $\alpha$ is given by 
$\alpha = -\mbox{$\frac{1}{2}$}\min\{k\in\mathbb{Z}_+: \mbox{$\tilde{a}_k\ne 0$ and $k$ is an odd number}\}$. 
\end{remark}

In order to apply Lemma \ref{le:asympto_formula_ss} to the vector generating function $\bvarphi_1(z)$ (resp.\ $\bvarphi_2(z)$), we analytically extend $\bvarphi_1(z)$ (resp.\ $\bvarphi_2(z)$) over $\tilde{\Delta}_{r_1}(\varepsilon,\theta)$ (resp.\ $\tilde{\Delta}_{r_2}(\varepsilon,\theta)$) for some $\varepsilon>0$ and $\theta\in[0,\pi/2)$ in Lemma \ref{le:varphi1_extended} of Subsection \ref{sec:extend_varphi1and2}. 
It is also clarified that the point $z=r_1$ (resp.\ $z=r_2$) is the unique singularity of $\bvarphi_1(z)$ (resp.\ $\bvarphi_2(z)$) on the circle $\partial\Delta(0,r_1)$ (resp.\ $\partial\Delta(0,r_1)$). 
For the case of Type I ($\psi_1(\bar{z}_1^*)>1$) and that of Type III, we reveal that the point $z=r_1$ is a pole of $\bvarphi_1(z)$  with order one and give its Laurent series expansion in Lemma \ref{le:varphi1_limit_typeIa} of Subsection \ref{sec:expansion_typeIa}. The corresponding result for $\bvarphi_2(z)$ is stated in Corollary \ref{co:varphi2_limit_typeIa}. 
In Lemma \ref{le:varphi1_limit_typeIbc} of Subsection \ref{sec:expansion_varpshi_typeI_II_III}, we give the Puiseux series expansions for $\bvarphi_1(z)$ in the case of Type I. In that case, if $\psi_1(\bar{z}_1^*)= 1$, then the point $z=r_1$ is a pole of $\bvarphi_1(z)$ with order one and also a branch point with order one; if $\psi_1(\bar{z}_1^*)< 1$, then it is just a branch point of $\bvarphi_1(z)$ with order one. 
In the case of Type I, if $\psi_2(\bar{z}_1^*)\le 1$, then analogous results hold for $\bvarphi_2(z)$. 
In Lemma \ref{le:varphi1_limit_typeII} of Subsection \ref{sec:expansion_varpshi_typeI_II_III}, we give the Laurent and Puiseux series expansions of $\bvarphi_1(z)$ for the case of Type II. 
In that case, if $\eta_2^{(c)}<\theta_2^{(c)}$, then the point $z=r_1$ is a pole with order one; if $\eta_2^{(c)}=\theta_2^{(c)}$ and $\psi_1(\bar{z}_1^*)>1$, then it is a pole with order two; if $\eta_2^{(c)}=\theta_2^{(c)}$ and $\psi_1(\bar{z}_1^*)=1$, then it is a pole with order two and also a branch point with order one; otherwise ($\eta_2^{(c)}=\theta_2^{(c)}$ and $\psi_1(\bar{z}_1^*)<1$), it is a pole with order one and also a branch point with order one. 
With respect to $\bvarphi_2(z)$ in the case of Type III, results analogous to those for $\bvarphi_1(z)$ in the case of Type II hold. 

As a result, from Lemmas \ref{le:asympto_formula_ss} through \ref{le:varphi1_limit_typeII} and Remark \ref{le:asympto_formula_ss_sub}, the exact asymptotic formulae in Theorem \ref{th:mainresults} (main theorem) are automatically obtained.

%
\subsection{Analytic extension of $\bvarphi_1(z)$ and $\bvarphi_2(z)$} \label{sec:extend_varphi1and2}

Recall that the vector generating function $\bvarphi_1(z)$ is entry-wise analytic in the open disk $\Delta_{r_1}=\Delta(0,r_1)$, where $r_1$ is the radius of convergence of $\bvarphi_1(z)$. We analytically extend it over $\tilde{\Delta}_{r_1}(\varepsilon,\theta)$ for some $\varepsilon>0$ and $\theta\in[0,\pi/2)$, by using equation (\ref{eq:varphi1}) of Lemma \ref{le:varphi1}.  
$\bvarphi_2(z)$ is also extended by using equation (\ref{eq:varphi2}). The following lemma states those points. 

\begin{lemma} \label{le:varphi1_extended}
$\bvarphi_1(z)$ (resp.\ $\bvarphi_2(z)$) can analytically be extended over $\tilde{\Delta}_{r_1}(\varepsilon,\theta)$ (resp.\ $\tilde{\Delta}_{r_2}(\varepsilon,\theta)$) for some $\varepsilon>0$ and $\theta\in[0,\pi/2)$. The extended $\bvarphi_1(z)$ (resp.\ $\bvarphi_2(z)$) is entry-wise analytic on $\partial\Delta_{r_1}\setminus\{r_1\}$ (resp.\ $\partial\Delta_{r_2}\setminus\{r_2\}$).
\end{lemma}

Since discussion for $\bvarphi_2(z)$ proceeds in parallel to that for $\bvarphi_1(z)$, we prove the lemma only for $\bvarphi_1(z)$. Before doing it, we present the following propositions. Their proofs are given in Appendix \ref{sec:varphis2_C1_proofs}. 

\begin{proposition} \label{pr:varphi2_analytic}
If $G_1(z)$ is entry-wise analytic and satisfies $\spr(G_1(z))<r_2$ in a region $\Omega\subset\mathbb{C}$, then $\bvarphi_2(G_1(z))$ and $\bvarphi_2^{\hat{C}_2}(z,G_1(z))$ are also entry-wise analytic in the same region $\Omega$. 
\end{proposition}

\begin{proposition} \label{pr:C1_inequality}
Under Assumption \ref{as:Akl_irreducible}, for $z\in\mathbb{C}(\underline{z}_1^*,\bar{z}_1^*]$ such that $z\ne |z|$, we have 
\begin{equation}
\spr(C_1(z,G_1(z)))<\spr(C_1(|z|,G_1(|z|))). 
\end{equation}
\end{proposition}

\begin{proof}[Proof of Lemma \ref{le:varphi1_extended}]
Since $\bvarphi_1(z)$ is entry-wise analytic in $\Delta_{r_1}$, it suffices by the identity theorem to show that the right hand side of equation (\ref{eq:varphi1}) is entry-wise analytic on $\partial\Delta_{r_1}\setminus\{r_1\}$.
The right hand side of equation (\ref{eq:varphi1}) is composed of four terms: $\bnu_{0,0}(C_0(z,G_1(z))-I)$, $\bvarphi_2(G_1(z))$, $\bvarphi_2^{\hat{C}_2}(z,G_1(z))$ and $(I-C_1(z,G_1(z)))^{-1}$. By Lemmas \ref{le:G_analytic1} and \ref{le:G1_singularity}, $G_1(z)$ is entry-wise analytic in $\mathbb{C}(\underline{z}_1^*,\bar{z}_1^*)$ and on $\partial\Delta_{\bar{z}_1^*}\setminus\{\bar{z}_1^*\}$. 
Hence, if $r_1<\bar{z}_1^*$, $\bnu_{0,0}(C_0(z,G_1(z))-I)$ is entry-wise analytic on $\partial\Delta_{r_1}$; if $r_1=\bar{z}_1^*$, it is entry-wise analytic on $\partial\Delta\bar{z}_1^*\setminus\{\bar{z}_1^*\}$. 
Analytic properties of the other terms are given as follows.

{\it In the case of Type I ($\psi_1(\bar{z}_1^*)>1$) and that of Type III}.\quad 
Since $r_1=e^{\theta_1^{(c)}}<\bar{z}_1^*$, $G_1(z)$ is entry-wise analytic on $\partial\Delta_{r_1}$ and by Proposition \ref{pr:sprG1} and Lemma \ref{le:Gmatrix}, for $z\in\partial\Delta_{r_1}$, 
\[
\spr(G_1(z))\le\spr(G_1(|z|))=\underline{\zeta}_2(|z|)<\eta_2^{(c)}=r_2.  
\]
Hence, by Proposition \ref{pr:varphi2_analytic}, both $\bvarphi_2(G_1(z))$ and $\bvarphi_2^{\hat{C}_2}(z,G_1(z))$ are entry-wise analytic on $\partial\Delta_{r_1}$. 
By Proposition \ref{pr:C1_inequality}, for $z\in\partial\Delta_{r_1}\setminus\{r_1\}$, 
\begin{align}
&\spr(C_1(z,G_1(z)))\le\spr(|C_1(z,G_1(z))|)<\spr(C_1(r_1,G_1(r_1))=1, 
\end{align}
and we know that $\det(I-C_1(z,G_1(z)))\ne 0$ at any point on $\partial\Delta_{r_1}\setminus\{r_1\}$. Since we have 
\[
(I-C_1(z,G_1(z)))^{-1} = \frac{\adj\ (I-C_1(z,G_1(z)))}{\det (I-C_1(z,G_1(z)))}, 
\]
where $C_1(z,G_1(z))$ is entry-wise analytic on $\partial\Delta_{r_1}$, this implies that each entry of $(I-C_1(z,G_1(z)))^{-1}$ is analytic on $\partial\Delta_{r_1}\setminus\{r_1\}$. 

{\it In the case of Type I ($\psi_1(\bar{z}_1^*)\le 1$)}.\quad 
In this case, we have $r_1=\bar{z}_1^*$ and $G_1(z)$ is entry-wise analytic on $\partial\Delta_{\bar{z}_1^*}\setminus\{\bar{z}_1^*\}$. 
In a manner similar to that used in the case of Type I ($\psi_1(\bar{z}_1^*)>1$), we obtain $\spr(G_1(z))<r_2$ for every $z\in\partial\Delta_{\bar{z}_1^*}$, and by Proposition \ref{pr:varphi2_analytic}, $\bvarphi_2(G_1(z))$ and $\bvarphi_2^{\hat{C}_2}(z,G_1(z))$ are entry-wise analytic on $\partial\Delta_{\bar{z}_1^*}\setminus\{\bar{z}_1^*\}$. 
Furthermore, $(I-C_1(z,G_1(z)))^{-1}$ is also entry-wise analytic on $\partial\Delta_{\bar{z}_1^*}\setminus\{\bar{z}_1^*\}$.

{\it In the case of Type II ($\eta_2^{(c)}<\theta_2^{(c)}$)}.\quad 
We have $r_1=e^{\bar{\eta}_1^{(c)}}<e^{\theta_1^{(c)}}\le \bar{z}_1^*$ and $G_1(z)$ is entry-wise analytic on $\partial\Delta_{r_1}$. 
By Lemma \ref{le:zeros_and_G1eigenvalues}, if $z\in\partial\Delta_{r_1}\setminus\{r_1\}$, the modulus of every eigenvalue of $G_1(z)$ is less than $\underline{\zeta}_2(|z|)$ ($=\underline{\zeta}_2(r_1)$). 
Hence, we have, for $z\in\partial\Delta_{r_1}\setminus\{r_1\}$,  
\[
\spr(G_1(z)) < \underline{\zeta}_2(e^{\bar{\eta}_1^{(c)}}) = e^{\eta_2^{(c)}} = r_2, 
\]
and by Proposition \ref{pr:varphi2_analytic}, $\bvarphi_2(G_1(z))$ and $\bvarphi_2^{\hat{C}_2}(z,G_1(z))$ are entry-wise analytic on $z\in\partial\Delta_{r_1}\setminus\{r_1\}$.
We know that $\psi_1(z)=\spr(C_1(z,G_1(z)))$ is convex in $z\in\mathbb{R}_+\setminus\{0\}$ and $\psi_1(1)=\psi_1(e^{\theta_1^{(c)}})=1$. Therefore, we have, for $z\in\partial\Delta_{r_1}$,  
\[
\spr(C_1(z,G_1(z)))\le \spr(C_1(r_1,G_1(r_1)))<\spr(C_1(e^{\theta_1^{(c)}},G_1(e^{\theta_1^{(c)}})))=1, 
\]
and $(I-C_1(z,G_1(z)))^{-1}$ is entry-wise analytic on $\partial\Delta_{r_1}$. 

{\it In the case of Type II ($\eta_2^{(c)}=\theta_2^{(c)}$, $\psi_1(\bar{z}_1^*)>1$)}.\quad 
In a manner similar to that used in the case of Type II ($\eta_2^{(c)}<\theta_2^{(c)}$), we see that $G_1(z)$ is entry-wise analytic on $\partial\Delta_{r_1}$ and that $\bvarphi_2(G_1(z))$ and $\bvarphi_2^{\hat{C}_2}(z,G_1(z))$ are entry-wise analytic on $\partial\Delta_{r_1}\setminus\{r_1\}$. 
Furthermore, in a manner similar to that used in the case of Type I ($\psi_1(\bar{z}_1^*)>1$), we see that $(I-C_1(z,G_1(z)))^{-1}$ is entry-wise analytic on $\partial\Delta_{r_1}\setminus\{r_1\}$.

{\it In the case of Type II ($\eta_2^{(c)}=\theta_2^{(c)}$, $\psi_1(\bar{z}_1^*)\le 1$)}.\quad 
In this case, we have $r_1=e^{\bar{\eta}_1^{(c)}}=e^{\theta_1^{(c)}}=\bar{z}_1^*$. Hence, $G_1(z)$ is entry-wise analytic on $\partial\Delta_{\bar{z}_1^*}\setminus\{\bar{z}_1^*\}$. 
In a manner similar to that used in the case of Type II ($\eta_2^{(c)}<\theta_2^{(c)}$), we see that $\bvarphi_2(G_1(z))$ and $\bvarphi_2^{\hat{C}_2}(z,G_1(z))$ are entry-wise analytic on $\partial\Delta_{\bar{z}_1^*}\setminus\{\bar{z}_1^*\}$.  
Furthermore, we see that $(I-C_1(z,G_1(z)))^{-1}$ is entry-wise analytic on $\partial\Delta_{\bar{z}_1^*}\setminus\{\bar{z}_1^*\}$. 
\end{proof}

\bigskip
As mentioned before, since discussion for $\bvarphi_2(z)$ proceeds in parallel to that for $\bvarphi_1(z)$, we explain only about $\bvarphi_1(z)$ in the following subsection.

%
\subsection{Singularity of $\bvarphi_1(z)$: Type I ($\psi_1(\bar{z}_1^*)>1$) and Type III} \label{sec:expansion_typeIa}

By formula (\ref{eq:varphi1}), $\bvarphi_1(z)$ is represented as 
\begin{equation}
\bvarphi_1(z) 
= \bg_1(z)\,(I-C_1(z,G_1(z)))^{-1}
= \frac{\bg_1(z)\,\adj(I-C_1(z,G_1(z)))}{f_1(1,z)}, 
\label{eq:varphi1_gf1}
\end{equation}
where
\begin{align*}
&\bg_1(z) = \bvarphi_2^{\hat{C}_2}(z,G_1(z)) - \bvarphi_2(G_1(z)) + \bnu_{0,0} l(C_0(z,G_1(z))-Ir), \\
&f_1(\lambda,z) = \det(\lambda I-C_1(z,G_1(z))). 
\end{align*}
In the case of Type I ($\psi_1(\bar{z}_1^*)>1$), we have $r_1<\bar{z}_1^*$ and $G_1(z)$ is entry-wise analytic at $z=r_1$. By Proposition \ref{pr:varphi2_analytic}, since $\spr(G_1(r_1))=\underline{\zeta}_2(r_1)<r_2$, both $\bvarphi_2(G_1(z))$ and $\bvarphi_2^{\hat{C}_2}(z,,G_1(z))$ are also entry-wise analytic at $z=r_1$. 
Therefore, $\bg_1(z)$, $\adj(I-C_1(z,G_1(z)))$ and $f_1(1,z)$ are analytic in a neighborhood of $z=r_1$. Furthermore, we know $f_1(1,r_1)=0$, and each entry of $\bvarphi_1(z)$ is a meromorphic function in a neighborhood of $z=r_1$. 
For $f_1(1,z)$, we give the following proposition. 
\begin{proposition} \label{pr:f1_limit}
For $z_0\in(\underline{z}_1^*,\bar{z}_1^*)$ such that $\psi_1(z_0)=1$, the point $z=z_0$ is a zero of $f_1(1,z)$ with multiplicity one and we have
\begin{align}
\lim_{z\to z_0} (z_0-z)^{-1} f_1(1,z) = \psi_{1,z}(z_0) f_{1,\lambda}(1,z_0) \ne 0, 
\label{eq:f1_limit}
\end{align}
where $\psi_{1,z}(z)=(d/dz) \psi_1(z)$ and $f_{1,\lambda}(\lambda,z)=(\partial/\partial\lambda) f_1(\lambda,z)$. 
\end{proposition}
\begin{proof}
$f_1(\lambda,z)$ is the characteristic polynomial of $C_1(z,G_1(z))$. 
Under Assumption \ref{as:Akl_irreducible}, $C_1(1,G_1(1))$ is irreducible and hence, for $z\in[\underline{z}_1^*,\bar{z}_1^*]$, $C_1(z,G_1(z))$ is also irreducible. This implies that for $z\in[\underline{z}_1^*,\bar{z}_1^*]$, the Perron-Frobenius eigenvalue of $C_1(z,G_1(z))$, given by $\psi_1(z)=\spr(C_1(z,G_1(z))$, is simple, i.e., the algebraic multiplicity of $\psi_1(z)$ is one. 
Denote by $\lambda^{C_1}(z)$ the eigenvalue of $C_1(z,G_1(z))$ corresponding to $\psi_1(z)$ when $z\in[\underline{z}_1^*,\bar{z}_1^*]$. Then, $\lambda^{C_1}(z)$ is analytic on $(\underline{z}_1^*,\bar{z}_1^*)$, and since $\lambda^{C_1}(z_0)=\psi_1(z_0)=1$, we have
\begin{align}
\lim_{z\to z_0} (1-\lambda^{C_1}(z))^{-1} f_1(1,z) = f_{C_1,\lambda}(1,z_0) \ne 0. 
\end{align}
Combining this with the fact that $\lim_{z\to z_0} (1-\lambda^{C_1}(z))/(z_0-z) = \psi_{1,z}(z)$, we obtain equation (\ref{eq:f1_limit}). Since $\psi_1(1)\le 1$, $\psi_1(z_0)= 1$, $\psi_1(z)\not\equiv 1$ and $\psi_1(e^s)$ is convex in $s\in[\log\underline{z}_1^*,\log\bar{z}_1^*]$, we have $\psi_{1,z}(z_0)> 0$, and this completes the proof.
\end{proof}

Denote by $\bu^{C_1}(z)$ and $\bv^{C_1}(z)$ the left and right eigenvectors of $C_1(z,G_1(z))$ with respect to the eigenvalue $\lambda^{C_1}(z)$, satisfying $\bu^{C_1}(z) \bv^{C_1}(z)=1$. Then, the following proposition holds. Its proof is given in Appendix \ref{sec:adjC1}. 
\begin{proposition} \label{pr:adjC1}
For $z_0\in(\underline{z}_1^*,\bar{z}_1^*)$ such that $\psi_1(z_0)=1$, 
\begin{align}
&\adj\big(I-C_1(z_0,G_1(z_0))\big) 
= f_{1,\lambda}(1,z_0) \bv^{C_1}(z_0) \bu^{C_1}(z_0), 
\label{eq:adjC1}
\end{align}
where $\bv^{C_1}(z_0)$ and $\bu^{C_1}(z_0)$ are positive. 
\end{proposition}

Since $\bvarphi_1(z)$ is entry-wise meromorphic in a neighborhood of $z=r_1$, we obtain the following lemma. 
\begin{lemma}[Type I ($\psi_1(\bar{z}_1^*)>1$) and Type III] \label{le:varphi1_limit_typeIa}
In the case of Type I, if $\psi_1(\bar{z}_1^*)>1$, the pint $z=r_1$ is a pole of $\bvarphi_1(z)$ with order one and its Laurent series expansion is represented as 
\begin{align}
\bvarphi_1(z) = \sum_{k=-1}^\infty \bvarphi_{1,k}^I (r_1-z)^k, 
\label{eq:varphi1_expansion1}
\end{align}
where $\bvarphi_{1,-1}^I = c_{pole}^{\bvarphi_1}\, \bu^{C_1}(r_1)> \bzero^\top$ and $c_{pole}^{\bvarphi_1} = \psi_{1,z}(r_1)^{-1} \bg_1(r_1) \bv^{C_1}(r_1) >0$. 
In the case of Type III, the same result also holds for $\bvarphi_1(z)$. 
\end{lemma}
\begin{proof}
By Proposition \ref{pr:f1_limit}, in both the cases, the point $z=r_1$ is probably a pole of $\bvarphi_1(z)$ with order one and its Laurent series expansion is given by formula (\ref{eq:varphi1_expansion1}). If we have
\[
\bvarphi_{1,-1}^I = \lim_{z\to r_1} (r_1-z) \bvarphi_1(z) = \bzero^\top, 
\]
then $z=r_1$ is a removable singularity and $\bvarphi_1(z)$ is analytic at $z=r_1$. This contradicts that $\bvarphi_1(z)$ is singular at $z=r_1$. Therefore, $\bvarphi_{1,-1}^I \ne \bzero^\top$ and $z=r_1$ is a pole of $\bvarphi_1(z)$ with order one. 
By Propositions \ref{pr:f1_limit} and \ref{pr:adjC1}, we obtain the expression of $\bvarphi_{1,-1}^I$. Note that $\bvarphi_{1,-1}^I$ is nonzero and nonnegative and $\bu^{C_1}(r_1)$ is positive. This implies that $c_{pole}^{\bvarphi_1}$ is positive. 
\end{proof}

Since the corresponding result for $\bvarphi_2(z)$ will be used later, we state it as the following corollary. 
%
\begin{corollary}[Type I ($\psi_2(\bar{z}_2^*)>1$) and Type II] \label{co:varphi2_limit_typeIa}
In the case of Type I, if $\psi_1(\bar{z}_2^*)>1$, then the point $z=r_2$ is a pole of $\bvarphi_2(z)$ with order one and its Laurent series expansion is represented as 
\begin{equation}
\bvarphi_2(z) = \sum_{k=-1}^\infty \bvarphi_{2,k}^I (r_2-z)^k, 
\end{equation}
where $\bvarphi_{2,-1}^I = c_{pole}^{\bvarphi_2}\, \bu^{C_2}(r_2)> \bzero^\top$ and $c_{pole}^{\bvarphi_2} = \psi_{2,z}(r_2)^{-1} \bg_2(r_2) \bv^{C_2}(r_2) >0$; %
$\bg_2(z)$ is given as $\bg_2(z) = \bvarphi_1^{\hat{C}_1}(G_2(z),z) - \bvarphi_1(G_2(z)) + \bnu_{0,0} \bigl(C_0(G_2(z),z)-I\bigr)$, $\bu^{C_2}(z)$ and $\bv^{C_2}(z)$ are the left and right eigenvectors of $C_2(G_2(z),z)$ with respect to the eigenvalue $\psi_2(z)$, and  $\psi_{2,z}(z)=(d/dz)\psi_2(z)$. 
In the case of Type II, the same result also holds for $\bvarphi_2(z)$. 
\end{corollary}

\bigskip 
In Subsection \ref{sec:expansion_varpshi_typeI_II_III}, we consider a series expansion for the vector function $\bvarphi_1(z)$ around the point $z=\bar{z}_1^*$. To this end, we obtain the Puiseux series expansions for $G_1(z)$, $\bvarphi_2(G_1(z))$ and $\bvarphi_2^{\hat{C}_2}(z,G_1(z))$ around $z=\bar{z}_1^*$, in the following two subsections.

%
\subsection{Puiseux series expansion for $G_1(z)$}

In Assumption \ref{as:G1_eigen_z1max}, we assumed that all the eigenvalues of $G_1(r_1)$ are distinct.
This assumption seems to be rather strong, but under it, if $r_1<\bar{z}_1^*$, all the eigenvalues of $G_1(z)$ are analytic at $z=r_1$; if $r_1=\bar{z}_1^*$, those of $G_1(z)$ except for one corresponding to $\underline{\zeta}_2(z)$ are analytic at $z=\bar{z}_1^*$. Furthermore, all the eigenvectors of $G_1(z)$ are linearly independent in a neighborhood of $z=r_1$. 

%
In this subsection, we obtain the Puiseux series expansion for $G_1(z)$ around $z=\bar{z}_1^*$. Let $\alpha_j(z),\,j=1,2,...,s_0,$ be the eigenvalues of $G_1(z)$, counting multiplicity, and without loss of generality, let $\alpha_{s_0}(z)$ be the eigenvalue corresponding to $\underline{\zeta}_2(z)$ when $z\in[\underline{z}_1^*,\bar{z}_1^*]$. For $j\in\{1,2,...,s_0\}$, let $\bv_j(z)$ be the right eigenvector of $G_1(z)$ with respect to the eigenvalue $\alpha_j(z)$. 
Define matrices $J_1(z)$ and $V_1(z)$ as $J_1(z) = \diag(\alpha_1(z),\alpha_2(z),...,\alpha_{s_0}(z))$ and $V_1(z)=\begin{pmatrix} \bv_1(z) & \bv_2(z) & \cdots & \bv_{s_0}(z) \end{pmatrix}$, respectively. As mentioned above, under Assumption \ref{as:G1_eigen_z1max}, $V_1(z)$ is nonsingular in a neighborhood of $z=r_1$, and $G_1(z)$ is factorized as $G_1(z)=V_1(z) J_1(z) V_1(z)^{-1}$ (Jordan decomposition).  
We represent $V_1(z)^{-1}$ as 
\[
V_1(z)^{-1} = \begin{pmatrix} \bu_1(z) \cr \bu_2(z) \cr \vdots \cr \bu_{s_0}(z) \end{pmatrix}, 
\]
where each $\bu_j(z)$ is a $1\times s_0$ vector. For $j\in\{1,2,...,s_0\}$, $\bu_j(z)$ is the left eigenvector of $G_1(z)$ with respect to the eigenvalue $\alpha_j(z)$ and satisfies $\bu_j(z) \bv_k(z)=\delta_{jk}$ for $k\in\{1,2,...,s_0\}$, where $\delta_{jk}$ is the Kronecker delta. 
Furthermore, from the results of Section \ref{sec:Gmatrix}, we see that, under Assumption \ref{as:G1_eigen_z1max}, if $r_1<\bar{z}_1^*$, $J_1(z)$ and $V_1(z)$ are entry-wise analytic in a neighborhood of $z=r_1$; if $r_1=\bar{z}_1^*$, $J_1(z)$ and $V_1(z)$ except for $\alpha_{s_0}(z)$ and $\bv_{s_0}(z)$ are also entry-wise analytic in a neighborhood of $z=\bar{z}_1^*$, where $\alpha_{s_0}(z)$ and $\bv_{s_0}(z)$ are given in terms of a function $\tilde{\alpha}_{s_0}(\zeta)$ and vector function $\tilde{\bv}_{s_0}(\zeta)$ being analytic in a neighborhood of the origin as $\alpha_{s_0}(z)=\tilde{\alpha}_{s_0}((\bar{z}_1^*-z)^{\frac{1}{2}})$ and $\bv_{s_0}(z)=\tilde{\bv}_{s_0}((\bar{z}_1^*-z)^{\frac{1}{2}})$ (see Remark \ref{re:alphas0_analytic}). 
Here, we give another representation of $G_1(z)$ around $z=\bar{z}_1^*$. Define $\tilde{J}_1(\zeta)$ and $\tilde{V}_1(\zeta)$ as 
\begin{align*}
&\tilde{J}_1(\zeta) = \diag\big(\alpha_1(\bar{z}_1^*-\zeta^2),...,\alpha_{s_0-1}(\bar{z}_1^*-\zeta^2),\tilde{\alpha}_{s_0}(\zeta) \big), \\
&\tilde{V}_1(\zeta) = \begin{pmatrix} \bv_1(\bar{z}_1^*-\zeta^2) & \cdots & \bv_{s_0-1}(\bar{z}_1^*-\zeta^2) & \tilde{\bv}_{s_0}(\zeta) \end{pmatrix}.
\end{align*}
Under Assumption \ref{as:G1_eigen_z1max}, $\tilde{J}_1(\zeta)$ and $\tilde{V}_1(\zeta)$ are entry-wise analytic in a neighborhood of the origin and $\tilde{V}_1(\zeta)$ is nonsingular. Therefore, $\tilde{G}_1(\zeta)$ defined as $\tilde{G}_1(\zeta) = \tilde{V}_1(\zeta) \tilde{J}_1(\zeta) \tilde{V}_1(\zeta)^{-1}$ is entry-wise analytic in a neighborhood of the origin and satisfies 
\begin{equation}
G_1(z) = \tilde{G}_1((\bar{z}_1^*-z)^{\frac{1}{2}}). 
\label{eq:tildeG1}
\end{equation}
Since $\tilde{\alpha}_{s_0}(\zeta)$ and $\tilde{G}_1(\zeta)$ are analytic in a neighborhood of the origin, their Taylor expansions are represented as 
\begin{align}
\tilde{\alpha}_{s_0}(\zeta) = \sum_{k=0}^\infty \tilde{\alpha}_{s_0,k}\, \zeta^k, \quad 
\tilde{G}_1(\zeta) = \sum_{k=0}^\infty \tilde{G}_{1,k}\, \zeta^k, 
\end{align}
where $\tilde{\alpha}_{s_0,0}=\underline{\zeta}_2(\bar{z}_1^*)$ and $\tilde{G}_{1,0}=G_1(\bar{z}_1^*)$. The Puiseux series expansion for $G_1(z)$ around $z=\bar{z}_1^*$ is represented as 
\[
G_1(z) = \sum_{k=0}^\infty \tilde{G}_{1,k}\, (\bar{z}_1^*-z)^{\frac{k}{2}}. 
\]

Considering relation between $\bar{\zeta}_1(w)$ and $\underline{\zeta}_2(z)$, we obtain the following result in a manner similar to that used in the proof of Lemma 10 of Kobayashi and Miyazawa \cite{Kobayashi13}. 
\begin{proposition} \label{pr:limit_eigenG1}
We have 
\begin{align}
\lim_{\tilde{\Delta}_{\bar{z}_1^*}\ni z\to \bar{z}_1^*} \frac{\alpha_{s_0}(\bar{z}_1^*)-\alpha_{s_0}(z)}{(\bar{z}_1^*-z)^{\frac{1}{2}}} 
&= -\tilde{\alpha}_{s_0,1}
= \frac{\sqrt{2}}{\sqrt{-\bar{\zeta}_{1,w^2}(\underline{\zeta}_2(\bar{z}_1^*))}} > 0, 
\label{eq:limit_eigen_G1}
\end{align}
where $\bar{\zeta}_{1,w^2}(w)=(d^2/d w^2)\, \bar{\zeta}_1(w)$.
\end{proposition}

For $G_1(z)$, we give the following proposition, where $\tilde{\alpha}_{s_0,1}$ and $\bu_{s_0}(z)$ are denoted by $\tilde{\alpha}_{s_0,1}^{G_1}$ and $\bu_{s_0}^{G_1}(z)$, respectively, in order to explicitly indicate that they are the coefficient and vector with respect to $G_1(z)$. The proof of the proposition is given in Appendix \ref{sec:limit_G1}.
\begin{proposition} \label{pr:limit_G1}
\begin{align}
\lim_{\tilde{\Delta}_{\bar{z}_1^*}\ni z\to \bar{z}_1^*} \frac{G_1(\bar{z}_1^*)-G_1(z)}{(\bar{z}_1^*-z)^{\frac{1}{2}}} 
&= -\tilde{G}_{1,1} 
=  -\tilde{\alpha}_{s_0,1}^{G_1} N_1(\bar{z}_1^*) \bv^{R_1}(\bar{z}_1^*) \bu_{s_0}^{G_1}(\bar{z}_1^*) \ge \bzero^\top,\ \ne\bzero^\top, 
\label{eq:limit_G1}
\end{align}
where $\bv^{R_1}(\bar{z}_1^*)$ is the right eigenvector of $R_1(\bar{z}_1^*)$ with respect to the eigenvalue $\bar{\zeta}_2(\bar{z}_1^*)^{-1}$, satisfying $\bu_{s_0}^{G_1}(\bar{z}_1^*) N_1(\bar{z}_1^*) \bv^{R_1}(\bar{z}_1^*) =1$. 
\end{proposition}

%
\subsection{Puiseux series expansions for $\bvarphi_2(G_1(z))$ and $\bvarphi_2^{\hat{C}_2}(z,,G_1(z))$}

First, we analytically extend $\bvarphi_2(G_1(z))$ and $\bvarphi_2^{\hat{C_2}}(z,G_1(z))$ around $z=r_1$ when $\spr(G_1(r_1))=\underline{\zeta}_2(r_1)=r_2<\bar{z}_2^*$. 
Recall that, for $z\in\mathbb{C}(\underline{z}_1^*,\bar{z}_1^*)$ such that $\spr(G_1(z))<r_2$, $\bvarphi_2(G_1(z))$ and $\bvarphi_2^{\hat{C}_2}(z,G_1(z))$ are entry-wise analytic (see Proposition \ref{pr:varphi2_analytic}). 
For any $z$ in such a region, $\bvarphi_2(G_1(z))$ is represented as 
\begin{align}
\bvarphi_2(G_1(z))^\top
&= \sum_{k=1}^\infty (V_1(z)^\top)^{-1} J_1(z)^k\,V_1(z)^\top \bnu_{0,k}^\top \cr
&= \sum_{k=1}^\infty \left( \bnu_{0,k}\otimes (V_1(z)^\top)^{-1} J_1(z)^k \right) \vecM(V_1(z)^\top) \cr
&= (V_1(z)^\top)^{-1} \sum_{k=1}^\infty \left( \bnu_{0,k}\otimes J_1(z)^k \right) \vecM(V_1(z)^\top), 
\label{eq:varphi2G1_extension0}
\end{align}
where we use the identity $\vecM(A B C)=(C^\top \otimes A)\, \vecM(B)$ for matrices $A$, $B$ and $C$. This equation leads us to 
\begin{align}
\bvarphi_2(G_1(z)) = \begin{pmatrix} \bvarphi_2(\alpha_1(z))\bv_1(z) & \bvarphi_2(\alpha_2(z))\bv_2(z) & \cdots & \bvarphi_2(\alpha_{s_0}(z))\bv_{s_0}(z) \end{pmatrix} V(z)^{-1}. 
\label{eq:varphi2G1_extension1}
\end{align} 
By expression (\ref{eq:varphi2G1_extension1}), we analytically extend $\bvarphi_2(G_1(z))$ around $z=r_1$. 
Assume $r_1<\bar{z}_1^*$ and $\spr(G_1(r_1))=r_2<\bar{z}_2^*$. Then, under Assumption \ref{as:G1_eigen_z1max}, for $j\in\{1,2,...,s_0\}$, $\alpha_j(z)$ and $\bv_j(z)$ as well as $V(z)^{-1}$ are analytic at $z=r_1$. Furthermore, for $j\in\{1,2,...,s_0-1\}$, $|\alpha_j(r_1)|<r_2$ and $\bvarphi_2(\alpha_j(z))$ is analytic at $z=r_1$. 
By Corollary \ref{co:varphi2_limit_typeIa}, since $\alpha_{s_0}(r_1)=r_2<\bar{z}_2^*$, $\bvarphi_2(\alpha_{s_0}(z))$ is analytic in a neighborhood of $z=r_1$ except for the point $z=r_1$, which is a pole of $\bvarphi_2(\alpha_{s_0}(z))$ with order one. Hence, the Laurent series expansion for $\bvarphi_2(G_1(z))$ around $z=r_1$ is represented as
\begin{equation}
\bvarphi_2(G_1(z)) = \sum_{k=-1}^\infty \bvarphi_{2,k}^{G_1}\, (r_1-z)^k. 
\label{eq:varphi2G1_1}
\end{equation}

Next, assume $r_1=\bar{z}_1^*$ and define $\tilde{\bvarphi}_2(\tilde{G}_1(\zeta))$ as 
\begin{align}
\tilde{\bvarphi}_2(\tilde{G}_1(\zeta)) 
&= \Bigl( \bvarphi_2(\alpha_1(\bar{z}_1^*-\zeta^2))\bv_1(\bar{z}_1^*-\zeta^2)\ \  \cdots \cr
&\qquad \bvarphi_2(\alpha_{s_0-1}(\bar{z}_1^*-\zeta^2))\bv_{s_0-1}(\bar{z}_1^*-\zeta^2)\ \ \bvarphi_2(\tilde{\alpha}_{s_0}(\zeta))\tilde{\bv}_{s_0}(\zeta) \Bigr) \tilde{V}(\zeta)^{-1}. 
\label{eq:varphi2G1_extension2}
\end{align} 
If $\spr(\tilde{G}_1(0))=\spr(G_1(\bar{z}_1^*))<r_2$, then $\tilde{\bvarphi}_2(\tilde{G}_1(\zeta))$ is entry-wise analytic at $\zeta=0$ and the Puiseux series expansion for $\bvarphi_2(G_1(z))$ around $z=\bar{z}_1^*$ is represented as
\begin{align}
\bvarphi_2(G_1(z)) 
= \tilde{\bvarphi}_2(\tilde{G}_1((\bar{z}_1^*-z)^{\frac{1}{2}})) 
= \sum_{k=0}^\infty \tilde{\bvarphi}_{2,k}^{\tilde{G}_1}\, (\bar{z}_1^*-z)^{\frac{k}{2}}, 
\label{eq:varphi2G1_2}
\end{align}
where $\tilde{\bvarphi}_{2,0}^{\tilde{G}_1}=\bvarphi_2(G_1(\bar{z}_1^*))$. If $\spr(\tilde{G}_1(0))=\spr(G_1(\bar{z}_1^*))=r_2$, then $\zeta=0$ is a pole of $\tilde{\bvarphi}_2(\tilde{G}_1(\zeta))$ with order one and the Puiseux series expansion for $\bvarphi_2(G_1(z))$ around $z=\bar{z}_1^*$ is represented as
\begin{align}
\bvarphi_2(G_1(z)) 
= \tilde{\bvarphi}_2(\tilde{G}_1((\bar{z}_1^*-z)^{\frac{1}{2}})) 
= \sum_{k=-1}^\infty \tilde{\bvarphi}_{2,k}^{\tilde{G}_1}\, (\bar{z}_1^*-z)^{\frac{k}{2}}. 
\label{eq:varphi2G1_3}
\end{align}
From the formulae above, we obtain the following proposition, where $\bu_{s_0}(z)$ and $\bv_{s_0}(z)$ are denoted by $\bu_{s_0}^{G_1}(z)$ and $\bv_{s_0}^{G_1}(z)$, respectively (the proof is given in Appendix \ref{sec:varphis2_limit_proofs}). 
\begin{proposition} \label{pr:limit_varphi2G1} 
In the case of Type I, if $r_1=\bar{z}_1^*$, then $\spr(G_1(\bar{z}_1^*))<r_2$ and $\tilde{\bvarphi}_{2,1}^{\tilde{G}_1}$ in expansion (\ref{eq:varphi2G1_2}) is given by 
\begin{align}
\tilde{\bvarphi}_{2,1}^{\tilde{G}_1} 
&= \sum_{k=1}^\infty \bnu_{0,k} \sum_{l=1}^k \underline{\zeta}_2(\bar{z}_1^*)^{k-l} G_1(\bar{z}_1^*)^{l-1} \tilde{G}_{1,1} \le \bzero^\top,\ \ne \bzero^\top. 
\label{eq:limit_varphi2G1_z1max_1}
\end{align}
In the case of Type II, if $r_1<\bar{z}_1^*$, then $\spr(G_1(r_1))=r_2<\bar{z}_2^*$ and $\bvarphi_{2,-1}^{G_1}$ in expansion (\ref{eq:varphi2G1_1}) is given by 
\begin{align}
\bvarphi_{2,-1}^{G_1} 
&= \underline{\zeta}_{2,z}(r_1)^{-1}\,c_{pole}^{\bvarphi_2} \bu^{C_2}(r_2)\,\bv_{s_0}^{G_1}(r_1) \bu_{s_0}^{G_1}(r_1) \ge \bzero^\top,\ \ne \bzero^\top;  
\label{eq:limit_varphi2G1_r1_II_1}
\end{align}
if $r_1=\bar{z}_1^*$, then $\spr(G_1(\bar{z}_1^*))=r_2<\bar{z}_2^*$ and $\tilde{\bvarphi}_{2,-1}^{\tilde{G}_1}$ in expansion (\ref{eq:varphi2G1_3}) is given by 
\begin{align}
\tilde{\bvarphi}_{2,-1}^{\tilde{G}_1} 
&= (-\tilde{\alpha}_{s_0,1}^{G_1})^{-1}\,c_{pole}^{\bvarphi_2} \bu^{C_2}(r_2)\,\bv_{s_0}^{G_1}(\bar{z}_1^*) \bu_{s_0}^{G_1}(\bar{z}_1^*) \ge \bzero^\top,\ \ne \bzero^\top. 
\label{eq:limit_varphi2G1_z1max_II_2}
\end{align}
\end{proposition}

Discussion about $\bvarphi_2^{\hat{C}_2}(z,G_1(z))$ proceeds in parallel to that about $\bvarphi_2(G_1(z))$. 
In a manner similar to that used in deriving expression (\ref{eq:varphi2G1_extension0}), we obtain 
\begin{align}
\bvarphi_2^{\hat{C}_2}(z,G_1(z))^\top
&= (V_1(z)^\top)^{-1} \sum_{k=1}^\infty \left( \bnu_{0,k}\otimes J_1(z)^k \right) \vecM((\hat{C}_2(z,G_1(z))V_1(z))^\top),  
\end{align}
and this leads us to 
\begin{align}
\bvarphi_2^{\hat{C}_2}(z,G_1(z)) 
&= \Bigl( \bvarphi_2(\alpha_1(z)) C_2(z,\alpha_1(z)) \bv_1(z)\ \ \bvarphi_2(\alpha_2(z)) C_2(z,\alpha_2(z)) \bv_2(z) \cr
&\qquad\quad \cdots\ \ \bvarphi_2(\alpha_{s_0}(z)) C_2(z,\alpha_{s_0}(z)) \bv_{s_0}(z) \Bigr)  V(z)^{-1}. 
\label{eq:varphi2tildeCG1_extension}
\end{align} 
Assume $r_1<\bar{z}_1^*$ and $\spr(G_1(r_1))=r_2<\bar{z}_1^*$. By Corollary \ref{co:varphi2_limit_typeIa}, since $\alpha_{s_0}(z)$ is analytic at $z=r_1$ and $\alpha_{s_0}(r_1)=r_2$, the point $z=r_1$ is a pole of $\bvarphi_2(\alpha_{s_0}(z))$ with order one. Hence, the Laurent series expansion for $\bvarphi_2^{\hat{C}_2}(z,G_1(z))$ around $z=r_1$ is represented as
\begin{equation}
\bvarphi_2^{\hat{C}_2}(z,G_1(z)) = \sum_{k=-1}^\infty \bvarphi_{2,k}^{\hat{C}_2}\, (r_1-z)^k. 
\label{eq:varphi2C2G1_expansion1}
\end{equation}
Assume $r_1=\bar{z}_1^*$ and define $\tilde{\bvarphi}_2^{\hat{C}_2}(\zeta,\tilde{G}_1(\zeta))$ as 
\begin{align}
\tilde{\bvarphi}_2^{\hat{C}_2}(\zeta,\tilde{G}_1(\zeta)) 
&= \Bigl( \bvarphi_2(\alpha_1(\bar{z}_1^*-\zeta^2)) C_2(\bar{z}_1^*-\zeta^2,\alpha_1(\bar{z}_1^*-\zeta^2)) \bv_1(\bar{z}_1^*-\zeta^2)\ \ \cdots \cr
&\qquad\quad \bvarphi_2(\alpha_{s_0-1}(\bar{z}_1^*-\zeta^2)) C_2(\bar{z}_1^*-\zeta^2,\alpha_{s_0-1}(\bar{z}_1^*-\zeta^2)) \bv_{s_0-1}(\bar{z}_1^*-\zeta^2)\cr
&\qquad\quad\quad \bvarphi_2(\tilde{\alpha}_{s_0}(\zeta)) C_2(\bar{z}_1^*-\zeta^2,\tilde{\alpha}_{s_0}(\zeta)) \tilde{\bv}_{s_0}(\zeta) \Bigr)  \tilde{V}(\zeta)^{-1}. 
\label{eq:varphi2G1_extension3}
\end{align} 
If $\spr(\tilde{G}_1(0))=\spr(G_1(\bar{z}_1^*))<r_2$, then $\tilde{\bvarphi}_2^{\hat{C}_2}(\zeta,\tilde{G}_1(\zeta))$ is entry-wise analytic at $\zeta=0$ and the Puiseux series expansion for $\bvarphi_2^{\hat{C}_2}(z,G_1(z))$ around $z=\bar{z}_1^*$ is represented as
\begin{align}
\bvarphi_2^{\hat{C}_2}(z,G_1(z)) 
= \tilde{\bvarphi}_2^{\hat{C}_2}((\bar{z}_1^*-z)^{\frac{1}{2}},\tilde{G}_1((\bar{z}_1^*-z)^{\frac{1}{2}})) 
= \sum_{k=0}^\infty \tilde{\bvarphi}_{2,k}^{\hat{C}_2}\, (\bar{z}_1^*-z)^{\frac{k}{2}}.
\label{eq:varphi2C2G1_expansion2}
\end{align}
If $\spr(\tilde{G}_1(0))=r_2<\bar{z}_1^*$, then the point $\zeta=0$ is a pole of $\bvarphi_2^{\hat{C}_2}(\zeta,\tilde{G}_1(\zeta))$ with order one and the Puiseux series expansion for $\bvarphi_2^{\hat{C}_2}(z,G_1(z))$ around $z=\bar{z}_1^*$ is represented as
\begin{align}
\bvarphi_2^{\hat{C}_2}(z,G_1(z)) 
= \tilde{\bvarphi}_2^{\hat{C}_2}((\bar{z}_1^*-z)^{\frac{1}{2}},\tilde{G}_1((\bar{z}_1^*-z)^{\frac{1}{2}})) 
= \sum_{k=-1}^\infty \tilde{\bvarphi}_{2,k}^{\hat{C}_2}\, (\bar{z}_1^*-z)^{\frac{k}{2}}. 
\label{eq:varphi2C2G1_expansion3}
\end{align}
For $\bvarphi_2^{\hat{C}_2}(z,G_1(z))$, we give the following proposition. Since this proposition is proved in a manner similar to that used for proving Proposition \ref{pr:limit_varphi2G1}, we omit the proof. 
\begin{proposition} \label{pr:limit_varphi2C2G1}
In the case of Type I, if $r_1=\bar{z}_1^*$, then $\spr(G_1(\bar{z}_1^*))<r_2$ and $\tilde{\bvarphi}^{\hat{C}_2}_{2,1}$ in expansion (\ref{eq:varphi2C2G1_expansion2}) is given by 
\begin{align}
\tilde{\bvarphi}^{\hat{C}_2}_{2,1} 
&= \underline{\zeta}_2(\bar{z}_1^*)^{-1} \bvarphi_2(\underline{\zeta}_2(\bar{z}_1^*)) \left( A_{*,0}^{(2)}(\bar{z}_1^*) + A_{*,1}^{(2)}(\bar{z}_1^*) \bigl( \underline{\zeta}_2(\bar{z}_1^*) I + G_1(\bar{z}_1^*) \bigr) \right) \tilde{G}_{1,1} \cr
&\quad + \sum_{k=1}^\infty \bnu_{0,k} \hat{C}_2(\bar{z}_1^*,G_1(\bar{z}_1^*)) \sum_{l=1}^{k-1} \underline{\zeta}_2(\bar{z}_1^*)^{k-l-1} G_1(\bar{z}_1^*)^{l-1} \tilde{G}_{1,1} \le \bzero^\top,\ \ne\bzero^\top, 
\label{eq:limit_varphi2C2G1_z1max_1}
\end{align}
In the case of Type II, if $r_1<\bar{z}_1^*$, then $\spr(G_1(r_1))=r_2<\bar{z}_2^*$ and $\bvarphi^{\hat{C}_2}_{2,-1}$ in expansion (\ref{eq:varphi2C2G1_expansion1}) is given by 
\begin{align}
\bvarphi^{\hat{C}_2}_{2,-1} 
&= \underline{\zeta}_{2,z}(r_1)^{-1}\,c_{pole}^{\bvarphi_2} \bu^{C_2}(r_2)\,C_2(r_1,r_2)\,\bv_{s_0}^{G_1}(r_1) \bu_{s_0}^{G_1}(r_1) \ge \bzero^\top,\ \ne \bzero^\top; 
\label{eq:limit_varphi2C2G1_r1_II_1}
\end{align}
if $r_1=\bar{z}_1^*$, then $\spr(G_1(\bar{z}_1^*))=r_2<\bar{z}_2^*$ and $\tilde{\bvarphi}^{\hat{C}_2}_{2,-1}$ in expansion (\ref{eq:varphi2C2G1_expansion3}) is given by 
\begin{align}
\tilde{\bvarphi}^{\hat{C}_2}_{2,-1} 
&= (-\tilde{\alpha}_{s_0,1}^{G_1})^{-1}\,c_{pole}^{\bvarphi_2} \bu^{C_2}(r_2)\,C_2(\bar{z}_1^*,r_2)\,\bv_{s_0}^{G_1}(\bar{z}_1^*) \bu_{s_0}^{G_1}(\bar{z}_1^*) \ge \bzero^\top,\ \ne \bzero^\top. 
\label{eq:limit_varphi2C2G1_z1max_II_2}
\end{align}
\end{proposition}

\begin{table}[htp]
\caption{Singularity of each terms in formula (\ref{eq:varphi1}) at $z=r_1$ (Types I and III)}
\begin{center}
\begin{tabular}{l|ccc|c} 
& \multicolumn{3}{c|}{Type I} & Type III \cr
& $\psi_1(\bar{z}_1^*)>1$ & $\psi_1(\bar{z}_1^*)=1$ &  $\psi_1(\bar{z}_1^*)<1$ \cr \hline
$\bnu_{0,0}(C_0(z,G_1(z))-I)$ & analytic & branch & branch & analytic \cr
$\bvarphi_2(G_1(z))$ & analytic & branch & branch & analytic \cr
$\bvarphi_2^{\hat{C}_2}(z,,G_1(z))$ & analytic & branch & branch & analytic \cr
$(I-C_1(z,G_1(z)))^{-1}$ & pole & pole and branch & branch & pole
\end{tabular}
\end{center}
\label{tab:varphi1_singularity_IandIII}
\end{table}%

\begin{table}[htp]
\caption{Singularity of each terms in formula (\ref{eq:varphi1}) at $z=r_1$ (Type II)}
\begin{center}
\begin{tabular}{l|c|ccc} 
& $\eta_2^{(c)}<\theta_2^{(c)}$ &  \multicolumn{3}{c}{$\eta_2^{(c)}=\theta_2^{(c)}$}  \cr
& & $\psi_1(\bar{z}_1^*)>1$ & $\psi_1(\bar{z}_1^*)=1$ & $\psi_1(\bar{z}_1^*)<1$ \cr \hline
$\bnu_{0,0}(C_0(z,G_1(z))-I)$ & analytic & analytic & branch & branch \cr
$\bvarphi_2(G_1(z))$ & pole & pole & pole and branch & pole and branch \cr
$\bvarphi_2^{\hat{C}_2}(z,G_1(z))$ & pole & pole & pole and branch & pole and branch \cr
$(I-C_1(z,G_1(z)))^{-1}$ & analytic & pole & pole and branch & branch
\end{tabular}
\end{center}
\label{tab:varphi1_singularity_II}
\end{table}%

%
\subsection{Singularity of $\bvarphi_1(z)$ in the other cases} \label{sec:expansion_varpshi_typeI_II_III}

Define $\tilde{\bvarphi}_1(\zeta)$ in a neighborhood of $\zeta=0$ as 
\begin{equation}
\tilde{\bvarphi}_1(\zeta) 
= \tilde{\bg}_1(\zeta)\,(I-C_1(\bar{z}_1^*-\zeta^2,\tilde{G}_1(\zeta)))^{-1} 
= \frac{\tilde{\bg}_1(\zeta)\,\adj(I-C_1(\bar{z}_1^*-\zeta^2,\tilde{G}_1(\zeta)))}{\tilde{f}_1(1,\zeta)}, 
\label{eq:varphi1_gf2}
\end{equation}
where
\begin{align*}
\tilde{\bg}_1(\zeta) &= \tilde{\bvarphi}_2^{\hat{C}_2}(\zeta,\tilde{G}_1(\zeta)) - \tilde{\bvarphi}_2(\tilde{G}_1(\zeta)) + \bnu_{0,0} (C_0(\bar{z}_1^*-\zeta^2,\tilde{G}_1(\zeta))-I ), \\
\tilde{f}_1(\lambda,\zeta) &= \det(\lambda I-C_1(\bar{z}_1^*-\zeta^2,\tilde{G}_1(\zeta))). 
\end{align*}
$\tilde{f}_1(\lambda,\zeta)$ is the characteristic function of $C_1(\bar{z}_1^*-\zeta^2,\tilde{G}_1(\zeta))$. Let $\tilde{\lambda}^{C_1}(\zeta)$ be the eigenvalue of $C_1(\bar{z}_1^*-\zeta^2,\tilde{G}_1(\zeta))$ satisfying $\tilde{\lambda}^{C_1}(0)=\psi_1(\bar{z}_1^*)$. We have $\lambda^{C_1}(z)=\tilde{\lambda}^{C_1}((\bar{z}_1^*-z)^{\frac{1}{2}})$ in a neighborhood of $z=\bar{z}_1^*$, where $\lambda^{C_1}(z)$ is the eigenvalue of $C_1(z,G_1(z))$ corresponding to $\psi_1(z)$ when $z\in[\underline{z}_1^*,\bar{z}_1^*]$. 
For $\tilde{f}_1(\zeta)$, we give the following proposition, which corresponds to Proposition \ref{pr:f1_limit}. 
\begin{proposition} \label{pr:tildef1_limit}
If $\psi_1(\bar{z}_1^*)=1$, then the point $\zeta=0$ is a zero of $\tilde{f}_1(1,\zeta)$ with multiplicity one and we have 
\begin{align}
\lim_{\zeta\to 0} -\zeta^{-1} \tilde{f}_1(1,\zeta) = \tilde{\lambda}^{C_1}_\zeta(0) f_{1,\lambda}(1,\bar{z}_1^*) \ne 0, 
\label{eq:tildef1_limit}
\end{align}
where $\tilde{\lambda}^{C_1}_\zeta(\zeta)=(\partial/\partial\,\zeta) \tilde{\lambda}^{C_1}(\zeta)$ and $\tilde{\lambda}^{C_1}_\zeta(0)=\bu^{C_1}(\bar{z}_1^*) A_{*,1}^{(1)}(\bar{z}_1^*) \tilde{G}_{1,1} \bv^{C_1}(\bar{z}_1^*)<0$. 
\end{proposition}

In the proof of Proposition \ref{pr:tildef1_limit}, we use the following proposition (its proof is given in Appendix \ref{sec:uC1_positivity}). 
\begin{proposition} \label{pr:uC1_positivity}
For any $z_0\in[\underline{z}_1^*,\bar{z}_1^*]$ such that $\psi_1(z_0)=1$ and for any $k\in\mathbb{Z}_+$, $\bu^{C_1}(z_0)$ and $\bu^{C_1}(z_0) A_{*,1}^{(1)}(z_0) N_1(z_0) R_1(z_0)^k$ are positive. 
Analogously, for any $z_0\in[\underline{z}_2^*,\bar{z}_2^*]$ such that $\psi_2(z_0)=1$ and for any $k\in\mathbb{Z}_+$, $\bu^{C_2}(z_0)$ and $\bu^{C_2}(z_0) A_{1,*}^{(2)}(z_0) N_2(z_0) R_2(z_0)^k$ are also positive.
\end{proposition}

\begin{proof}[Proof of Proposition \ref{pr:tildef1_limit}]
Under Assumption \ref{as:Akl_irreducible}, $C_1(\bar{z}_1^*,G_1(\bar{z}_1^*))$ is irreducible and the Perron-Frobenius eigenvalue $\psi_1(\bar{z}_1^*)$ is simple. Hence, the eigenvalue $\tilde{\lambda}^{C_1}(0)=\psi_1(\bar{z}_1^*)$ is also simple and $\tilde{\lambda}^{C_1}(\zeta)$ is analytic at $\zeta=0$. 
Since $\tilde{\lambda}^{C_1}(0)=\psi_1(\bar{z}_1^*)=1$, we have
\begin{align}
\lim_{\zeta\to 0} (1-\tilde{\lambda}^{C_1}_{s_0}(\zeta))^{-1} \tilde{f}_1(1,\zeta) = f_{1,\lambda}(1,\bar{z}_1^*) \ne 0. 
\end{align}
Combining this with the fact that $\lim_{\zeta\to 0} (1-\tilde{\lambda}^{C_1}(\zeta))/(0-\zeta) = \tilde{\lambda}^{C_1}_\zeta(0)$, we obtain equation (\ref{eq:tildef1_limit}). 
Let $\tilde{\bu}^{C_1}(\zeta)$ and $\tilde{\bv}^{C_1}(\zeta)$ be the left and right eigenvectors of $C_1(\bar{z}_1^*-\zeta^2,\tilde{G}_1(\zeta))$ with respect to the eigenvalue $\tilde{\lambda}^{C_1}(\zeta)$, satisfying $\tilde{\bu}^{C_1}(\zeta) \tilde{\bv}^{C_1}(\zeta) =1$. Since $\tilde{\lambda}^{C_1}(0)$ is simple, $\tilde{\bu}^{C_1}(\zeta)$ and $\tilde{\bv}^{C_1}(\zeta)$ are analytic at $\zeta=0$. 
Through some manipulation, we obtain 
\begin{equation}
\tilde{\lambda}^{C_1}_\zeta(\zeta) = \tilde{\bu}^{C_1}(\zeta) \left(\frac{d}{d \zeta} C_1(\bar{z}_1^*-\zeta^2,\tilde{G}_1(\zeta)) \right)  \tilde{\bv}^{C_1}(\zeta), 
\end{equation}
and this leads us to 
\begin{equation}
\tilde{\lambda}^{C_1}_\zeta(0) 
= \tilde{\bu}^{C_1}(0) A_{*,1}^{(1)}(\bar{z}_1^*) \tilde{G}_{1,1} \tilde{\bv}^{C_1}(0)
= \bu^{C_1}(\bar{z}_1^*) A_{*,1}^{(1)}(\bar{z}_1^*) \tilde{G}_{1,1} \bv^{C_1}(\bar{z}_1^*). 
\end{equation}
Recall that $\tilde{G}_{1,1}=\tilde{\alpha}_{s_0,1}^{G_1} N_1(\bar{z}_1^*) \bv^{R_1}(\bar{z}_1^*) \bu_{s_0}^{G_1}(\bar{z}_1^*)$, where $\tilde{\alpha}_{s_0,1}^{G_1}<0$ and $\bv^{R_1}(\bar{z}_1^*)$ and $\bu_{s_0}^{G_1}(\bar{z}_1^*)$ are nonzero and nonnegative. 
Since $C_1(\bar{z}_1^*,G_1(\bar{z}_1^*))$ is irreducible, $\bv^{C_1}(\bar{z}_1^*)$ is positive and $\bu_{s_0}^{G_1}(\bar{z}_1^*) \bv^{C_1}(\bar{z}_1^*)>0$. 
By Proposition \ref{pr:uC1_positivity}, $\bu^{C_1}(\bar{z}_1^*) A_{*,1}^{(1)}(\bar{z}_1^*)N_1(\bar{z}_1^*) \bv^{R_1}(\bar{z}_1^*)>0$. 
Hence, we have $\tilde{\lambda}^{C_1}_\zeta(0)<0$. 
\end{proof}

The following proposition corresponds to Proposition \ref{pr:adjC1}. Since it can analogously be proved, we omit its proof. 

\begin{proposition} \label{pr:adjC1_zs}
If $\psi_1(\bar{z}_1^*)=1$, then 
\begin{align}
&\adj\big(I-C_1(\bar{z}_1^*,G_1(\bar{z}_1^*))\big) 
= f_{1,\lambda}(1,\bar{z}_1^*) \bv^{C_1}(\bar{z}_1^*) \bu^{C_1}(\bar{z}_1^*), 
\label{eq:adjC1_zs}
\end{align}
where $\bv^{C_1}(\bar{z}_1^*)$ and $\bu^{C_1}(\bar{z}_1^*)$ are positive. 
\end{proposition}

Using the results in the previous subsections, we obtain from expressions (\ref{eq:varphi1_gf1}) and (\ref{eq:varphi1_gf2}) the following two lemmas (see Tables \ref{tab:varphi1_singularity_IandIII} and \ref{tab:varphi1_singularity_II}).

%
\begin{lemma}[Type I] \label{le:varphi1_limit_typeIbc}
(1) Type I ($\psi_1(\bar{z}_1^*)=1$).\quad The point $z=\bar{z}_1^*$ is a branch point of $\bvarphi_1(z)$ with order one and it is also a pole with order one. The Puiseux series expansion for $\bvarphi_1(z)$ around $z=\bar{z}_1^*$ is represented as 
\begin{equation}
\bvarphi_1(z) 
= \tilde{\bvarphi}_1((\bar{z}_1^*-z)^{\frac{1}{2}})) 
= \sum_{k=-1}^\infty \tilde{\bvarphi}_{1,k}^I\,(\bar{z}_1^*-z)^{\frac{k}{2}}, 
\label{eq:varphi1_expansion_Ib}
\end{equation}
where 
\begin{equation}
\tilde{\bvarphi}_{1,-1}^I 
=(-\tilde{\lambda}^{C_1}_\zeta(0))^{-1} \bg_1(\bar{z}_1^*)\bv^{C_1}(\bar{z}_1^*)\,\bu^{C_1}(\bar{z}_1^*) > \bzero^\top. 
\label{eq:tildevarphiI1m1}
\end{equation}

(2) Type I ($\psi_1(\bar{z}_1^*)<1$).\quad The point $z=\bar{z}_1^*$ is a branch point of $\bvarphi_1(z)$ with order one and its Puiseux series expansion is represented as 
\begin{equation}
\bvarphi_1(z) 
= \tilde{\bvarphi}_1((\bar{z}_1^*-z)^{\frac{1}{2}})) 
= \sum_{k=0}^\infty \tilde{\bvarphi}_{1,k}^I\,(\bar{z}_1^*-z)^{\frac{k}{2}}, 
\label{eq:varphi1_expansion_Ic}
\end{equation}
where $\tilde{\bvarphi}_{1,0}^I=\bvarphi_1(\bar{z}_1^*)$ and 
\begin{align}
\tilde{\bvarphi}_{1,1}^I &= 
\Bigl( \underline{\zeta}_2(\bar{z}_1^*)^{-1} \bvarphi_2(\underline{\zeta}_2(\bar{z}_1^*)) \bigl( A_{*,0}^{(2)}(\bar{z}_1^*) + A_{*,1}^{(2)}(\bar{z}_1^*) ( \underline{\zeta}_2(\bar{z}_1^*) I + G_1(\bar{z}_1^*) ) \bigr) \cr
&\quad + \sum_{k=1}^\infty \bnu_{0,k} \hat{C}_2(\bar{z}_1^*,G_1(\bar{z}_1^*)) \sum_{l=1}^{k-1} \underline{\zeta}_2(\bar{z}_1^*)^{k-l-1} G_1(\bar{z}_1^*)^{l-1}  \cr
&\quad -\sum_{k=1}^\infty \bnu_{0,k} \sum_{l=1}^k \underline{\zeta}_2(\bar{z}_1^*)^{k-l} G_1(\bar{z}_1^*)^{l-1} 
+ \bnu_{0,0} A_{*,1}^{(0)}(\bar{z}_1^*) \cr
&\quad + \bvarphi_1(\bar{z}_1^*) A_{*,1}^{(1)}(\bar{z}_1^*) \Bigr) \tilde{G}_{1,1} \bigl(I-C_1(\bar{z}_1^*,G_1(\bar{z}_1^*) \bigr)^{-1}. 
\label{eq:tildevarphiI11}
\end{align}
\end{lemma}

\begin{proof}
(1) In the case of Type I ($\psi_1(\bar{z}_1^*)=1$), since $r_1=\bar{z}_1^*$ and $\spr(\tilde{G}_1(0))=\spr(G_1(\bar{z}_1^*))<r_2$, $\tilde{\bg}_1(\zeta)$ in formula (\ref{eq:varphi1_gf2}) is analytic at $\zeta=0$. $\tilde{f}_1(1,\zeta)$ in formula (\ref{eq:varphi1_gf2}) is also analytic at $\zeta=0$. 
By Proposition \ref{pr:tildef1_limit}, since $\psi_1(\bar{z}_1^*)=1$, the point $\zeta=0$ is a zero of $\tilde{f}_1(1,\zeta)$ with multiplicity one. 
Hence, $\zeta=0$ is probably a pole of $\tilde{\bvarphi}_1(\zeta)$ with order one and the Puiseux series expansion for $\bvarphi_1(z)$ is given by formula (\ref{eq:varphi1_expansion_Ib}). Derivation of the expression for $\tilde{\bvarphi}_{1,-1}^I$ is given in Appendix \ref{sec:get_coefficients}. 
Suppose $\bg_1(\bar{z}_1^*) \bv^{C_1}(\bar{z}_1^*) = 0$. Then, since $\bv^{C_1}(\bar{z}_1^*)$ is the right eigenvalue of $C_1(\bar{z}_1^*,G_1(\bar{z}_1^*))$ with respect to the eigenvalue $\psi_1(\bar{z}_1^*)=1$, we have $\bg_1(\bar{z}_1^*) C_1(\bar{z}_1^*,G_1(\bar{z}_1^*))^k \bv^{C_1}(\bar{z}_1^*) = 0$ for any $k\ge 1$, and this leads us to
\[
0 = \lim_{K\to\infty} \sum_{k=0}^K \bg_1(\bar{z}_1^*) C_1(\bar{z}_1^*,G_1(\bar{z}_1^*))^k \bv^{C_1}(\bar{z}_1^*) 
= \bvarphi_1(\bar{z}_1^*) \bv^{C_1}(\bar{z}_1^*). 
\]
Since $\bv^{C_1}(\bar{z}_1^*)>\bzero^\top$, this contradicts to positivity of $\bvarphi_1(\bar{z}_1^*)$, where we include the case where some entries of $\bvarphi_1(\bar{z}_1^*)$ diverge to positive infinity. Hence, we have $\bg_1(\bar{z}_1^*) \bv^{C_1}(\bar{z}_1^*) \ne 0$, and $\tilde{\bvarphi}_{1,-1}^I$ is positive. This implies that $\zeta=0$ is definitely a pole of $\tilde{\bvarphi}_1(\zeta)$ with order one. 

(2) In the case of Type I ($\psi_1(\bar{z}_1^*)<1$), because of the same reason, $\tilde{\bg}_1(\zeta)$ in formula (\ref{eq:varphi1_gf2}) is analytic at $\zeta=0$. Since $\psi_1(\bar{z}_1^*)<1$, we have $\tilde{f}_1(1,0)\ne 0$. 
Hence, $\tilde{\bvarphi}_1(\zeta)$ is analytic at $\zeta=0$ and the Puiseux series expansion for $\bvarphi_1(z)$ is given by formula (\ref{eq:varphi1_expansion_Ic}). Derivation of the expression for $\tilde{\bvarphi}_{1,1}^I$ is given in Appendix \ref{sec:get_coefficients}. 
\end{proof}

%
\begin{lemma}[Type II] \label{le:varphi1_limit_typeII}
(1) Type II ($\eta_2^{(c)}<\theta_2^{(c)}$).\quad The point $z=r_1$ is a pole of $\bvarphi_1(z)$ with order one and its Laurent series expansion is represented as 
\begin{equation}
\bvarphi_1(z) 
= \sum_{k=-1}^\infty \bvarphi_{1,k}^{II}\,(r_1-z)^k, 
\label{eq:varphi1_expansion_IIa}
\end{equation}
where 
\begin{align}
\bvarphi_{1,-1}^{II} 
&= \frac{c_{pole}^{\bvarphi_2} \bu^{C_2}(r_2) A_{1,*}^{(2)}(r_2) N_2(r_2) \bv^{R_2}(r_2)}{\underline{\zeta}_{2,z}(r_1)}\, \bu_{s_0}^{G_1}(r_1) \bigl( I-C_1(r_1,G_1(r_1)) \bigr)^{-1} > \bzero^\top 
\label{eq:varphiII1m1}
\end{align}
and $\bv^{R_2}(r_2)$ is the right eigenvector of $R_2(r_2)$ with respect to the eigenvalue $r_1^{-1}$.

(2) Type II ($\eta_2^{(c)}=\theta_2^{(c)}$ and $\psi_1(\bar{z}_1^*)>1$).\quad The point $z=\bar{z}_1^*$ is a pole of $\bvarphi_1(z)$ with order two and its Laurent series expansion is represented as 
\begin{equation}
\bvarphi_1(z) 
= \sum_{k=-2}^\infty \bvarphi_{1,k}^{II}\,(r_1-z)^k, 
\label{eq:varphi1_expansion_IIb}
\end{equation}
where 
\begin{align}
\bvarphi_{1,-2}^{II} &=
\frac{c_{pole}^{\bvarphi_2} \bu^{C_2}(r_2) A_{1,*}^{(2)}(r_2) N_2(r_2) \bv^{R_2}(r_2) \bu_{s_0}^{G_1}(r_1) \bv^{C_1}(r_1)}{\underline{\zeta}_{2,z}(r_1)\, \psi_{1,z}(r_1)}\, \bu^{C_1}(r_1) > \bzero^\top. 
\label{eq:varphiII1m2}
\end{align}

(3) Type II ($\eta_2^{(c)}=\theta_2^{(c)}$ and $\psi_1(\bar{z}_1^*)=1$).\quad The point $z=\bar{z}_1^*$ is a branch point of $\bvarphi_1(z)$ with order one and it is also a pole with order two. The Puiseux series expansion for $\bvarphi_1(z)$ around $z=\bar{z}_1^*$ is represented as 
\begin{equation}
\bvarphi_1(z) 
= \tilde{\bvarphi}_1((\bar{z}_1^*-z)^{\frac{1}{2}})) 
= \sum_{k=-2}^\infty \tilde{\bvarphi}_{1,k}^{II}\,(\bar{z}_1^*-z)^{\frac{k}{2}}, 
\label{eq:varphi1_expansion_IIc}
\end{equation}
where 
\begin{align}
\tilde{\bvarphi}_{1,-2}^{II} &=
\frac{c_{pole}^{\bvarphi_2} \bu^{C_2}(r_2) A_{1,*}^{(2)}(r_2) N_2(r_2) \bv^{R_2}(r_2) \bu_{s_0}^{G_1}(\bar{z}_1^*) \bv^{C_1}(\bar{z}_1^*)}{\tilde{\alpha}_{s_0,1}^{G_1}\, \tilde{\lambda}_\zeta^{C_1}(0)}\, \bu^{C_1}(\bar{z}_1^*) > \bzero^\top. 
\label{eq:tildevarphiII1m2}
\end{align}

(4) Type II ($\eta_2^{(c)}=\theta_2^{(c)}$ and $\psi_1(\bar{z}_1^*)<1$).\quad The point $z=\bar{z}_1^*$ is a branch point of $\bvarphi_1(z)$ with order one and it is also a pole with order one. The Puiseux series expansion for $\bvarphi_1(z)$   $z=\bar{z}_1^*$ is represented as 
\begin{equation}
\bvarphi_1(z) 
= \tilde{\bvarphi}_1((\bar{z}_1^*-z)^{\frac{1}{2}})) 
= \sum_{k=-1}^\infty \tilde{\bvarphi}_{1,k}^{II}\,(\bar{z}_1^*-z)^{\frac{k}{2}}, 
\label{eq:varphi1_expansion_IId}
\end{equation}
where 
\begin{align}
\tilde{\bvarphi}_{1,-1}^{II} 
&= \frac{c_{pole}^{\bvarphi_2} \bu^{C_2}(r_2) A_{1,*}^{(2)}(r_2) N_2(r_2) \bv^{R_2}(r_2)}{-\tilde{\alpha}_{s_0,1}^{G_1}}\, \bu_{s_0}^{G_1}(\bar{z}_1^*) \bigl( I-C_1(\bar{z}_1^*,G_1(\bar{z}_1^*)) \bigr)^{-1} > \bzero^\top. 
\label{eq:tildevarphiII1m1}
\end{align}
\end{lemma}

\begin{proof}
(1) In the case of Type II ($\eta_2^{(c)}<\theta_2^{(c)}$), since $r_1=e^{\bar{\eta}_1^{(c)}}<e^{\theta_1^{(c)}}$, we have $f_1(1,r_1)\ne 0$. Since $\spr(G_1(r_1))=\underline{\zeta}_2(r_1)=e^{\eta_2^{(c)}}=r_2$, the point $z=r_1$ is probably a pole of $\bg_1(z)$ with order one. 
Hence, $z=r_1$ is probably a pole of $\bvarphi_1(z)$ with order one and the Laurent series expansion for $\bvarphi_1(z)$ is given by formula (\ref{eq:varphi1_expansion_IIa}). Derivation of the expression for $\bvarphi_{1,-1}^{II}$ is given in Appendix \ref{sec:get_coefficients}. 
By Proposition \ref{pr:uC1_positivity}, $\bu^{C_2}(r_2) A_{1,*}^{(2)}(r_2) N_2(r_2) \bv^{R_2}(r_2)>0$, and we know that $(I-C_1(r_1,G_1(r_1)))^{-1}$ is positive. Hence, $\bvarphi_{1,-1}^{II}$ is positive, and this implies that $z=r_1$ is definitely a pole of $\bvarphi_1(z)$ with order one. 

(2) In the case of Type II ($\eta_2^{(c)}=\theta_2^{(c)}$ and $\psi_1(\bar{z}_1^*)>1$), since $r_1=e^{\bar{\eta}_1^{(c)}}=e^{\theta_1^{(c)}}<\bar{z}_1^*$, we have $\psi_1(r_1)=1$ and, by Proposition \ref{pr:f1_limit}, the point $z=r_1$ is a zero of $f_1(1,z)$ with multiplicity one. Because of the same reason as that used in (1), $z=r_1$ is probably a pole of $\bg_1(z)$ with order one. 
Hence, $z=r_1$ is probably a pole of $\bvarphi_1(z)$ with order two and the Laurent series expansion for $\bvarphi_1(z)$ is given by formula (\ref{eq:varphi1_expansion_IIb}). Derivation of the expression for $\bvarphi_{1,-2}^{II}$ is given in Appendix \ref{sec:get_coefficients}. 
We have $\bu^{C_2}(r_2) A_{1,*}^{(2)}(r_2) N_2(r_2) \bv^{R_2}(r_2)>0$ and $\bu_{s_0}^{G_1}(r_1) \bv^{C_1}(r_1)>0$. Hence, $\bvarphi_{1,-1}^{II}$ is positive, and this implies that $z=r_1$ is definitely a pole of $\bvarphi_1(z)$ with order two.

(3) In the case of Type II ($\eta_2^{(c)}=\theta_2^{(c)}$ and $\psi_1(\bar{z}_1^*)=1$), we have $r_1=e^{\bar{\eta}_1^{(c)}}=e^{\theta_1^{(c)}}=\bar{z}_1^*$. Since $\spr(\tilde{G}_1(0))=\spr(G_1(\bar{z}_1^*))=r_2$, the point $\zeta=0$ is probably a pole of $\tilde{\bg}_1(\zeta)$ with order one. 
By Proposition \ref{pr:tildef1_limit}, since $\psi_1(\bar{z}_1^*)=1$, $\zeta=0$ is a zero of $\tilde{f}_1(1,\zeta)$ with multiplicity one. 
Hence, $\zeta=0$ is probably a pole of $\tilde{\bvarphi}_1(\zeta)$ with at most order two and the Puiseux series expansion for $\bvarphi_1(z)$ is given by formula (\ref{eq:varphi1_expansion_IIc}). Derivation of the expression for $\tilde{\bvarphi}_{1,-2}^{II}$ is given in Appendix \ref{sec:get_coefficients}. 
In the same reason as that used in (2), $\tilde{\bvarphi}_{1,-2}^{II}$ is positive, and this implies that $\zeta=0$ is definitely a pole of $\tilde{\bvarphi}_1(\zeta)$ with order two.

(4) In the case of Type II ($\eta_2^{(c)}=\theta_2^{(c)}$ and $\psi_1(\bar{z}_1^*)<1$), we have $r_1=\bar{z}_1^*$ and the point $\zeta=0$ is probably a pole of $\tilde{\bg}_1(\zeta)$ with order one. Since $\psi_1(\bar{z}_1^*)<1$, we have $\tilde{f}_1(1,0)\ne 0$. 
Hence, $\zeta=0$ is probably a pole of $\tilde{\bvarphi}_1(\zeta)$ with order one and the Puiseux series expansion for $\bvarphi_1(z)$ is given by formula (\ref{eq:varphi1_expansion_IId}). Derivation of the expression for $\tilde{\bvarphi}_{1,-1}^{II}$ is given in Appendix \ref{sec:get_coefficients}. 
In the same reason as that used in (1), $\tilde{\bvarphi}_{1,-1}^{II}$ is positive, and this implies that $\zeta=0$ is definitely a pole of $\tilde{\bvarphi}_1(\zeta)$ with order one. 
\end{proof}

\begin{remark}
We have not yet clarified whether the coefficient vector $\tilde{\bvarphi}_{1,1}^I$ in Lemma \ref{le:varphi1_limit_typeIbc} is nonzero or not. We leave this point as a further study. 
\end{remark}

%
%

\appendix

%
%
\section{Proof of Lemma \ref{le:barGamma_bounded}} \label{sec:proof_barGamma_bounded}

\begin{proof}[Proof of Lemma \ref{le:barGamma_bounded}]
First we note that, for $n\ge 1$, $j\in S_0$ and $(s_1,s_2)\in\mathbb{R}^2$, $C(e^{s_1},e^{s_2})^n$ satisfies 
\begin{align}
&\big[ C(e^{s_1},e^{s_2})^n \big]_{jj} 
= \sum_{\bk_{(n)}\in\mathbb{H}^n} \sum_{\bl_{(n)}\in\mathbb{H}^n} \big[ A_{k_1,l_1} A_{k_2,l_2} \cdots A_{k_n,l_n} \big]_{jj}\,e^{s_1\sum_{p=1}^n k_p+s_2\sum_{p=1}^n l_p},
\label{eq:Cs1s2n}
\end{align}
where $\bk_{(n)}=(k_1,k_2,...,k_n)$ and $\bl_{(n)}=(l_1,l_2,...,l_n)$.
Consider the Markov chain $\{\tilde{\bY}_n\}=\{(\tilde{X}_{1,n},\tilde{X}_{2,n},\tilde{J}_n)\}$ generated by $\{A_{k,l},k,l\in\mathbb{H}\}$ (see Definition \ref{def:inducedMC_Akl}) and assume that $\{\tilde{\bY}_n\}$ starts from the state $(0,0,j)$. Since $\{A_{k,l},k,l\in\mathbb{H}\}$ is irreducible and aperiodic, there exists $n_0\ge 1$ such that, for every $j\in S_0$, $\mathbb{P}(\tilde{\bY}_{n_0}=(1,0,j)\,|\,\tilde{\bY}_0=(0,0,j))>0$, $\mathbb{P}(\tilde{\bY}_{n_0}=(0,1,j)\,|\,\tilde{\bY}_0=(0,0,j))>0$, $\mathbb{P}(\tilde{\bY}_{n_0}=(-1,0,j)\,|\,\tilde{\bY}_0=(0,0,j))>0$ and $\mathbb{P}(\tilde{\bY}_{n_0}=(0,-1,j)\,|\,\tilde{\bY}_0=(0,0,j))>0$. 
This implies that, for some $\bk_{1,(n_0)}^{(j)}$, $\bl_{1,(n_0)}^{(j)}$, $\bk_{2,(n_0)}^{(j)}$, $\bl_{2,(n_0)}^{(j)}$, $\bk_{3,(n_0)}^{(j)}$, $\bl_{3,(n_0)}^{(j)}$, $\bk_{4,(n_0)}^{(j)}$ and $\bl_{4,(n_0)}^{(j)}$ in $\mathbb{H}^{n_0}$, we have 
\begin{align*}
&b_{1,j}=\big[ A_{k_{1,1}^{(j)},l_{1,1}^{(j)}} A_{k_{1,2}^{(j)},l_{1,2}^{(j)}} \cdots A_{k_{1,n_0}^{(j)},l_{1,n_0}^{(j)}} \big]_{jj}>0,\quad \sum_{p=1}^{n_0} k_{1,p}^{(j)}=1,\quad \sum_{p=1}^{n_0} l_{1,p}^{(j)} =0,\\
&b_{2,j}=\big[ A_{k_{2,1}^{(j)},l_{2,1}^{(j)}} A_{k_{2,2}^{(j)},l_{2,2}^{(j)}} \cdots A_{k_{2,n_0}^{(j)},l_{2,n_0}^{(j)}} \big]_{jj}>0,\quad \sum_{p=1}^{n_0} k_{2,p}^{(j)}=0,\quad \sum_{p=1}^{n_0} l_{2,p}^{(j)} =1,\\
&b_{3,j}=\big[ A_{k_{3,1}^{(j)},l_{3,1}^{(j)}} A_{k_{3,2}^{(j)},l_{3,2}^{(j)}} \cdots A_{k_{3,n_0}^{(j)},l_{3,n_0}^{(j)}} \big]_{jj}>0,\quad \sum_{p=1}^{n_0} k_{3,p}^{(j)}=-1,\quad \sum_{p=1}^{n_0} l_{3,p}^{(j)} =0,\\
&b_{4,j}=\big[ A_{k_{4,1}^{(j)},l_{4,1}^{(j)}} A_{k_{4,2}^{(j)},l_{4,2}^{(j)}} \cdots A_{k_{4,n_0}^{(j)},l_{4,n_0}^{(j)}} \big]_{jj}>0,\quad \sum_{p=1}^{n_0} k_{4,p}^{(j)}=0,\quad \sum_{p=1}^{n_0} l_{4,p}^{(j)} =-1.
\end{align*} 
Hence, the sum of any row of $C(e^{s_1},e^{s_2})^{n_0}$ is greater than or equal to $b^*(e^{s_1}+e^{s_2}+e^{-s_1}+e^{-s_2})$, where
\[
b^* = \min_{j\in S_0} \min_{1\le i\le 4} b_{i,j}>0, 
\]
and we obtain 
\begin{align}
\chi(e^{s_1},e^{s_2})
=\spr(C(e^{s_1},e^{s_2})) 
= \spr(C(e^{s_1},e^{s_2})^{n_0})^{\frac{1}{n_0}}
\ge \left( b^*(e^{s_1}+e^{s_2}+e^{-s_1}+e^{-s_2}) \right)^{\frac{1}{n_0}}, 
\end{align}
where we use the fact that, for a nonnegative square matrix $A=(a_{i,j})$, $\spr(A)\ge \min_i \sum_j a_{i,j}$ (see, for example, Theorem 8.1.22 of Horn and Johnson \cite{Horn85}). This means that $\chi(e^{s_1},e^{s_2})$ is unbounded in any direction, and since $\chi(e^{s_1},e^{s_2})$ is convex in $(s_1,s_2)$, we see that $\bar{\Gamma}$ is a bounded set. 
\end{proof}

%
%
\section{Proof of Lemma \ref{le:mgf_convergence}} \label{sec:proof_lemma_mgf_convergence}

\begin{proof}
We denote by $J_k(\lambda)$ the $k$-dimensional Jordan block of eigenvalue $\lambda$. 
Note that the $n$-th power of $J_k(\lambda)$ is given by
\begin{equation}
J_k(\lambda)^n = \begin{pmatrix}
{}_nC_0\lambda^n & {}_nC_1\lambda^{n-1} & \cdots & {}_nC_{k-1}\lambda^{n-k+1} \cr
&  \ddots & \ddots & & \cr
&  & {}_nC_0\lambda^n & {}_nC_1\lambda^{n-1} \cr
& & & {}_nC_0\lambda^n 
\end{pmatrix}. 
\end{equation}
Without loss of generality, the Jordan canonical form $J$ of $X$ is represented as
\begin{equation}
J 
= T^{-1} X T 
= \begin{pmatrix}
J_{k_1}(\lambda_1) & & & \cr
 & J_{k_2}(\lambda_2) & & \cr
 &  & \ddots & \cr
 & & & J_{k_l}(\lambda_l) 
\end{pmatrix},
\end{equation}
where $T$ is a nonsingular matrix, $\lambda_1,\lambda_2,...,\lambda_l$ are the eigenvalues of $X$ and $k_1,k_2,...,k_l$ are positive integers satisfying $k_1+\cdots+k_l=m$. 
We have 
\begin{align}
\Big( \sum_{n=0}^\infty |\ba_n X^n| \Big)^\top 
&\le \sum_{n=0}^\infty |T^{-1}|^\top |J|^n |T|^\top |\ba_n|^\top \cr
&= \sum_{n=0}^\infty \left(\bone^\top\otimes |T^{-1}|^\top \right) \left( \diag(|\ba_n|)\otimes |J|^n \right) {\rm vec}(|T|^\top) \cr
&= \left(\bone^\top\otimes (|T^{-1}|)^\top \right) \diag\Big( \sum_{n=0}^\infty [|\ba_n|]_j |J|^n,\,j=1,2,...,s_0 \Big) {\rm vec}(|T|^\top), 
\end{align}
where we use the identity $\vecM(ABC)=(C^\top\otimes A) \vecM(B)$ for matrices $A$, $B$ and $C$. Note that 
\begin{align}
&\sum_{n=0}^\infty [|\ba_n|]_j |J|^n = \diag\Big( \sum_{n=0}^\infty [|\ba_n|]_j J_{k_s}(|\lambda_s|)^n,\,s=1,2,...,l \Big). 
\end{align}
and we have, for $t,u\in\{1,2,...,k_s\}$ such that $t\le u$,  
\begin{align}
&\sum_{n=0}^\infty \bigl[ [|\ba_n|]_j J_{k_s}(|\lambda_s|)^n \bigr]_{t,u} \cr
=&\ \sum_{n=0}^{k_s-2} \bigl[ [|\ba_n|]_j J_{k_s}(|\lambda_s|)^n \bigr]_{t,u} + \sum_{n=k_s-1}^\infty \frac{n!}{(u-t)! (n-u+t)!} [|\ba_n|]_j\,|\lambda_s|^{n-u+t} \cr
\le&\ \sum_{n=0}^{k_s-2} \bigl[ [|\ba_n|]_j J_{k_s}(|\lambda_s|)^n \bigr]_{t,u} + \sum_{n=u-t}^\infty \frac{n!}{(n-u+t)!} [|\ba_n|]_j\,|\lambda_s|^{n-u+t} \cr
=& \sum_{n=0}^{k_s-2} \bigl[ [|\ba_n|]_j J_{k_s}(|\lambda_s|)^n \bigr]_{t,u} + \frac{d^{u-t}}{d\,w^{u-t}}\,[\bphi_{abs}(w)]_j \Big|_{w=|\lambda_s|}, 
\label{eq:anJlambda}
\end{align}
where $\bphi_{abs}(w)=\sum_{n=0}^\infty |\ba_n|\,w^n$. For $w\in\mathbb{C}$ such that $|w|<r$, since $\bphi(w)$ is absolutely convergent, $\bphi_{abs}(w)$ is also absolutely convergent and analytic. Hence, $\bphi_{abs}(w)$ is differentiable any times and we know that the second term on the last line of formula (\ref{eq:anJlambda}) is finite since $|\lambda_s|\le \spr(X)<r$; this completes the proof.
\end{proof}

%
%
\section{Proof of Lemma \ref{le:Gmatrix}} \label{sec:proof_prop_Gmatrix}

\begin{proof}
{\it (i)}\quad 
First, we note that if $z$ is a positive real number in $[z_1^{min},z_1^{max}]$, then $N_1(z)$, $R_1(z)$ and $G_1(z)$ are finite. 
Let $z$ be a complex number satisfying $|z|\in[z_1^{min},z_1^{max}]$. Since $A_{i,j},\,i,j\in\mathbb{H}$, are nonnegative, we have for $j\in\mathbb{H}$, 
\[
|A_{*,j}(z)| = \biggl| \sum_{i\in\mathbb{H}} A_{i,j} z^{i} \biggr| 
\le \sum_{i\in\mathbb{H}} A_{i,j} |z|^i = A_{*,j}(|z|), 
\]
and for $n\ge 0$,  
\[
|Q_{11}^{(n)}(z)| = \biggl| \sum_{\bi_{(n)}\in\scrI_n} A_{*,i_1}(z) A_{*,i_2}(z) \cdots A_{*,i_n}(z) \biggr|
\le Q_{11}^{(n)}(|z|). 
\]
From this and the fact that $N_1(|z|)=\sum_{n=0}^\infty Q_{11}^{(n)}(|z|)$ is finite (convergent), we see that $N_1(z)=\sum_{n=0}^\infty Q_{11}^{(n)}(z)$ converges absolutely and obtain $|N_1(z)|\le N_1(|z|)$. 
Analogously, we see that both $G_1(z)$ and $R_1(z)$ also converge absolutely and satisfy expression (\ref{eq:absNRG}).

{\it (ii)}\quad 
Since, for $n\ge 1$, $\scrI_{U,1,n}$ and $\scrI_{U,1,n}$ satisfy 
\begin{align*}
\scrI_{D,1,n} &= \biggl\{\bi_{(n)}\in\mathbb{H}^n:\ \sum_{l=1}^k i_l\ge 0\ \mbox{for $k\in\{1,2,...,n-2\}$},\ \sum_{l=1}^{n-1} i_l=0\ \mbox{and}\ i_n=-1 \biggr\} \cr
&= \bigl\{(\bi_{(n-1)},-1):\ \bi_{(n-1)}\in\scrI_{n-1} \bigr\}, \\
\scrI_{U,1,n} &= \biggl\{\bi_{(n)}\in\mathbb{H}^n:\ i_1=1,\,\sum_{l=2}^k i_l\ge 0\ \mbox{for $k\in\{2,...,n-1\}$}\ \mbox{and} \sum_{l=2}^n i_l=0 \biggr\} \cr 
&= \bigl\{(1,\bi_{(n-1)}):\ \bi_{(n-1)}\in\scrI_{n-1} \bigr\},
\end{align*}
where $\bi_{(n)}=(i_1,i_2,...,i_n)$, we have 
\begin{align*}
&G_1(z) = \sum_{n=1}^\infty Q_{11}^{(n-1)}(z) A_{*,-1}(z)  = N_1(z) A_{*,-1}(z), \\
&R_1(z) = \sum_{n=1}^\infty A_{*,1}(z) Q_{11}^{(n-1)}(z) = A_{*,1}(z) N_1(z). 
\end{align*}. 

{\it (iii)}\quad 
We prove only equation (\ref{eq:G1equation}) since equation (\ref{eq:R1equation}) can analogously be proved. 
For $n\ge 3$, $\scrI_{D,1,n}$ satisfies 
\begin{align*}
\scrI_{D,1,n} 
&= \biggl\{\bi_{(n)}\in\mathbb{H}^n:\ i_1=0,\ \sum_{l=2}^k i_l\ge 0\ \mbox{for $k\in\{2,...,n-1\}$},\ \sum_{l=2}^n i_l=-1 \biggr\} \cr
&\qquad \bigcup\,\biggl\{\bi_{(n)}\in\mathbb{H}^n:\ i_1=1,\ \sum_{l=2}^k i_l\ge -1\ \mbox{for $k\in\{2,...,n-1\}$},\ \sum_{l=2}^n i_l=-2 \biggr\} \cr
&= \bigl\{(0,\bi_{(n-1)}):\ \bi_{(n-1)}\in\scrI_{D,1,n-1} \bigr\} \cup \bigl\{(1,\bi_{(n-1)}):\ \bi_{(n-1)}\in\scrI_{D,2,n-1} \bigr\}, 
\end{align*}
and $\scrI_{D,2,n}$ satisfies 
\begin{align*}
\scrI_{D,2,n} 
&= \bigcup_{m=1}^{n-1} \biggl\{\bi_{(n)}\in\mathbb{H}^n:\ \sum_{l=1}^k i_l\ge 0\ \mbox{for $k\in\{1,2,...,m-1\}$},\ \sum_{l=1}^{m} i_l=-1,\cr
&\qquad\qquad\qquad \sum_{l=m+1}^k i_l\ge 0\ \mbox{for $k\in\{m+1,m+2,...,n-1\}$}\ \mbox{and} \sum_{l=m+1}^{n} i_l=-1\biggr\} \cr
&= \bigcup_{m=1}^{n-1} \bigl\{(\bi_{(m)},\bi_{(n-m)}):\ \bi_{(m)}\in\scrI_{D,1,m}\ \mbox{and}\ \bi_{(n-m)}\in\scrI_{D,1,m-n} \bigr\}. 
\end{align*}
Hence, we have, for $n\ge 3$, 
\begin{align*}
D_1^{(n)}(z) 
&= A_{*,0}(z) D_1^{(n-1)}(z) + A_{*,1}(z) \sum_{\bi_{(n-1)}\in\scrI_{D,2,n-1}} A_{*,i_1}(z) A_{*,i_2}(z) \cdots A_{*,i_{n-1}}(z) \cr
&= A_{*,0}(z) D_1^{(n-1)}(z) + A_{*,1}(z) \sum_{m=1}^{n-2} D_1^{(m)}(z) D_1^{(n-m-1)}(z), 
\end{align*} 
and obtain
\begin{align*}
G_1(z) 
&= D_1^{(1)}(z) + \sum_{n=2}^\infty A_{*,0}(z) D_1^{(n-1)}(z) + A_{*,1}(z) \sum_{n=3}^\infty \sum_{m=1}^{n-2} D_1^{(m)}(z) D_1^{(n-m-1)}(z) \cr
&= A_{*,-1}(z) + A_{*,0}(z) G_1(z) + A_{*,1}(z) G_1(z)^2, 
\end{align*}
where we use the fact that $D_1^{(1)}(z)=A_{*,-1}(z)$ and $D_1^{(2)}(z)=A_{*,0}(z) D_1^{(1)}(z)=A_{*,0}(z) A_{*,-1}(z)$. 

{\it (iv)}\quad 
For $n\ge 1$, $\scrI_{n}$ satisfies 
\begin{align*}
\scrI_{n} 
&= \biggl\{\bi_{(n)}\in\mathbb{H}^n:\ i_1=0,\ \sum_{l=2}^k i_l\ge 0\ \mbox{for $k\in\{2,...,n-1\}$},\ \sum_{l=2}^n i_l=0 \biggr\} \cr
&\qquad \bigcup\,\biggl\{\bi_{(n)}\in\mathbb{H}^n:\ i_1=1,\ \sum_{l=2}^k i_l\ge -1\ \mbox{for $k\in\{2,...,n-1\}$},\ \sum_{l=2}^n i_l=-1 \biggr\} \cr
&= \{(0,\bi_{(n-1)}):\ \bi_{(n-1)}\in\scrI_{n-1} \} \cr
&\qquad \bigcup_{m=2}^n\,\biggl\{\bi_{(n)}\in\mathbb{H}^n:\ i_1=1,\ \sum_{l=2}^k i_l\ge 0\ \mbox{for $k\in\{2,...,m-1\}$},\ \sum_{l=2}^m i_l=-1, \cr
&\qquad\qquad\quad \sum_{l=m+1}^k i_l\ge 0\ \mbox{for $k\in\{m+1,...,n-1\}$},\ \sum_{l=m+1}^n i_l=0 \biggr\} \cr
&= \{(0,\bi_{(n-1)}):\ \bi_{(n-1)}\in\scrI_{n-1} \} \cr
&\qquad \cup \left(\cup_{m=2}^n \{(1,\bi_{(m-1)},\bi_{(n-m)}):\ \bi_{(m-1)}\in\scrI_{D,1,m-1}, \bi_{(n-m)}\in\scrI_{n-m} \} \right). 
\end{align*}
Hence, we have, for $n\ge 1$, 
\[
Q_{1,1}^{(n)}(z) = A_{*,0}(z) Q_{1,1}^{(n-1)} + \sum_{m=2}^n A_{*,1}(z) D_1^{(m-1)}(z) Q_{1,1}^{(n-m)}(z), 
\] 
and this leads us to 
\begin{align*}
N_1(z) = I + A_{*,0}(z) N_1(z) + A_{*,1}(z) G_1(z) N_1(z).
\end{align*}
From this equation, we immediately obtain equation (\ref{eq:H1N1_relation}).

{\it (v)}\quad 
Substituting $N_1(z) A_{*,-1}(z)$, $A_{*,1}(z) N_1(z)$ and $A_{*,0}(z)+A_{*,1}(z) N_1(z) A_{*,-1}(z)$ for $G_1(z)$, $R_1(z)$ and $H_1(z)$, respectively, in the right hand side of equation (\ref{eq:WFfact}), we obtain the left hand side of the equation via straightforward calculation.
\end{proof}

%
%
\section{Proof of Proposition \ref{pr:sprCzw}} \label{sec:sprCzw}
\begin{proof}
Set $z=r_1 e^{i \theta_1}$ and $w=r_2 e^{i \theta_2}$, where $r_1,r_2>0$, $\theta_1, \theta_2\in[0,2\pi)$ and $i=\sqrt{-1}$. 
For $n\ge 1$ and $j\in S_0$, $C(z,w)^n$ satisfies 
\begin{align}
\Bigl| \big[ C(z,w)^n \big]_{jj} \Bigr|
&= \biggl| \sum_{\bk_{(n)}\in\mathbb{H}^n} \sum_{\bl_{(n)}\in\mathbb{H}^n} \big[ A_{k_1,l_1} A_{k_2,l_2} \cdots A_{k_n,l_n} \big]_{jj} \cr  
&\qquad\qquad\qquad\qquad \cdot r_1^{\sum_{p=1}^n k_p} r_2^{\sum_{p=1}^n l_p} e^{i (\theta_1 \sum_{p=1}^n k_p + \theta_2 \sum_{p=1}^n l_p)} \biggr| \cr 
&\le \sum_{\bk_{(n)}\in\mathbb{H}^n} \sum_{\bl_{(n)}\in\mathbb{H}^n} \big[ A_{k_1,l_1} A_{k_2,l_2} \cdots A_{k_n,l_n} \big]_{jj}\,r_1^{\sum_{p=1}^n k_p} r_2^{\sum_{p=1}^n l_p} \cr
&= \big[ C(|z|,|w|) \big]_{jj}, 
\label{eq:Czwn}
\end{align}
where $\bk_{(n)}=(k_1,k_2,...,k_n)$ and $\bl_{(n)}=(l_1,l_2,...,l_n)$. In this formula, equality holds only when, for every $\bk_{(n)},\bl_{(n)}\in\mathbb{H}^n$ such that $\big[ A_{k_1,l_1} A_{k_2,l_2} \cdots A_{k_n,l_n} \big]_{jj}\ne 0$, $e^{i (\theta_1 \sum_{p=1}^n k_p + \theta_2 \sum_{p=1}^n l_p)}$ takes some common value. 
Consider the Markov chain $\{\tilde{\bY}_n\}=\{(\tilde{X}_{1,n},\tilde{X}_{2,n},\tilde{J}_n)\}$ generated by $\{A_{k,l},k,l\in\mathbb{H}\}$ (see Definition \ref{def:inducedMC_Akl}) and assume that $\{\tilde{\bY}_n\}$ starts from the state $(0,0,j)$. Since $\{A_{k,l},k,l\in\mathbb{H}\}$ is irreducible and aperiodic, there exists $n_0\ge 1$ such that $\mathbb{P}(\tilde{\bY}_{n_0}=(0,0,j)\,|\,\tilde{\bY}_0=(0,0,j))>0$, $\mathbb{P}(\tilde{\bY}_{n_0}=(1,0,j)\,|\,\tilde{\bY}_0=(0,0,j))>0$ and $\mathbb{P}(\tilde{\bY}_{n_0}=(0,1,j)\,|\,\tilde{\bY}_0=(0,0,j))>0$. 
This implies that, for some $\bk_{(n_0)}$, $\bl_{(n_0)}$, $\bk_{(n_0)}'$, $\bl_{(n_0)}'$, $\bk_{(n_0)}''$ and $\bl_{(n_0)}''$ in $\mathbb{H}^{n_0}$, we have 
\begin{align*}
&\big[ A_{k_1,l_1} A_{k_2,l_2} \cdots A_{k_{n_0},l_{n_0}} \big]_{jj}>0,\quad \sum_{p=1}^{n_0} k_p=0,\quad \sum_{p=1}^{n_0} l_p =0,\\
&\big[ A_{k_1',l_1'} A_{k_2',l_2'} \cdots A_{k_{n_0}',l_{n_0}'} \big]_{jj}>0,\quad \sum_{p=1}^{n_0} k_p'=1,\quad \sum_{p=1}^{n_0} l_p' =0,\\
&\big[ A_{k_1'',l_1''} A_{k_2'',l_2''} \cdots A_{k_{n_0}'',l_{n_0}''} \big]_{jj}>0,\quad \sum_{p=1}^{n_0} k_p''=0,\quad \sum_{p=1}^{n_0} l_p'' =1. 
\end{align*} 
Hence, we see that if equality holds in formula (\ref{eq:Czwn}), then $e^{i \theta_1}=e^{i \theta_2}=1$ and both $\theta_1$ and $\theta_2$ must be zero. This implies that if $\theta_1\ne 0$ or $\theta_2\ne 0$, then we have 
\[
\big| [C(z,w)^{n_0}]_{jj} \big| = \Bigl| \big[ C(r_1 e^{i\theta_1},r_2 e^{i\theta_2})^{n_0} \big]_{jj} \Bigr| < \big[ C(r_1,r_2)^{n_0} \big]_{jj} = \big[ C(|z|,|w|)^{n_0} \big]_{jj}, 
\]
and obtain 
\begin{align}
\spr(C(z,w))^{n_0} &\le \spr(|C(z,w)^{n_0}|) < \spr(C(|z|,|w|)^{n_0}) = \spr(C(|z|,|w|))^{n_0},  
\end{align}
where we use Theorem 1.5 of Seneta \cite{Seneta06} and the fact that $C(|z|,|w|))$ is irreducible. Obviously, this implies that $\spr(C(z,w))<\spr(C(|z|,|w|))$. 
\end{proof}

%
%
\section{Proofs of Propositions \ref{pr:varphi2_analytic} and \ref{pr:C1_inequality}} \label{sec:varphis2_C1_proofs}

%
\begin{proof}[Proof of Proposition \ref{pr:varphi2_analytic}]
Let $X=(x_{k,l})$ be an $s_0\times s_0$ complex matrix. By Lemma \ref{le:mgf_convergence}, if $\spr(X)<r_2$, then $\bvarphi_2(X)$ converges absolutely. This means that each element of $\bvarphi_2(X)$ is an absolutely convergent series with $s_0^2$ valuables $x_{11},x_{12},...,x_{s_0,s_0}$ and it is an analytic function with $s_0^2$ valuables in the region $\{X=(x_{k,l})\in\mathbb{C}^{s_0^2}: \spr(X)<r_2\}$.
Therefore, if $G_1(z)$ is entry-wise analytic on $\Omega$ and $\spr(G_1(z))<r_2$, we see that $\bvarphi_2(G_1(z))=(\bvarphi_2\circ G_1)(z)$ is also entry-wise analytic on $\Omega$. 
For $z\in\mathbb{C}$ such that $|z|<r_2$, we have 
\[
\sum_{j=1}^\infty \Bigl| \bnu_{0,j} \hat{C}_2(z,G_1(z))\, z^{j-1} \Bigr|
\le |z|^{-1} \sum_{j=1}^\infty \bnu_{0,j} |z|^j\, \hat{C}_2(|z|,G_1(|z|))  
< \infty. 
\]
Hence, by Lemma \ref{le:mgf_convergence}, if $\spr(X)<r_2$, $\sum_{j=1}^\infty \bnu_{0,j} \hat{C}_2(z,G_1(z))\, X^{j-1}$ converges absolutely, and we see that if $G_1(z)$ is entry-wise analytic on $\Omega$ and $\spr(G_1(z))<r_2$, then $\bvarphi_2^{\hat{C}_2}(z,G_1(z))$ is entry-wise analytic on $\Omega$. 
\end{proof}

%
\begin{proof}[Proof of Proposition \ref{pr:C1_inequality}]
Set $z=r e^{i\theta}$, where $r>0$, $\theta\in[0,2\pi)$ and $i=\sqrt{-1}$. By the definition of $G_1(z)$, we have
\begin{align}
C_1(z,G_1(z)) &= \sum_{n=1}^\infty\,\sum_{\bk_{(n+1)}\in\mathbb{H}^{n+1}}\,\sum_{\bl_{(n)}\in\scrI_{D,1,n}} A^{(1)}_{k_1,1} A_{k_2,l_1} A_{k_3,l_2} \cdots A_{k_{n+1},l_{n}} \cr
&\qquad\qquad\qquad\qquad\qquad \cdot r^{\sum_{p=1}^{n+1} k_p} e^{i (\theta \sum_{p=1}^{n+1} k_p)} + \sum_{k\in\mathbb{H}} A^{(1)}_{k,0} r^{k} e^{i\theta k}, 
\label{eq:C1zG1_Y1n}
\end{align}
where $\bk_{(n+1)}=(k_1,k_2,...,k_{n+1})$ and $\bl_{(n)}=(l_1,l_2,...,l_n)$.
Consider the Markov chain $\{\tilde{\bY}^{(1)}_n\}=\{(\tilde{X}^{(1)}_{1,n},\tilde{X}^{(1)}_{2,n},\tilde{J}^{(1)}_n)\}$ generated by $\{ \{A_{k,l},k,l\in\mathbb{H}\}, \{A^{(1)}_{k,l},k\in\mathbb{H}, l\in\mathbb{H}_+\} \}$ (see Definition \ref{def:inducedMC_Akl12}), and assume that $\{\tilde{\bY}^{(1)}_n\}$ starts from a state in $\{0\}\times\{0\}\times S_0$. 
For $n_0\in\mathbb{N}$, let $\tau_{n_0}$ be the time when the Markov chain enters a state in $\mathbb{Z}\times\{0\}\times S_0$ for the $n_0$-th time. Then, the term $\sum_{p=1}^{n+1} k_p$ (resp.\ $k$) in expression (\ref{eq:C1zG1_Y1n}) indicates that $\tilde{X}^{(1)}_{1,\tau_1}=\sum_{p=1}^{n+1} k_p$ (resp.\ $\tilde{X}^{(1)}_{1,\tau_1}=k$). 
This point analogously holds for $\tilde{X}^{(1)}_{1,\tau_{n_0}}$ when $n_0>1$. Hence, $C_1(z,G_1(z))^{n_0}$ can be represented as 
\begin{align}
C_1(z,G_1(z))^{n_0} = \sum_{k\in\mathbb{Z}} \tilde{D}_k r^k e^{i\theta k}, 
\label{eq:C1zG1_Dre1}
\end{align}
where $k$ indicates that $\tilde{X}^{(1)}_{1,\tau_{n_0}}=k$ and the $(j,j')$-entry of nonnegative square matrix $\tilde{D}_k$ is given as $[\tilde{D}_k]_{j,j'}=\mathbb{P}(\tilde{\bY}^{(1)}_{\tau_{n_0}}=(k,0,j')\,|\,\tilde{\bY}^{(1)}_0=(0,0,j))$. 
We have 
\begin{align}
|C_1(z,G_1(z))^{n_0}| = \left|\sum_{k\in\mathbb{Z}} \tilde{D}_k r^k e^{i\theta k}\right| \le C_1(|z|,G_1(|z|))^{n_0}, 
\label{eq:C1zG1_Dre2}
\end{align}
and equality holds only when, for each $j,j'\in S_0$ and for every $k\in\mathbb{Z}$ such that $[\tilde{D}_k]_{j,j'}>0$, $e^{i\theta k}$ takes some common value. 
Under Assumption \ref{as:Akl_irreducible}, $\{\tilde{\bY}^{(1)}_n\}$ is irreducible and aperiodic, and for any $j,j'\in S_0$, there exists $n_0\ge 1$ such that $\mathbb{P}(\tilde{\bY}^{(1)}_{\tau_{n_0}}=(0,0,j')\,|\,\tilde{\bY}^{(1)}_0=(0,0,j))>0$ and $\mathbb{P}(\tilde{\bY}^{(1)}_{\tau_{n_0}}=(1,0,j')\,|\,\tilde{\bY}^{(1)}_0=(0,0,j))>0$. 
This implies that $[\tilde{D}_0]_{jj'}>0$ and $[\tilde{D}_1]_{jj'}>0$. Hence, if $\theta\ne 0$ ($z\ne |z|$), then we have 
\[
|[C_1(z,G_1(z))^{n_0}]_{j,j'}| < [C_1(|z|,G_1(|z|))^{n_0}]_{j,j'}, 
\label{eq:C1zG1_Dre3}
\]
and obtain 
\begin{align}
\spr(C_1(z,G_1(z))^{n_0}) \le \spr(|C_1(z,G_1(z))^{n_0}|) < \spr(C_1(|z|,G_1(|z|))^{n_0}),  
\end{align}
where we use Theorem 1.5 of Seneta \cite{Seneta06} and the fact that $C_1(|z|,G_1(|z|))$ is irreducible, which is an immediate consequence from Assumption \ref{as:Akl_irreducible}. Obviously, this implies that $\spr(C_1(z,G_1(z))) < \spr(C_1(|z|,G_1(|z|)))$ when $z\ne |z|$. 
\end{proof}

%
%
\section{Proof of Proposition \ref{pr:adjC1}} \label{sec:adjC1}

\begin{proof}[Proof of Proposition \ref{pr:adjC1}]
Since $\psi_1(z_0)=1$, we have
\begin{align}
O &=(I-C_1(z_0,G_1(z_0)))\, \adj(I-C_1(z_0,G_1(z_0))) \cr
&= \adj(I-C_1(z_0,G_1(z_0)))\, (I-C_1(z_0,G_1(z_0))). 
\end{align}
Furthermore, since ${\rm rank}(I-C_1(z_0,G_1(z_0)))=s_0-1$, we have $\adj(I-C_1(z_0,G_1(z_0)))\ne O$. 
Hence, for $i,j\in\{1,2,...,s_0\}$, the $i$-th row vector $\bu_i$  and $j$-th column vector $\bv_j$ composing $\adj(I-C_1(z_0,G_1(z_0)))$ are nonzero and satisfy
\begin{equation}
\bu_i C_1(z_0,G_1(z_0)) = \bu_i, \quad  C_1(z_0,G_1(z_0)) \bv_j= \bv_j. 
\end{equation}
Since $C_1(z_0,G_1(z_0))$ is irreducible and the algebraic multiplicity of the eigenvalue $\psi_1(z_0)$ is one, we see that, for some $s_0\times 1$ vector $\ba=(a_i)$ and $1\times s_0$ vector $\bb=(b_j)$ and for every $i,j\in\{1,2,...,s_0\}$, $\bu_i=a_i \bu^{C_1}(z_0)$ and $\bv_j=\bv^{C_1}(z_0) b_j$.
This implies that $\adj(I-C_1(z_0,G_1(z_0)))=\ba \bu^{C_1}(z_0)=\bv^{C_1}(z_0) \bb$. From this, we obtain $\ba = \ba \bu^{C_1}(z_0) \bv^{C_1}(z_0) = \bv^{C_1}(z_0) \bb \bu^{C_1}(z_0)$, and $\adj(I-C_1(z_0,G_1(z_0)))$ is represented as
\begin{equation}
\adj(I-C_1(z_0,G_1(z_0))) = c\, \bv^{C_1}(z_0) \bu^{C_1}(z_0), 
\label{eq:adjIC1_cvu}
\end{equation}
where $c=\bb \bu^{C_1}(z_0)$. Note that both $\bu^{C_1}(z_0)$ and $\bv^{C_1}(z_0)$ are positive since $C_1(z_0,G_1(z_0))$ is irreducible. 
Next, we determine the coefficient $c$. For $z<z_0$, we have 
\[
\big(I-C_1(z,G_1(z))\big)^{-1} \bv^{C_1}(z) 
= \sum_{k=0}^\infty C_1(z,G_1(z))^k \bv^{C_1}(z) 
= (1-\psi_1(z))^{-1} \bv^{C_1}(z).
\]
From this and Proposition \ref{pr:f1_limit}, we obtain 
\begin{align*}
\frac{\adj(I-C_1(z_0,G_1(z_0)))\bv^{C_1}(z)}{\psi_{1,z}(z_0) f_{1,\lambda}(1,z_0)} 
&= \lim_{z\uparrow z_0} (z_0-z) \big(I-C_1(z,G_1(z))\big)^{-1} \bv^{C_1}(z)
= \frac{\bv^{C_1}(z)}{\psi_{1,z}(z_0)}, 
\end{align*}
where we use the fact that $\psi_1(z)$ is differentiable on $(\underline{z}_1^*,\bar{z}_1^*)$. 
Hence, from equation (\ref{eq:adjIC1_cvu}), we obtain $c=f_{1,\lambda}(1,z_0)$ and this completes the proof. 
\end{proof}

%
%
\section{Proof of Proposition \ref{pr:limit_G1}} \label{sec:limit_G1}

\begin{proof}[Proof of Proposition \ref{pr:limit_G1}]　
$\tilde{G}_{1,1}$ is given by $\tilde{G}_{1,1}=(d/d\zeta)\,\tilde{G}_1(\zeta)\,|_{\zeta=0}$.
Differentiating the both sides of $\tilde{G}_1(\zeta)=\tilde{V}_1(\zeta)\tilde{J}_1(\zeta)\tilde{V}_1(\zeta)^{-1}$ and setting $\zeta=0$, we obtain 
\begin{equation}
\tilde{G}_{1,1} = \tilde{\alpha}_{s_0,1}^{G_1}\, \bv^\dagger\, \bu_{s_0}^{G_1}(\bar{z}_1^*), 
\label{eq:tildeG1_diff1}
\end{equation}
where 
$\bv^\dagger = \bv_{s_0}^{G_1}(\bar{z}_1^*) + (\tilde{\alpha}_{s_0,1}^{G_1} )^{-1} ( \underline{\zeta}_2(\bar{z}_1^*) I - G_1(\bar{z}_1^*) ) \tilde{\bv}_{s_0,1}$ and 
$\tilde{\bv}_{s_0,1} = (d/d\zeta)\,\tilde{\bv}_{s_0}(\zeta)|_{\zeta=0}$. In derivation of equation(\ref{eq:tildeG1_diff1}), we use the following identity:
\[
\frac{d}{d\,\zeta} \tilde{V}_1(\zeta)^{-1} = -\tilde{V}_1(\zeta)^{-1} \left(\frac{d}{d\,\zeta} \tilde{V}_1(\zeta)\right) \tilde{V}_1(\zeta)^{-1}. 
\]
Note that $\bu_{s_0}^{G_1}(\bar{z}_1^*) \bv^\dagger = 1$ and $\tilde{\alpha}_{s_0,1}^{G_1}$ is an eigenvalue of $\tilde{G}_{1,1}$. 
By equation (\ref{eq:G1equation}), $\tilde{G}_1(\zeta)$ satisfies 
\begin{equation}
\tilde{G}_1(\zeta) = A_{*,-1}(\bar{z}_1^*-\zeta^2)+A_{*,0}(\bar{z}_1^*-\zeta^2) \tilde{G}_1(\zeta)+A_{*,1}(\bar{z}_1^*-\zeta^2) \tilde{G}_1(\zeta)^2. 
\label{eq:tildeG1_equation}
\end{equation}
Differentiating both the sides of equation (\ref{eq:tildeG1_equation}) and setting $\zeta=0$, we obtain 
\begin{equation}
\tilde{G}_{1,1} = A^\dagger\, \tilde{G}_{1,1}, 
\label{eq:tildeG1_diff2}
\end{equation}
where $A^\dagger=A_{*,0}(\bar{z}_1^*) + \underline{\zeta}_2(\bar{z}_1^*) A_{*,1}(\bar{z}_1^*) + A_{*,1}(\bar{z}_1^*) G_1(\bar{z}_1^*)$. In derivation of equation (\ref{eq:tildeG1_diff2}), we use the fact that $\tilde{G}_{1,1} G_1(\bar{z}_1^*) = \underline{\zeta}_2(\bar{z}_1^*) \tilde{G}_{1,1}$.
Multiplying the both sides of equation (\ref{eq:tildeG1_diff2}) by $\bv_{s_0}^{G_1}(\bar{z}_1^*)$ from the right, we obtain $A^\dagger \bv^\dagger = \bv^\dagger$. Hence, $\bv^\dagger$ is the right eigenvector of $A^\dagger$ with respect to the eigenvalue 1. 
From equation (\ref{eq:WFfact}) in Lemma \ref{le:Gmatrix}, we obtain 
\begin{equation}
I-A^\dagger = (\bar{\zeta}_2(\bar{z}_1^*)^{-1}I-R_1(\bar{z}_1^*)) (I-H_1(\bar{z}_1^*)),
\end{equation}
and from equation (\ref{eq:H1N1_relation}), we know that $N_1(\bar{z}_1^*)=(I-H_1(\bar{z}_1^*))^{-1}$. Hence, $N_1(\bar{z}_1^*) \bv^{R_1}(\bar{z}_1^*)$ is the right eigenvector of $A^\dagger$ with respect to the eigenvalue 1, and we see that $\bv^\dagger$ can be given by $\bv^\dagger = N_1(\bar{z}_1^*) \bv^{R_1}(\bar{z}_1^*)$. 
Since $\bu_{s_0}^{G_1}(\bar{z}_1^*)$ is the left eigenvector of $\tilde{G}_{1,1}$ with respect to the eigenvalue $\tilde{\alpha}_{s_0,1}^{G_1}$, $\bv^{R_1}(\bar{z}_1^*)$ must satisfy $\bu_{s_0}^{G_1}(\bar{z}_1^*) N_1(\bar{z}_1^*) \bv^{R_1}(\bar{z}_1^*)=1$. 
Since $\tilde{\alpha}_{s_0,1}^{G_1}$ is negative, positivity of $-\tilde{G}_{1,1}$ is obvious. Since $N_1(\bar{z}_1^*)\ge I$, we have $N_1(\bar{z}_1^*) \bv^{R_1}(\bar{z}_1^*)\ge \bv^{R_1}(\bar{z}_1^*)\ne \bzero$ and $-\tilde{G}_{1,1}$ is nonzero. 
\end{proof}

%
%
\section{Proof of Proposition \ref{pr:limit_varphi2G1} } \label{sec:varphis2_limit_proofs}

Before proving Proposition \ref{pr:limit_varphi2G1}, we give the following proposition. 
\begin{proposition} \label{pr:limit_varphi2} 
In the case of Type II, if $r_1<\bar{z}_1^*$, then $\underline{\zeta}_2(r_1)=r_2<\bar{z}_2^*$ and 
\begin{align}
\lim_{z\to r_1} (r_1-z) \bvarphi_2(\alpha_{s_0}(z)) 
&= \underline{\zeta}_{2,z}(r_1)^{-1}\,c_{pole}^{\bvarphi_2}\, \bu^{C_2}(r_2) > \bzero^\top, 
\label{eq:limit_varphi2_r1}
\end{align}
where $\underline{\zeta}_{2,z}(z)=(d/d z) \underline{\zeta}_2(z)$; if $r_1=\bar{z}_1^*$, then $\underline{\zeta}_2(\bar{z}_1^*)=r_2<\bar{z}_2^*$ and 
\begin{align}
\lim_{\tilde{\Delta}\ni z\to \bar{z}_1^*} (\bar{z}_1^*-z)^{\frac{1}{2}} \bvarphi_2(\alpha_{s_0}(z)) 
&= (-\tilde{\alpha}_{s_0,1}^{G_1})^{-1}\,c_{pole}^{\bvarphi_2}\, \bu^{C_2}(r_2) > \bzero^\top. 
\label{eq:limit_varphi2_z1max}
\end{align}
\end{proposition}
%
\begin{proof}
In the case of Type II, we have $\alpha_{s_0}(r_1)=\underline{\zeta}_2(e^{\bar{\eta}_1^{(c)}})=e^{\eta_2^{(c)}}=r_2<\bar{z}_2^*$, and the point $w=r_2$ is a pole of $\bvarphi_2(w)$ with order one. 
If $r_1<\bar{z}_1^*$, $\alpha_{s_0}(z)$ is analytic at $z=r_1$ and, by Corollary \ref{co:varphi2_limit_typeIa}, we have 
\begin{align}
\lim_{z\to r_1} (r_1-z)\,\bvarphi_2(\alpha_{s_0}(z)) 
&= \lim_{z\to r_1} (r_2-\alpha_{s_0}(z))\,\bvarphi_2(\alpha_{s_0}(z))\, \frac{r_1-z}{r_2-\alpha_{s_0}(z)} \cr
&= c_{pole}^{\bvarphi_2}\, \bu^{C_2}(r_2)/\underline{\zeta}_{2,z}(r_1), 
\end{align}
and if $r_1=\bar{z}_1^*$, 
\begin{align}
\lim_{\tilde{\Delta}\ni z\to \bar{z}_1^*} (\bar{z}_1^*-z)^{\frac{1}{2}}\,\bvarphi_2(\alpha_{s_0}(z)) 
&= \lim_{\tilde{\Delta}\ni z\to \bar{z}_1^*} (r_2-\alpha_{s_0}(z))\,\bvarphi_2(\alpha_{s_0}(z))\, \biggl( - \frac{(\bar{z}_1^*-z)^\frac{1}{2}}{\alpha_{s_0}(z)-r_2} \biggr) \cr
&= c_{pole}^{\bvarphi_2}\, \bu^{C_2}(r_2) /(-\tilde{\alpha}_{s_0,1}^{G_1}), 
\end{align}
where we use Proposition \ref{pr:limit_eigenG1}. 
\end{proof}

%
\begin{proof}[Proof of Proposition \ref{pr:limit_varphi2G1}]
In the case of Type I, we always have $\spr(G_1(r_1))=\underline{\zeta}_2(r_1)<r_2$. Assume $r_1=\bar{z}_1^*$. By Proposition \ref{pr:varphi2_analytic}, since $\tilde{G}_1(0)=\spr(G_1(\bar{z}_1^*))<r_2$, $\tilde{\bvarphi}_2(\zeta)$ is given in a form of absolutely convergent series as
\begin{equation}
\tilde{\bvarphi}_2(\tilde{G}_1(\zeta)) = \sum_{k=0}^\infty \bnu_{0,k} \tilde{G}_1(\zeta)^k.
\end{equation}
This $\tilde{\bvarphi}_2(\tilde{G}_1(\zeta))$ is entry-wise analytic at $\zeta=0$ and we have
\begin{align}
\tilde{\bvarphi}_{2,1}^{\tilde{G}_1} 
= \frac{d}{d \zeta} \tilde{\bvarphi}_2(\tilde{G}_1(\zeta)) \Big|_{\zeta=0}
&= \sum_{k=1}^\infty \bnu_{0,k} \sum_{l=1}^k G_1(\bar{z}_1^*)^{l-1} \tilde{G}_{1,1} G_1(\bar{z}_1^*)^{k-l} \cr
&= \sum_{k=1}^\infty \bnu_{0,k} \sum_{l=1}^k \underline{\zeta}_2(\bar{z}_1^*)^{k-l} G_1(\bar{z}_1^*)^{l-1} \tilde{G}_{1,1},
\end{align}
where we use the fact that 
\[
\tilde{G}_{1,1} G_1(\bar{z}_1^*) = \tilde{\alpha}_{s_0,1}^{G_1} N_1(\bar{z}_1^*) \bv^{R_1}(\bar{z}_1^*) \bu_{s_0}^{G_1}(\bar{z}_1^*) G_1(\bar{z}_1^*) = \underline{\zeta}_2(\bar{z}_1^*) \tilde{G}_{1,1}. 
\]
Since $\tilde{G}_{1,1}$ is nonzero and nonpositive, $\tilde{\bvarphi}_{2,1}^{\tilde{G}_1}$ is also nonzero and nonpositive. 

In the case of Type II, we have $\spr(G_1(r_1))=\underline{\zeta}_2(r_1)=r_2<\bar{z}_2^*$ and $\bvarphi_2(w)$ has a pole at $w=\underline{\zeta}_2(r_1)$. Hence, if $r_1<\bar{z}_1^*$, we obtain from equation (\ref{eq:varphi2G1_extension1}) and Proposition \ref{pr:limit_varphi2} that 
\begin{align}
\bvarphi_{2,-1}^{G_1}
&=\lim_{z\to r_1} (r_1-z)\,\bvarphi_2(G_1(z)) \cr
&= \begin{pmatrix} \bzero & \cdots & \bzero & \underline{\zeta}_{2,z}(r_1)^{-1}\,c_{pole}^{\bvarphi_2}\, \bu^{C_2}(r_2)\,\bv_{s_0}^{G_1}(r_1) \end{pmatrix} V_1(r_1)^{-1} \cr
&= \underline{\zeta}_{2,z}(r_1)^{-1}\,c_{pole}^{\bvarphi_2}\, \bu^{C_2}(r_2)\,\bv_{s_0}^{G_1}(r_1) \bu_{s_0}^{G_1}(r_1), 
\label{eq:limit_varphisG1_II_1b}
\end{align}
where $\bu^{C_2}(r_2)\,\bv_{s_0}^{G_1}(r_1)>0$; if $r_1=\bar{z}_1^*$, we also obtain from equation (\ref{eq:varphi2G1_extension2}) and Proposition \ref{pr:limit_varphi2} that 
\begin{align}
\tilde{\bvarphi}_{2,-1}^{\tilde{G}_1}
&= \lim_{\tilde{\Delta}\ni z\to r_1} (\bar{z}_1^*-z)^{\frac{1}{2}}\,\bvarphi_2(G_1(z)) \cr
&= \begin{pmatrix} \bzero & \cdots & \bzero & (-\tilde{\alpha}_{s_0}^{G_1})^{-1}\,c_{pole}^{\bvarphi_2}\, \bu^{C_2}(r_2)\,\bv_{s_0}^{G_1}(\bar{z}_1^*) \end{pmatrix} V_1(\bar{z}_1^*)^{-1} \cr
&= (-\tilde{\alpha}_{s_0}^{G_1})^{-1}\,c_{pole}^{\bvarphi_2}\, \bu^{C_2}(r_2)\,\bv_{s_0}^{G_1}(\bar{z}_1^*) \bu_{s_0}^{G_1}(\bar{z}_1^*), 
\label{eq:limit_varphisG1_II_1c}
\end{align}
where $\bu^{C_2}(r_2)\,\bv_{s_0}^{G_1}(\bar{z}_1^*) >0$.
\end{proof}

%
%
\section{Proof of Proposition \ref{pr:uC1_positivity}} \label{sec:uC1_positivity}

\begin{proof}[Proof of Proposition \ref{pr:uC1_positivity}]
We prove only the first half of the proposition. 
For $z\in[\underline{z}_1^*,\bar{z}_1^*]$, define a nonnegative block tri-diagonal matrix $A^{(1)}_*(z)$ as 
\[
A^{(1)}_*(z) = 
\begin{pmatrix}
A^{(1)}_{*,0}(z) & A^{(1)}_{*,1}(z) & & & \cr
A_{*,-1}(z) & A_{*,0}(z) & A_{*,1}(z) & & \cr
& A_{*,-1}(z) & A_{*,0}(z) & A_{*,1}(z) & \cr
& & \ddots & \ddots & \ddots 
\end{pmatrix}.
\]
Under Assumption \ref{as:Akl_irreducible}, $A^{(1)}_*(z)$ is irreducible and aperiodic. 
Furthermore, for any $z_0\in[\underline{z}_1^*,\bar{z}_1^*]$ such that $\psi_1(z_0)=1$, the invariant measure $\bu_*^{(1)}(z_0)$ satisfying $\bu_*^{(1)}(z_0) A^{(1)}_*(z_0) = \bu_*^{(1)}(z_0)$ is given as follows (see Theorem 3.1 of Ozawa \cite{Ozawa13}): 
\begin{align}
\bu_*^{(1)}(z_0) &= \big( \bu^{C_1}(z_0)\ \ \bu^{C_1}(z_0) A_{*,1}^{(1)}(z_0) N_1(z_0)\ \ \bu^{C_1}(z_0) A_{*,1}^{(1)}(z_0) N_1(z_0) R_1(z_0) \cr
&\qquad\qquad \bu^{C_1}(z_0) A_{*,1}^{(1)}(z_0) N_1(z_0) R_1(z_0)^2\ \ \cdots\  \big).
\end{align}
By Theorem 6.3 of Seneta \cite{Seneta06}, this $\bu_*^{(1)}(z_0)$ is positive and hence the results of the proposition holds. 
\end{proof}

%
%
\section{Derivations of the coefficient vectors} \label{sec:get_coefficients}

\noindent\textit{In the case of Type I}
\smallskip

(1) $\tilde{\bvarphi}_{1,-1}^{I}$ of formula (\ref{eq:tildevarphiI1m1}).\quad 
From Propositions \ref{pr:tildef1_limit} and \ref{pr:adjC1_zs}, we obtain 
\begin{align*}
\tilde{\bvarphi}_{1,-1}^{I} 
&= \lim_{\tilde{\Delta}_{\bar{z}_1^*}\ni z\to \bar{z}_1^*} (\bar{z}_1^*-z)^{\frac{1}{2}}\,\frac{\bg_1(z)\,\adj(I-C_1(z,G_1(z))}{\tilde{f}_1(1,(\bar{z}_1^*-z)^{\frac{1}{2}})} 
=\frac{\bg_1(\bar{z}_1^*)\,\bv^{C_1}(\bar{z}_1^*)\,\bu^{C_1}(\bar{z}_1^*)}{-\tilde{\lambda}^{C_1}_\zeta(0)}.
\end{align*}

\medskip
(2) $\tilde{\bvarphi}_{1,1}^{I}$ of formula (\ref{eq:tildevarphiI11}).\quad 
From Propositions \ref{pr:limit_varphi2G1}  and \ref{pr:limit_varphi2C2G1}, we obtain
\begin{align*}
\tilde{\bvarphi}_{1,1}^I 
&= \frac{d}{d\zeta}\, \tilde{\bg}_1(\zeta)\,(I-C_1(\bar{z}_1^*-\zeta^2,\tilde{G}_1(\zeta))^{-1}\,\Big|_{\zeta=0} \cr
&= \bigl(\tilde{\bvarphi}^{\hat{C}_2}_{2,1} -\tilde{\bvarphi}_{2,1}^{\tilde{G}_1} +\bnu_{0,0} A_{*,1}^{(0)}(\bar{z}_1^*)\, \tilde{G}_{1,1} \bigr) (I-C_1(\bar{z}_1^*,G_1(\bar{z}_1^*)))^{-1} \cr
&\qquad + \bg_1(\bar{z}_1^*) (I-C_1(\bar{z}_1^*,G_1(\bar{z}_1^*)))^{-1} A_{*,1}^{(1)}(\bar{z}_1^*)\, \tilde{G}_{1,1}\, (I-C_1(\bar{z}_1^*,G_1(\bar{z}_1^*)))^{-1}, 
\end{align*}
and this leads us to formula (\ref{eq:tildevarphiI11}). 

\bigskip
\noindent\textit{In the case of Type II}

\smallskip
Before deriving expressions for the coefficient vectors, we give the following proposition. 
\begin{proposition} \label{pr:vR2}
In the case of Type II, we have $\spr(G_1(r_1))=r_2$, $\spr(R_2(r_2))=r_1^{-1}$ and 
\begin{equation}
\bv_{s_0}^{G_1}(r_1) = (r_1 I-G_2(r_2))^{-1} N_2(r_2)\, \bv^{R_2}(r_2), 
\label{eq:vR2}
\end{equation}
where $\bv^{R_2}(r_2)$ is the right eigenvector of $R_2(r_2)$ with respect to the eigenvalue $r_1^{-1}$.
\end{proposition}
\begin{proof}
In the case of Type II, we have $\spr(G_1(r_1))=\underline{\zeta}_2(r_1)=r_2$, $\spr(R_1(r_1))=\bar{\zeta}_2(r_1)^{-1} \le r_2^{-1}$, $\spr(G_2(r_2))=\underline{\zeta}_1(r_2)<r_1$ and $\spr(R_2(r_2))=\bar{\zeta}_1(r_2)^{-1}=r_1^{-1}$. 
Furthermore, by Lemma \ref{le:Gmatrix} and Remark \ref{re:Gmatrix2}, we have 
\begin{align}
I-C(r_1,r_2)
&= \bigl(r_2^{-1}I-R_1(r_1)\bigr) \bigl(I-H_1(r_1)\bigr) \bigl(r_2 I-G_1(r_1)\bigr) \cr
&= \bigl(r_1^{-1}I-R_2(r_2)\bigr) \bigl(I-H_2(r_2)\bigr) \bigl(r_1 I-G_2(r_2)\bigr). 
\end{align}
Multiplying both the sides of this equation by $\bv_{s_0}^{G_1}(r_1)$ from the right, we obtain 
\begin{equation}
\bigl(r_1^{-1}I-R_2(r_2)\bigr) \bigl(I-H_2(r_2)\bigr) \bigl(r_1 I-G_2(r_2)\bigr) \bv_{s_0}^{G_1}(r_1) = \bzero.
\end{equation}
Since both $\bigl(I-H_2(r_2)\bigr)$ and $\bigl(r_1 I-G_2(r_2)\bigr)$ are nonsingular, we obtain 
\[
\bv^{R_2}(r_2) = \bigl(I-H_2(r_2)\bigr) \bigl(r_1 I-G_2(r_2)\bigr) \bv_{s_0}^{G_1}(r_1)\ne \bzero, 
\]
and this leads us to expression (\ref{eq:vR2}), where we use the fact that $N_2(r_2)=(I-H_2(r_2))^{-1}$. 
\end{proof}

\bigskip
(1) $\bvarphi_{1,-1}^{II}$ of formula (\ref{eq:varphiII1m1}).\quad 
Note that we have 
\begin{align}
\bu^{C_2}(r_2) (C_2(r_1,r_2)-I) 
&= \bu^{C_2}(r_2) (C_2(r_1,r_2)-C_2(G_2(r_2),r_2)) \cr
&= \bu^{C_2}(r_2) A_{1,*}^{(2)}(r_2)(r_1 I-G_2(r_2)). 
\label{eq:uC2_C2_G2}
\end{align}
Hence, from Propositions \ref{pr:limit_varphi2G1}  and \ref{pr:limit_varphi2C2G1}, we obtain
\begin{align*}
\bvarphi_{1,-1}^{II} 
&= \lim_{z\to r_1} (r_1-z)\, \bg_1(z)\,(I-C_1(z,G_1(z))^{-1} \cr
&= \bigl(\bvarphi^{\hat{C}_2}_{2,-1} -\bvarphi_{2,-1}^{\tilde{G}_1}\bigr) (I-C_1(r_1,G_1(r_1)))^{-1} \cr
&= \frac{c_{pole}^{\bvarphi_2} \bu^{C_2}(r_2) A_{1,*}^{(2)}(r_2)(r_1 I-G_2(r_2)) \bv_{s_0}^{G_1}(r_1) \bu_{s_0}^{G_1}(r_1) \bigl( I-C_1(r_1,G_1(r_1)) \bigr)^{-1}}{\underline{\zeta}_{2,z}(r_1)}.
\end{align*}
This and Proposition \ref{pr:vR2} lead us to formula (\ref{eq:varphiII1m1}). 

\medskip
(2) $\bvarphi_{1,-2}^{II}$ of formula (\ref{eq:varphiII1m2}).\quad 
From Propositions \ref{pr:f1_limit}, \ref{pr:adjC1}, \ref{pr:limit_varphi2G1} and \ref{pr:limit_varphi2C2G1}, we obtain
\begin{align*}
\bvarphi_{1,-2}^{II} 
&= \lim_{z\to r_1} (r_1-z)^2\, \frac{\bg_1(z)\,\adj\big(I-C_1(z,G_1(z))\big)}{f_1(1,z)} \cr
&= \frac{ (\bvarphi^{\hat{C}_2}_{2,-1} -\bvarphi_{2,-1}^{\tilde{G}_1} )\, \bu^{C_1}(r_1) \bv^{C_1}(r_1)}{\psi_{1,z}(r_1)} \cr
&= \frac{c_{pole}^{\bvarphi_2} \bu^{C_2}(r_2) A_{1,*}^{(2)}(r_2)(r_1 I-G_2(r_2)) \bv_{s_0}^{G_1}(r_1)\bu_{s_0}^{G_1}(r_1) \bv^{C_1}(r_1) \bu^{C_1}(r_1)}{\underline{\zeta}_{2,z}(r_1) \psi_{1,z}(r_1)}.
\end{align*}
This and Proposition \ref{pr:vR2} lead us to formula (\ref{eq:varphiII1m2}). 

\medskip
(3) $\tilde{\bvarphi}_{1,-2}^{II}$ of formula (\ref{eq:tildevarphiII1m2}).\quad
From Propositions \ref{pr:limit_varphi2G1}, \ref{pr:limit_varphi2C2G1}, \ref{pr:tildef1_limit} and \ref{pr:adjC1_zs}, we obtain
\begin{align*}
\tilde{\bvarphi}_{1,-2}^{II} 
&= \lim_{\tilde{\Delta}_{\bar{z}_1^*}\ni z\to \bar{z}_1^*} (\bar{z}_1^*-z)\, \frac{\tilde{\bg}_1((\bar{z}_1^*-z)^{\frac{1}{2}})\,\adj\big(I-C_1(z,\tilde{G}_1((\bar{z}_1^*-z)^{\frac{1}{2}}))\big)}{\tilde{f}_1(1,(\bar{z}_1^*-z)^{\frac{1}{2}})} \cr
&= \frac{(\tilde{\bvarphi}^{\hat{C}_2}_{2,-1} -\tilde{\bvarphi}_{2,-1}^{\tilde{G}_1} )\,\bu^{C_1}(\bar{z}_1^*) \bv^{C_1}(\bar{z}_1^*)}{-\tilde{\lambda}_\zeta^{C_1}(0)} \cr
&= \frac{c_{pole}^{\bvarphi_2} \bu^{C_2}(r_2) A_{1,*}^{(2)}(r_2)(\bar{z}_1^* I-G_2(r_2)) \bv_{s_0}^{G_1}(\bar{z}_1^*)\bu_{s_0}^{G_1}(\bar{z}_1^*)\bv^{C_1}(\bar{z}_1^*) \bu^{C_1}(\bar{z}_1^*)}{\tilde{\alpha}_{s_0,1}^{G_1} \tilde{\lambda}_\zeta^{C_1}(0)}.
\end{align*}
This and Proposition \ref{pr:vR2} lead us to formula (\ref{eq:tildevarphiII1m2}). 

\medskip
(4) $\tilde{\bvarphi}_{1,-1}^{II}$ of formula (\ref{eq:tildevarphiII1m1}).\quad
From Propositions \ref{pr:limit_varphi2G1} and \ref{pr:limit_varphi2C2G1}, we obtain
\begin{align*}
\tilde{\bvarphi}_{1,-1}^{II} 
&= \lim_{\tilde{\Delta}_{\bar{z}_1^*}\ni z\to \bar{z}_1^*} (\bar{z}_1^*-z)^{\frac{1}{2}}\, \tilde{\bg}_1((\bar{z}_1^*-z)^{\frac{1}{2}})\, (I-C_1(z,\tilde{G}_1((\bar{z}_1^*-z)^{\frac{1}{2}})))^{-1} \cr
&= (\tilde{\bvarphi}^{\hat{C}_2}_{2,-1} -\tilde{\bvarphi}_{2,-1}^{\tilde{G}_1} )\, (I-C_1(\bar{z}_1^*,G_1(\bar{z}_1^*)))^{-1} \cr
&= \frac{c_{pole}^{\bvarphi_2} \bu^{C_2}(r_2) A_{1,*}^{(2)}(r_2)(\bar{z}_1^* I-G_2(r_2)) \bv_{s_0}^{G_1}(\bar{z}_1^*) \bu_{s_0}^{G_1}(\bar{z}_1^*) (I-C_1(\bar{z}_1^*,G_1(\bar{z}_1^*)))^{-1}}{-\tilde{\alpha}_{s_0,1}^{G_1} }.
\end{align*}
This and Proposition \ref{pr:vR2} lead us to formula (\ref{eq:tildevarphiII1m1}).

\end{document}